\documentclass[11pt,letterpaper,reqno]{amsart}

\usepackage[T1]{fontenc}
\usepackage[utf8]{inputenc}
\usepackage{amsmath,amssymb,amsfonts,mathrsfs,mathtools}
\usepackage{amsthm}
\usepackage{enumitem}
\usepackage[margin=1.08in]{geometry}
\usepackage{aliascnt}
\usepackage[colorlinks=true,linkcolor=blue,citecolor=blue,urlcolor=blue]{hyperref}
\usepackage[nameinlink,capitalize]{cleveref}
\usepackage[protrusion=true,expansion=false]{microtype}
\usepackage[nocompress]{cite}
\usepackage{graphicx}
\usepackage{tikz}
\usepackage{placeins}
\usetikzlibrary{arrows.meta,calc,positioning}

\allowdisplaybreaks
\numberwithin{equation}{section}
\newtheorem{theorem}{Theorem}[section]
\newaliascnt{proposition}{theorem}
\newtheorem{proposition}[proposition]{Proposition}
\aliascntresetthe{proposition}
\newaliascnt{lemma}{theorem}
\newtheorem{lemma}[lemma]{Lemma}
\aliascntresetthe{lemma}
\newaliascnt{corollary}{theorem}
\newtheorem{corollary}[corollary]{Corollary}
\aliascntresetthe{corollary}
\newaliascnt{claim}{theorem}

\aliascntresetthe{claim}
\theoremstyle{definition}
\newaliascnt{definition}{theorem}
\newtheorem{definition}[definition]{Definition}
\aliascntresetthe{definition}
\theoremstyle{remark}
\newaliascnt{remark}{theorem}
\newtheorem{remark}[remark]{\textbf{Remark}}
\aliascntresetthe{remark}

\newtheorem*{question}{\textbf{Question}}
\crefname{claim}{Claim}{Claims}
\Crefname{claim}{Claim}{Claims}
\newcommand{\R}{\mathbb R}
\newcommand{\N}{\mathbb N}
\newcommand{\T}{\mathbb T}

\newcommand{\cB}{\mathcal B}
\newcommand{\cK}{\mathcal K}
\newcommand{\cR}{\mathcal R}
\newcommand{\cS}{\mathcal S}

\newcommand{\eps}{\varepsilon}
\newcommand{\dd}{\,\mathrm d}
\newcommand{\supp}{\operatorname{supp}}
\newcommand{\osc}{\mathrm{osc}}
\newcommand{\rad}{\mathrm{rad}}
\newcommand{\err}{\mathrm{err}}

\newcommand{\pv}{\operatorname{p.v.}}
\newcommand{\evarphi}{e_{\varphi}}
\newcommand{\er}{e_r}
\newcommand{\Norm}[1]{\left\lVert#1\right\rVert}

\title[Gap loss for SQG]
{Instantaneous gap loss of Sobolev regularity\\
	for the surface quasi-geostrophic equation}
	
\author{Qionglei Chen}
\address{Qionglei Chen
	\newline \indent
	Institute of Applied Physics and Computational Mathematics, 100088 Beijing, China}
\email{chen\_qionglei@iapcm.ac.cn}

\author{Changxing Miao}
\address{Changxing Miao 
	\newline \indent 
	School of Mathematics and Physics, University of Science and Technology Beijing, 100083 Beijing, China}
\email{miao\_changxing@ustb.edu.cn, miao\_changxing@iapcm.ac.cn}

\author{Yaowei Xie}
\address{Yaowei Xie
	\newline \indent
	Institute of Applied Physics and Computational Mathematics, 100088 Beijing, China}
\email{mathxyw@163.com}

\begin{document}

\begin{abstract}
	We prove instantaneous loss of Sobolev regularity across a fixed
	gap for the inviscid surface quasi-geostrophic equation on
	$\mathbb R^2$. 
	More precisely, for every $1<s<2$, $T>0$ and
	$\varepsilon>0$, we construct initial data
	$\theta_0\in H^s(\mathbb R^2)$ with
	$\|\theta_0\|_{H^s}\leq\varepsilon$ that generate a solution
	$\theta$ on $[0,T]$ such that
	\[
	\theta\in L^\infty([0,T];H^{1+\gamma}(\mathbb R^2)),
	\quad
	\theta(t)\notin H^\sigma(\mathbb R^2)
	\quad\text{for every }t\in(0,T],\ \sigma>\sigma_*(s),
	\]
	where $\gamma>0$ and $\sigma_*(s)$ satisfy
	\[
	1+\gamma<\sigma_*(s):=
	\frac{s(3-s)}{1+2s-s^2}<s.
	\]
	The solution is unique in a determined family of classical solutions with initial datum $\theta_0$.
\end{abstract}


\maketitle

\noindent {\bf MSC(2020):} 35Q35, 35R25 \quad
\vskip 0.02cm
\noindent {\bf Keywords:} Surface quasi-geostrophic equation, Instantaneous gap loss of regularity, Ill-posedness

\section{Introduction}

\subsection{Background}

We consider the inviscid surface quasi-geostrophic (SQG) equation on
$\R^2$, an active scalar model arising in atmospheric and oceanic
dynamics; see \cite{Pedlosky,HeldPierrehumbert,CMT,MajdaTabak}:
\begin{equation}\label{eq:SQG}
	\begin{cases}
		\partial_t\theta+u[\theta]\cdot\nabla\theta=0,
		& (x,t)\in\R^2\times(0,\infty),\\[1mm]
		u[\theta]=\nabla^\perp\Lambda^{-1}\theta
		=\cR^\perp\theta,
		\quad \nabla\cdot u[\theta]=0,\\[1mm]
		\theta(\cdot,0)=\theta_0.
	\end{cases}
\end{equation}
Here $\Lambda=(-\Delta)^{\frac12}$, $\cR=(\cR_1,\cR_2)$ denotes the
vector of Riesz transforms, and
\[
\nabla^\perp:=(-\partial_2,\partial_1),
\quad\cR^\perp:=(-\cR_2,\cR_1).
\]
Equivalently,
\begin{equation}\label{Biot}
	u[\theta](x)
	=c\,\pv\int_{\R^2}
	\frac{(x-y)^\perp}{|x-y|^3}\theta(y)\dd y,
\end{equation}
where $c$ is a dimensional constant.

The standard energy estimate for the transport equation depends on
the Lipschitz norm of the velocity. Since
$\theta\mapsto u=\cR^\perp\theta$ is a homogeneous Fourier multiplier
of order zero, the
$L^2$ scale in two dimensions is classified as follows:
\[
H^r(\R^2)\ \text{is}\ 
\begin{cases}
	\text{subcritical}, & r>2,\\
	\text{critical}, & r=2,\\
	\text{supercritical}, & r<2.
\end{cases}
\]
The corresponding critical space on the $L^\infty$ scale is
$W^{1,\infty}$. We review the known well-posedness and ill-posedness
results in these regimes.

\medskip
\textbf{Subcritical regime.}
Constantin, Majda and Tabak \cite{CMT} established local
well-posedness in $H^r(\R^2)$ for $r>2$. Wu \cite{WuHolder} proved
local well-posedness in $C^{1,\alpha}(\R^2)\cap L^q(\R^2)$ for
$0<\alpha<1$ and $1<q<\infty$. It remains open whether solutions of the inviscid SQG equation \eqref{eq:SQG} with smooth initial data can become singular in finite time.

\medskip
\textbf{Critical regime.} 
The first ill-posedness result for the inviscid SQG equation \eqref{eq:SQG} 
was obtained by Elgindi and Masmoudi \cite{ElgindiMasmoudi},
who proved mild ill-posedness in $W^{1,\infty}(\R^2)$ near
stationary solutions. C\'ordoba and
Mart\'inez-Zoroa \cite{CMOSQG} established norm inflation in $H^2(\R^2)$.
Recently, Jeong and Kim \cite{JeongKimSQG} proved both norm
inflation and nonexistence in $H^2(\T^2)$. 

\medskip
\textbf{Supercritical regime.}
C\'ordoba and Mart\'inez-Zoroa \cite{CMOSQG} proved norm inflation
and nonexistence in $H^s(\R^2)$ with $\frac32<s<2$.
The same authors together with
O\.{z}a\'nski \cite{CMOcontinuous} also established norm
inflation and continuous loss of Sobolev regularity in this
range. To the best of our knowledge, there were previously no results on norm inflation or loss of Sobolev regularity in 
$H^s(\R^2)$ with $1<s\leq\frac32$.

For the incompressible Euler equation, norm inflation was
established by Bourgain and Li \cite{BourgainLi} in the critical
space and by Luo \cite{LuoSupercritical} in
supercritical spaces.
Several recent works have studied loss of regularity, which is a stronger form of ill-posedness than norm inflation. In two dimensions, C\'ordoba,
Mart\'inez-Zoroa and O\.{z}a\'nski \cite{EulerGap} constructed
solutions with $u(0)\in H^s(\R^2)$, $1<s<2$, for which
\[
u(t)\notin H^{s'}(\R^2)
\quad\text{for every }t>0
\text{ and every }s'>1+\frac{(3-s)(s-1)}{1+2s-s^2}.
\]
Jeong \cite{JeongEulerLoss} also proved loss in the
integrability exponent. In three dimensions, Jeong,
Mart\'inez-Zoroa and O\.{z}a\'nski \cite{JMZO25} obtained
continuous loss of Sobolev regularity.

For the inviscid SQG equation, C\'ordoba,
Mart\'inez-Zoroa and O\.{z}a\'nski \cite{CMOcontinuous}
constructed compactly supported initial data with arbitrarily
small \(H^s\) norm, \(3/2<s<2\), for which the corresponding
local-in-time solution remains compactly supported and satisfies
\[
\theta(t)\in H^{\frac{s}{1+\bar c t}}
\quad\text{and}\quad
\theta(t)\notin H^\sigma
\quad\text{for every }
\sigma>\frac{s}{1+\bar c t}.
\]
The solution is unique in a specified
class of classical solutions.

This leads to the following question.

\begin{question}
Is there an initial datum in a supercritical Sobolev space
for which the corresponding solution of \eqref{eq:SQG}
is unique and exhibits instantaneous gap loss of regularity?
\end{question}

The main result of this paper gives a positive answer. More precisely,
for every $1<s<2$, we construct arbitrarily small initial data
in $H^s(\R^2)$ whose solution exhibits instantaneous loss of
regularity across a fixed Sobolev gap and is unique
in the class of \cref{def:classical}.

\subsection{Main results}
For $1<s<2$, define
\begin{equation*}
	\sigma_*(s):=\frac{s(3-s)}{1+2s-s^2}.
\end{equation*}
A direct computation gives $1<\sigma_*(s)<s$.

For norm inflation, we have the following theorem.

\begin{theorem}
	\label{thm:block-inflation}
	Let $1<s<2$. For any $T>0$, $\tau\in(0,T)$,
	$\eps>0$, $M>0$ and $\sigma>\sigma_*(s)$, there exists
	$\theta_0\in C_c^\infty(\R^2)$ with
	\[
	\Norm{\theta_0}_{H^s}\leq\eps
	\]
	such that the corresponding smooth solution $\theta$ of
	\eqref{eq:SQG} exists on $[0,T]$ and satisfies
	\begin{equation}\label{eq:block-inflation}
		\Norm{\theta(t)}_{H^\sigma}\geq M
		\quad\text{for all }t\in[\tau,T].
	\end{equation}
\end{theorem}

Before stating the loss of regularity, we define the class
of classical solutions considered in this paper.

\begin{definition}[Classical solution]\label{def:classical}
	Let $T>0$. A function
	$\theta:\R^2\times[0,T]\to\R$ is called a classical solution of
	\eqref{eq:SQG} with initial datum $\theta_0$ if $\theta$ is continuous
	on $\R^2\times[0,T]$, its spatial and time derivatives exist at every point, \eqref{eq:SQG} holds pointwise with
	$\theta(\cdot,0)=\theta_0$, and
	\[
	\theta\in L^\infty([0,T];H^{1+\gamma}(\R^2))
	\]
	for some $\gamma>0$.
\end{definition}

\begin{theorem}[Instantaneous gap loss of regularity]
	\label{thm:main-loss}
	Let $1<s<2$. For any $T>0$ and $\eps>0$, there exists an initial
	datum $\theta_0\in H^s(\R^2)$ with $\Norm{\theta_0}_{H^s}\leq\eps$
	such that \eqref{eq:SQG} admits a classical solution $\theta$ on
	$[0,T]$ with $\theta(0)=\theta_0$. This solution satisfies
	\begin{equation}\label{eq:main-loss}
		\theta(t)\notin H^\sigma(\R^2)
		\quad\text{for every }t\in(0,T]
		\text{ and every }\sigma>\sigma_*(s).
	\end{equation}
	Moreover, there exists $\gamma>0$ satisfying
	$1+\gamma<\sigma_*(s)$ such that
	\[
\theta\in L^\infty([0,T];H^{1+\gamma}(\R^2)).
	\]
	The solution is unique in the class of \cref{def:classical} with initial datum $\theta_0$.
\end{theorem}

The fixed Sobolev gap is given explicitly by
\begin{equation*}
	s-\sigma_*(s)
	=
	\frac{s(2-s)(s-1)}{1+2s-s^2}>0.
\end{equation*}
The existence of $\gamma>0$ with
$1+\gamma<\sigma_*(s)$ follows from
\begin{equation}\label{eq:sigma-above-one}
	\sigma_*(s)-1
	=
	\frac{s-1}{1+2s-s^2}>0.
\end{equation}

\begin{remark}
	In \cite{CMOcontinuous}, C\'ordoba, Mart\'inez-Zoroa and
	O\.{z}a\'nski established an instantaneous continuous loss of Sobolev regularity
	from initial data in $H^s(\R^2)$ for the inviscid SQG equation, with
	$\frac32<s<2$. This initial datum generates a solution on some interval
	$[0,T_0]$ whose Sobolev gap at time $t$ is
	\[
	s-\frac{s}{1+\bar c t}
	=\frac{s\bar c t}{1+\bar c t},
	\]
	where $\bar c>0$ is a small constant. By contrast,
	\cref{thm:main-loss} gives the fixed gap $s-\sigma_*(s)>0$. The two
	gaps coincide at
	\[
	t_{\mathrm{cross}}
	:=\frac{1}{\bar c}
	\left(\frac{s}{\sigma_*(s)}-1\right),
	\]
	and the fixed gap is larger for $0<t<t_{\mathrm{cross}}$. In addition,
	\cref{thm:main-loss} holds on $[0,T]$ for any $T>0$ and in the wider
	range $1<s<2$.
\end{remark}

\begin{remark}
	The study of instantaneous loss of regularity across a fixed
	Sobolev gap originated in \cite{EulerGap} for the two-dimensional
	Euler equation. The shear mechanism used in \cite{EulerGap} has also been
	applied to various problems in fluid dynamics
	\cite{CMOSQG,DissipativeSQG,GuoLuo,CMLMZgeneralized,
		CMZHolder,LuoSupercritical}.
	However, the strategy in \cite{EulerGap} does not extend directly
	to SQG \eqref{eq:SQG}. 
	Our construction exploits the shear effect induced by an outer radial shell in its interior region and uses coordinates adapted to the background shear to obtain polynomial bounds that allow the solution to be continued up to any fixed time $T>0$.
\end{remark}

\subsection{Ideas of the proof}\label{subsec:proof-ideas}

We first highlight two new ingredients that are key to our construction. 

\begin{enumerate}[label=\textup{(\roman*)}]
	\item
	For the two-dimensional Euler equation, a radial vorticity shell
	induces no velocity in its interior. For SQG, a radial profile
	can generate angular shear on either side of its support.
	We use an outer radial component to generate a nondegenerate
	shear on the inner oscillatory annulus; see
	\cref{fig:nonlocal-radial-shear} and
	Appendix~\ref{app:radial-kernel}.
	
	\item 
	The global existence of classical solutions to the SQG equation \eqref{eq:SQG} remains open, even for smooth initial data. 
	To obtain loss of regularity across a fixed Sobolev gap at each $t\in(0,T]$ for any $T>0$, we need to choose initial data whose corresponding solution exists on $[0,T]$.
	We use the following key observation: in coordinates defined by the exact background shear flow, the background transport term vanishes. A commutator estimate involving the velocity operator and a weighted radial derivative yields bounds for the conjugated operator that depend polynomially on the shear strength. 
	These bounds allow us to continue the solution beyond $T$; see Step~3 below.
\end{enumerate}

We sketch the proof in four steps. For any $0<\tau<T$,
Steps~1--3 yield a smooth block solution on $[0,T]$ with
uniform Sobolev growth on $[\tau,T]$. In Step~4 we place
infinitely many such blocks sufficiently far apart and glue
them to obtain loss of regularity at each $t\in(0,T]$.

\smallskip
\noindent{\bf Step 1. Construction of the approximate solution.}
In polar coordinates $(r,\varphi)$, choose smooth radial profiles
$f\in C_c^\infty((4,5))$ and $g\in C_c^\infty((1,2))$, with
$f\geq0$ and $f,g\not\equiv0$. Let \(c_0>0\) be small, and let \(\lambda\) and the integer \(N\) be large parameters to be fixed later. We take the initial datum to be
\begin{align}\label{idea-initial-datum}
	\theta_0(r,\varphi)
	:=
	\underbrace{
		c_0\lambda^{1-s}f(\lambda r)
	}_{\theta_{\rad,0}(r)}
	+
	\underbrace{
		c_0\lambda^{1-s}N^{-s}g(\lambda r)\cos(N\varphi)
	}_{\theta_{\osc,0}(r,\varphi)}.
\end{align}
On the fixed scale $y=\lambda x$, set $R=|y|$ and
$\mathcal I:=\{y\in\R^2:1/2<R<3\}$. 
\begin{figure}[htbp]
	\centering
	\resizebox{\textwidth}{!}{%
		\begin{tikzpicture}[
			x=1cm,
			y=1cm,
			>=Latex,
			line cap=round,
			line join=round,
			every node/.style={font=\small},
			paneltitle/.style={font=\bfseries\small, text=black!78},
			mathlabel/.style={font=\small, fill=white, fill opacity=.88,
				text opacity=1, inner sep=1.6pt, rounded corners=1.5pt}
			]
			
			\definecolor{ProbeBlue}{RGB}{43,102,166}
			\definecolor{ProbeRed}{RGB}{194,67,54}
			\definecolor{RadialOrange}{RGB}{224,137,48}
			\definecolor{ShearTeal}{RGB}{22,132,132}
			\definecolor{InkGray}{RGB}{73,79,87}
			
			\begin{scope}[shift={(2.45,2.20)}]
				\node[paneltitle, anchor=south] at (0,3.72)
				{(a) Physical-space supports};
				
				\def\amin{-58}
				\def\amax{58}
				\def\rscale{0.72}
				
				\path[fill=RadialOrange!30]
				(\amin:{4*\rscale}) arc[start angle=\amin,end angle=\amax,
				radius={4*\rscale}]
				-- (\amax:{5*\rscale})
				arc[start angle=\amax,end angle=\amin,radius={5*\rscale}] -- cycle;
				
				\foreach \j in {0,...,11}{
					\pgfmathsetmacro{\aone}{\amin+\j*(\amax-\amin)/12}
					\pgfmathsetmacro{\atwo}{\amin+(\j+1)*(\amax-\amin)/12}
					\ifodd\j
					\def\probeclr{ProbeBlue!62}
					\else
					\def\probeclr{ProbeRed!62}
					\fi
					\path[fill=\probeclr]
					(\aone:{1.08*\rscale})
					arc[start angle=\aone,end angle=\atwo,radius={1.08*\rscale}]
					-- (\atwo:{1.92*\rscale})
					arc[start angle=\atwo,end angle=\aone,radius={1.92*\rscale}] -- cycle;
				}
				
				\foreach \rho in {1,2,4,5}{
					\draw[InkGray!66, line width=.55pt]
					(\amin:{\rho*\rscale})
					arc[start angle=\amin,end angle=\amax,radius={\rho*\rscale}];
				}
				\foreach \aa in {\amin,\amax}{
					\draw[InkGray!58, line width=.55pt]
					(\aa:{1*\rscale}) -- (\aa:{2*\rscale});
					\draw[InkGray!58, line width=.55pt]
					(\aa:{4*\rscale}) -- (\aa:{5*\rscale});
				}
				
				\fill[InkGray] (0,0) circle (1.1pt);
				\draw[InkGray!45, densely dotted]
				(0,0) -- (0:{5.24*\rscale});
				\node[font=\scriptsize, text=InkGray, below left=1pt] at (0,0) {$0$};
				
				\foreach \rho/\lab in
				{1/{\lambda^{-1}},2/{2\lambda^{-1}},4/{4\lambda^{-1}},5/{5\lambda^{-1}}}{
					\draw[InkGray!65]
					({\rho*\rscale},-.045) -- ({\rho*\rscale},.045);
					\node[font=\scriptsize, text=InkGray, below=2pt]
					at ({\rho*\rscale},0) {$\lab$};
				}
				
				\draw[-{Latex[length=2.2mm]}, ShearTeal, line width=1.05pt]
				(-37:{1.28*\rscale})
				arc[start angle=-37,end angle=-8,radius={1.28*\rscale}];
				\draw[-{Latex[length=2.2mm]}, ShearTeal, line width=1.4pt]
				(-40:{1.73*\rscale})
				arc[start angle=-40,end angle=16,radius={1.73*\rscale}];
				

				\node[font=\scriptsize, text=InkGray!80, rotate=-22]
				at (-31:{2.98*\rscale})
				{empty gap};
				
				\draw[-{Latex[length=2mm]}, RadialOrange!78!black, densely dashed,
				line width=.8pt]
				(23:{4.36*\rscale})
				to[bend left=18] (30:{1.72*\rscale});
				\node[mathlabel, text=RadialOrange!65!black, rotate=-18]
				at (31:{3.15*\rscale}) {nonlocal action};
				
			\end{scope}
			
			\draw[-{Latex[length=2.5mm]}, InkGray!65, line width=.85pt]
			(6.30,2.30) -- (7.42,2.30)
			node[midway, above=2pt, font=\scriptsize, text=InkGray] {unfold};
			
			\begin{scope}[shift={(7.78,.10)}]
				\node[paneltitle, anchor=south] at (2.92,5.88)
				{(b) Unfolded phase level curves};
				
				\path[fill=black!1] (0,0) rectangle (5.82,5.18);
				
				\begin{scope}
					\clip (0,0) rectangle (5.82,5.18);
					
					\foreach \yy in {.38,1.06,1.74,2.42,3.10,3.78,4.46,5.14}{
						\draw[InkGray!46, densely dashed, line width=.55pt]
						(0,\yy) -- (5.82,\yy);
					}
					
					\foreach \j in {-2,-1,0,1,2,3,4,5,6}{
						\pgfmathsetmacro{\base}{0.18+0.68*\j}
						\ifodd\j
						\def\lineclr{ProbeBlue}
						\else
						\def\lineclr{ProbeRed}
						\fi
						\draw[\lineclr!88, line width=1.0pt]
						plot[domain=0:5.82, samples=70, smooth]
						(\x,{\base+0.20*\x+0.035*\x*\x});
					}
				\end{scope}
				
				\draw[InkGray!82, line width=.7pt] (0,0) rectangle (5.82,5.18);
				
				\draw[-{Latex[length=2mm]}, InkGray, line width=.75pt]
				(0,-.02) -- (6.12,-.02);
				\draw[-{Latex[length=2mm]}, InkGray, line width=.75pt]
				(-.02,0) -- (-.02,5.47)
				node[left=3pt, anchor=south east] {$\widetilde\varphi$};
				\node[font=\scriptsize, below=3pt] at (0,0) {$1$};
				\node[font=\scriptsize, below=3pt] at (5.82,0) {$2$};
				\node[font=\small, below=11pt] at (2.91,0) {$R=\lambda r$};
				
				\draw[InkGray!50, densely dashed, line width=.55pt]
				(.32,5.46) -- (1.10,5.46);
				\node[font=\scriptsize, anchor=west, text=InkGray]
				at (1.18,5.46) {$t=0$};
				\draw[ProbeBlue, line width=1.0pt]
				(3.88,5.46) -- (4.66,5.46);
				\node[font=\scriptsize, anchor=west, text=InkGray]
				at (4.74,5.46) {$t>0$};
				
				\draw[-{Latex[length=2.3mm]}, ShearTeal, line width=1.05pt]
				(3.60,1.06) -- (4.78,1.06);
				\node[mathlabel, text=ShearTeal!65!black, anchor=south]
				at (4.18,1.18) {radial direction};
				
			\end{scope}
			
		\end{tikzpicture}
	}
	\caption{Supports and phase level curves. The outer radial profile acts
		nonlocally across the empty gap and generates a nondegenerate angular
		shear on the inner oscillatory annulus. After unfolding this annulus,
		the level curves of the phase $N\widetilde\varphi-NKt\Omega(R)$ are
		given by
		$\widetilde\varphi=Kt\Omega(R)+c/N$, where $c\in\R$. They are
		horizontal at $t=0$, while for $t>0$ they are strictly increasing
		and strictly convex for $1<R<2$, in particular on $\supp g$, because
		$\Omega'(R)>0$ and $\Omega''(R)>0$; see
		\cref{lem:appendix-SQG-shear}.}
	\label{fig:nonlocal-radial-shear}
\end{figure}

Since $f$ is radial, the velocity $u[f]$ is purely angular, and we write
$$
u[f](r,\varphi)
=
u_\varphi[f](r)\evarphi
=
r\Omega(r)\evarphi,
\quad
\Omega(r):=\frac{u_\varphi[f](r)}{r},
$$
where $\evarphi=(-\sin\varphi,\cos\varphi)$. 
We take the approximate solution $\overline\theta$ to solve
\begin{align}\label{idea-approximate-eqn}
	\begin{cases}
		\partial_t\overline\theta
		+u[\theta_{\rad,0}]\cdot\nabla\overline\theta=0,\\[2pt]
		\overline\theta(\cdot,0)=\theta_0.
	\end{cases}
\end{align}
Set $K:=c_0\lambda^{2-s}$. The solution of
\eqref{idea-approximate-eqn} is given explicitly by
\begin{equation*}
	\begin{aligned}
		\overline\theta(r,\varphi,t)
		=
		\underbrace{
			c_0\lambda^{1-s}f(\lambda r)
		}_{\overline\theta_{\rad}(r)}
		~	+~
		\underbrace{
			c_0\lambda^{1-s}N^{-s}g(\lambda r)
			\cos\bigl(
			N\varphi-NKt\,\Omega(\lambda r)
			\bigr)
		}_{\overline\theta_{\osc}(r,\varphi,t)}.
	\end{aligned}
\end{equation*}

\smallskip
\noindent{\bf Step 2. Growth of the approximate solution.} The stationary component satisfies
$\Norm{\overline\theta_{\rad}(t)}_{H^\rho}\simeq
c_0\lambda^{\rho-s}$ for $0\leq\rho\leq2$. 
The outer radial profile generates a nondegenerate angular shear on
$\mathcal I$: $\inf_{1/2\leq R\leq3}\Omega'(R)\geq c_f>0$, in
particular on $\supp g$. For $0<\tau<T$ and
sufficiently large $\lambda$, the WKB estimate and the orthogonality of the
radial and oscillatory components give
\begin{equation}\label{eq:idea-approximate-growth}
	\Norm{\overline\theta(t)}_{H^\rho}
	\simeq_{f,g,\rho}
	c_0\lambda^{\rho-s}
	+c_0^{1+\rho}t^\rho
	\lambda^{\rho(3-s)-s}N^{\rho-s},
	\qquad t\in[\tau,T].
\end{equation}
Fix $\sigma_*(s)<\rho<s$. By \eqref{eq:idea-approximate-growth}, the oscillatory term can be made arbitrarily large uniformly on $[\tau,T]$ provided that
\[
N^{s-\rho}=o\!\left(\lambda^{\rho(3-s)-s}\right)
\text{~as~} \lambda\to \infty.
\]
For $N\sim\lambda^a$, the above asymptotic requirement implies to $a<A_s(\rho)$, where
\[
A_s(\rho):=\frac{\rho(3-s)-s}{s-\rho},
\qquad
A_s\bigl(\sigma_*(s)\bigr)=A_s \left(\frac{s(3-s)}{1+2s-s^2}\right)=\frac{2(2-s)}{s-1}.
\]
Since $A_s$ is strictly increasing, we may choose $\delta>0$ sufficiently small that $a_\delta<A_s(\rho)$, and set
\begin{equation}\label{eq:idea-parameter-choice}
	N=\left\lfloor\lambda^{a_\delta}\right\rfloor,
	\qquad
	a_\delta:=\frac{2(2-s)+\delta}{s-1}.
\end{equation}
Then, for any $M>0$, the approximate solution satisfies $\Norm{\overline\theta(t)}_{H^\rho}\ge M$ throughout $[\tau,T]$ when $\lambda$ is sufficiently large.

 \smallskip
 \noindent{\bf Step 3. Adapted coordinates and continuation.}
 
 \medskip
 \noindent\emph{The required existence time.}
The approximate solution is smooth for all $t\geq0$. Let $\theta$ be the exact solution with initial datum in \eqref{idea-initial-datum}, defined on its maximal  existence interval $[0,T_{\max})$. To transfer the bound $\Norm{\overline\theta(t)}_{H^\rho}\ge M$ on $[\tau,T]$ to $\theta$, we first prove $T_{\max}>T$. Choose
$0<\alpha<\frac{\delta(s-1)}{1+2s-s^2+\delta}$.
By the $C^{1,\alpha}$ continuation criterion, it suffices to show
\begin{equation}\label{eq:idea-continuation-criterion}
	\sup_{0\leq t<\min\{T,T_{\max}\}}
	\Norm{\theta(t)}_{C^{1,\alpha}}<\infty.
\end{equation}

We decompose $\theta$ into $\theta_{\rad}$ and
$\theta_{\osc}$, which are transported by $u[\theta]$
from the respective initial data $\theta_{\rad,0}$ and
$\theta_{\osc,0}$ given in \eqref{idea-initial-datum}.
In the rescaled variable $y=\lambda x$, set
$\Theta_\ell(y,t):=\theta_\ell(y/\lambda,t)$ for $\ell\in\{\rad,\osc\}$ and $R:=|y|$. Recall that $K=c_0\lambda^{2-s}$, and write
$\Theta_{\rad,\err}:=\Theta_{\rad}-c_0\lambda^{1-s}f$.
Unless stated otherwise, all derivatives and norms below are in $y$. The radial error is easier to control, so we focus on the oscillatory component:
\begin{equation}\label{eq:idea-exact-osc}
	\partial_t\Theta_{\osc}
	+\bigl(Ku[f]+\lambda u[\Theta_{\osc}+\Theta_{\rad,\err}]\bigr)
	\cdot\nabla\Theta_{\osc}=0.
\end{equation}
A direct $C^{1,\alpha}$ estimate of \eqref{eq:idea-exact-osc} gives
\begin{equation}\label{eq:idea-direct-holder}
		\frac{\dd}{\dd t}\Norm{\Theta_{\osc}(t)}_{C^{1,\alpha}}
		\leq C_\alpha
		\bigl(K+\lambda\Norm{\Theta_{\osc}(t)}_{C^{1,\alpha}}
		+\lambda\Norm{\Theta_{\rad,\err}(t)}_{C^{1,\alpha}}\bigr)
		\Norm{\Theta_{\osc}(t)}_{C^{1,\alpha}}.
\end{equation}
The initial oscillatory component satisfies
\begin{equation}\label{eq-initial-osc}
\Norm{\Theta_{\osc}(0)}_{C^{1,\alpha}}
\lesssim \frac K\lambda N^{1+\alpha-s}.
\end{equation}
Equation \eqref{eq:idea-direct-holder} is a Riccati inequality up to the radial error term.
The upper bound obtained from this inequality does not
develop a singularity for $0\leq t<T_{\mathrm{dir}}$, where
 \begin{equation*}
 	T_{\mathrm{dir}}\simeq\frac{\log N}{K}\simeq \lambda^{s-2}\log \lambda=o(1)
		\qquad\text{as }\lambda\to\infty.
 \end{equation*}
 Thus $T_{\mathrm{dir}}<\tau$ for sufficiently large $\lambda$, so the direct estimate does not cover $[\tau,T]$.

\medskip
\noindent\emph{The exact background shear.} 
The equation \eqref{eq:idea-exact-osc} contains the background transport
\(Ku[f]\cdot\nabla\Theta_{\osc}\).
Since \(f\) is radial, \(Ku[f]\) is stationary and purely angular: it preserves $R$ and rotates each circle at angular speed $K\Omega(R)$. This observation allows us to absorb the large linear transport into the coordinates. 
Let \(\cS_t\) denote the flow generated by the \emph{background shear} \(Ku[f]\), and let \(\cS_t^{-1}\) denote its inverse:
\[
\cS_t(R,\varphi)=\bigl(R,\varphi+Kt\Omega(R)\bigr),
\qquad
\cS_t^{-1}(R,\varphi)=\bigl(R,\varphi-Kt\Omega(R)\bigr).
\]
Both maps preserve Lebesgue measure. We use the coordinates defined by $\cS_t$ and set
$\widetilde\Theta_{\osc}:=\Theta_{\osc}\circ\cS_t$ and
$\widetilde\Theta_{\rad,\err}:=\Theta_{\rad,\err}\circ\cS_t$.
For $H$ supported in a fixed annulus, define
\begin{equation*}
	\cB_t[H](y):=
	(D\cS_t(y))^{-1}
	\lambda u[H\circ\cS_t^{-1}](\cS_t(y)),
\end{equation*}
where $D\cS_t$ denotes the Jacobian matrix in $y$. 
Let
$V:=\cB_t[\widetilde\Theta_{\osc}]
+\cB_t[\widetilde\Theta_{\rad,\err}]$.
Then the oscillatory component after transformation satisfies
\begin{equation*}
	\partial_t\widetilde\Theta_{\osc}
	+V\cdot\nabla\widetilde\Theta_{\osc}=0.
\end{equation*}

On the annulus $\mathcal I$, we have
$\inf_{\overline{\mathcal I}}\Omega'>0$.
For $\nu\in\{r,\varphi\}$, let $u_\nu[F]$ denote the polar component of $u[F]$ in \eqref{Biot}.
Rotation covariance gives
$\partial_\varphi u_\nu[F]=u_\nu[\partial_\varphi F]$. Let $H\in C_c^{1,\alpha}(\R^2)$
satisfy $\supp H\subset\mathcal I$.
Letting the weighted radial derivative
$\Omega'^{-1}\partial_R$ act on $u_\nu[H\circ\cS_t^{-1}]\circ\cS_t$, we obtain
\begin{equation}\label{eq:idea-radial-commutator}
	\frac{1}{\Omega'}\partial_R
	\bigl(u_\nu[H\circ\cS_t^{-1}]\circ\cS_t\bigr)
	=
	\Bigl(
	u_\nu\bigl[
	(\Omega'^{-1}\partial_RH)\circ\cS_t^{-1}
	\bigr]
	+
	[\Omega'^{-1}\partial_R,u_\nu]
	(H\circ\cS_t^{-1})
	\Bigr)\circ\cS_t.
\end{equation}
The polar velocity operator $u_\nu$ and the commutator
$[\Omega'^{-1}\partial_R,u_\nu]$ are bounded on
$C^\alpha(\mathcal I)$.
Together with \eqref{eq:idea-radial-commutator},
we obtain
\begin{equation}\label{eq:step3-polynomial-operator}
	\Norm{\cB_t[H]}_{C^{1,\alpha}(\mathcal I)}
	\le C_{f,\alpha}\lambda(1+Kt)^{1+2\alpha}
	\Norm H_{C^{1,\alpha}(\mathcal I)}.
\end{equation}

The initial $C^{1,\alpha}$ bound in
\eqref{eq-initial-osc} contains $N^{1+\alpha-s}$, which offsets the polynomial dependence on
$Kt$ in \eqref{eq:step3-polynomial-operator}.
With the parameter in \eqref{eq:idea-parameter-choice},
a bootstrap argument gives
\begin{equation}\label{eq:idea-effective-smallness}
	\begin{aligned}
		\int_0^t
		\Norm{\cB_\zeta[\widetilde\Theta_{\osc}(\zeta)]}
		_{C^{1,\alpha}(\mathcal I)}\,\dd\zeta
		&\lesssim TK(1+KT)^{1+2\alpha}N^{1+\alpha-s}\\
		&\lesssim
		\lambda^{-\frac{\delta(s-1)^2}{1+2s-s^2+\delta}}
		\longrightarrow0
		\qquad\text{as }\lambda\to\infty,
	\end{aligned}
\end{equation}
 uniformly for
$t<\min\{T,T_{\max}\}$.
The radial error contributes only an exponentially small term to
$V$ on $\mathcal I$. Together with
\eqref{eq:idea-effective-smallness}, the transport estimate closes
the bootstrap and yields
\begin{equation*}
	\sup_{0\leq t<\min\{T,T_{\max}\}}
	\Norm{\widetilde\Theta_{\osc}(t)}_{C^{1,\alpha}}
	\lesssim \frac K\lambda N^{1+\alpha-s}.
\end{equation*}

Returning to the original variable $x$ leads to
\begin{equation*}
	\begin{aligned}
		\sup_{0\leq t<\min\{T,T_{\max}\}}
		\Norm{\theta_{\osc}(t)}_{C_x^{1,\alpha}}
		&\leq C_{f,g,\alpha}K\lambda^\alpha(1+KT)^{1+\alpha}
		N^{1+\alpha-s}\\
		&\lesssim_{f,g,\alpha,T,c_0}
		\lambda^{-\delta+
			\frac{1+2s-s^2+\delta}{s-1}\alpha}
		\longrightarrow 0
		\qquad\text{as }\lambda\to\infty.
	\end{aligned}
\end{equation*}
Together with the radial bound, this estimate verifies
\eqref{eq:idea-continuation-criterion} and so
$T_{\max}>T$ for sufficiently large $\lambda$.

\smallskip
\noindent{\bf Step 4. Outward gluing.}
The $L^2$ difference satisfies
$\Norm{\theta(t)-\overline\theta(t)}_2
=o(\Norm{\theta_{\osc,0}}_2)$ uniformly on $[0,T]$.
For every fixed $\sigma>\sigma_*(s)$, we combine this estimate with
\eqref{eq:idea-approximate-growth} and Sobolev interpolation
to obtain $\Norm{\theta(t)}_{H^\sigma}\geq M$ for all
$t\in[\tau,T]$ when $\lambda$ is sufficiently large.

To obtain loss of regularity at every positive time, choose
\[
\sigma_j
:=\sigma_*(s)+\frac{s-\sigma_*(s)}{j+1},
~\tau_j:=\frac{T}{j+1},
~j\geq1.
\]
For each $j$, let $\vartheta_j$ denote the exact solution
$\theta$ constructed in Step~3 with
$\sigma=\sigma_j$, $\tau=\tau_j$ and $M=2^{2j}$.
Thus
$\Norm{\vartheta_j(t)}_{H^{\sigma_j}}\geq2^{2j}$
throughout $[\tau_j,T]$.
Choose the centers $\{z_j\}_{j\geq1}$ sufficiently far apart
so that, for each $j$, the velocity
$\sum_{k\ne j}u[\vartheta_k(\cdot-z_k,t)]$ is small on the
support of $\vartheta_j(\cdot-z_j,t)$.
For each $J\geq1$, let $\theta^{(J)}$ be the solution with
initial datum
$\theta_0^{(J)}(x):=\sum_{j=1}^J\vartheta_j(x-z_j,0)$.
The block parameters and centers are chosen so that
$\theta^{(J)}$ converges to a classical solution $\theta$
with initial datum
$\theta_0=\sum_{j\geq1}\vartheta_j(\cdot-z_j,0)$.

Now fix $t>0$ and $\sigma>\sigma_*(s)$. For all sufficiently large
$j$, we have $\tau_j\leq t$ and $\sigma_j<\sigma$. The lower bound for
the $j$th block, 
together with the separation between distinct blocks, yields $
2^j\leq C_\sigma\Norm{\theta(t)}_{H^\sigma}$, and letting $j\to\infty$ yields
$\theta(t)\notin H^\sigma$.

\FloatBarrier
\subsection{Notation}
For $x\in\R^2$, we write
\[
x=r(\cos\varphi,\sin\varphi),
\quad
r=|x|\geq0,
\quad
\varphi\in\T.
\]
For $r>0$, the associated orthonormal polar frame is
\[
\er=(\cos\varphi,\sin\varphi),
\quad
\evarphi=(-\sin\varphi,\cos\varphi).
\]
We identify $\T$ with $\R/(2\pi\mathbb Z)$. Unless otherwise specified, all norms are taken over $\R^2$.
We omit the domain from the notation and abbreviate $L^p$ norms
as follows:
\[
\Norm{h}_X:=\Norm{h}_{X(\R^2)},
\quad
\Norm{h}_p:=\Norm{h}_{L^p(\R^2)},
\quad 1\leq p\leq\infty,
\]
where $X$ denotes a function space used below, such as
$L^p$, $C^{k,\alpha}$, $H^\rho$ or $\dot H^\rho$.

For nonnegative quantities $X$ and $Y$, $X\lesssim Y$ means
$X\leq CY$ for some constant $C>0$, and $X\simeq Y$ means both
$X\lesssim Y$ and $Y\lesssim X$. Subscripts indicate dependence of
the implicit constant; for example, the constant in
$X\lesssim_{f,g,s}Y$ may depend on $f$, $g$ and $s$. Unless stated
otherwise, these constants are independent of the parameters
$\lambda$ and $N$ introduced below.

\subsection{Organization}

\Cref{sec:prelim} introduces the relevant function spaces and collects the preliminary lemmas.
In \cref{sec:pseudo}, we construct the approximate solution and
establish its Sobolev growth.
In \cref{sec:norm-inflation}, we prove \cref{thm:block-inflation}
by establishing continuation in adapted coordinates.
\Cref{sec:gluing} proves \cref{thm:main-loss} by a gluing argument and, in particular, establishes the uniqueness.
The radial shell formulas for Euler and SQG are proved in
Appendix~\ref{app:radial-kernel}.

\section{Preliminaries}\label{sec:prelim}

We collect several basic definitions, scaling identities, and velocity estimates used to separate the radial background from the angular oscillations.

\subsection{Basic properties}

For $0<\alpha<1$, the H\"older seminorm on $E\subset\R^2$ is
defined by
\begin{equation*}
	[h]_{C^\alpha(E)}
	:=
	\sup_{\substack{x,y\in E\\x\neq y}}
	\frac{|h(x)-h(y)|}{|x-y|^\alpha}.
\end{equation*}
For a multi-index $\nu=(\nu_1,\nu_2)$, write
$|\nu|=\nu_1+\nu_2$ and
$\partial^\nu=\partial_1^{\nu_1}\partial_2^{\nu_2}$,
where $\partial_i=\partial/\partial x_i$.
For every integer $j\geq0$, set
\begin{equation*}
	\begin{aligned}
		\Norm{\nabla^j h}_{L^\infty(E)}
		:=
		\sum_{|\nu|=j}\Norm{\partial^\nu h}_{L^\infty(E)},\quad
		[\nabla^j h]_{C^\alpha(E)}:=
		\sum_{|\nu|=j}[\partial^\nu h]_{C^\alpha(E)}.
	\end{aligned}
\end{equation*}
For vector fields, these sums also run over the components.
In particular, we take
\[
\Norm{h}_{C^{1,\alpha}(E)}
:=
\Norm{h}_{L^\infty(E)}
+
\sum_{i=1}^2
\left(
\Norm{\partial_i h}_{L^\infty(E)}
+
[\partial_i h]_{C^\alpha(E)}
\right).
\]

For $\rho\in\R$, write
\[
\Norm{h}_{H^\rho}
:=
\Norm{(1-\Delta)^{\rho/2}h}_2,
\quad
\Norm{h}_{\dot H^\rho}
:=
\Norm{\Lambda^\rho h}_2.
\]
For $0<\beta<1$, the Gagliardo characterization gives
\begin{equation*}
	\Norm{h}_{\dot H^\beta}^2
	\simeq_\beta
	\int_{\R^2}\int_{\R^2}
	\frac{|h(x+z)-h(x)|^2}{|z|^{2+2\beta}}
	\,\dd z\,\dd x.
\end{equation*}
More generally, for $\rho=m+\beta$, where $m\geq0$ is an integer
and $0<\beta<1$, the Gagliardo characterization gives
\begin{equation*}
	\Norm{h}_{H^\rho}^2
	\simeq_\rho
	\sum_{|\nu|\leq m}\Norm{\partial^\nu h}_2^2
	+
	\sum_{|\nu|=m}
\Norm{\partial^\nu h}_{\dot H^\beta}^2.
\end{equation*}

Under the rescaling $h_\lambda(x):=h(\lambda x)$, homogeneous Sobolev norms in two
dimensions satisfy
\begin{equation}\label{eq:Sobolev-scaling}
	\Norm{h_\lambda}_{\dot H^\rho}
	=
	\lambda^{\rho-1}\Norm{h}_{\dot H^\rho}, \text{~for~} \rho\in\R.
\end{equation}
Since the map
$
h\longmapsto u[h]=\nabla^\perp\Lambda^{-1}h
$
is homogeneous of order zero, we have
\begin{equation}\label{eq:Riesz-scaling}
	u[h_\lambda](x)=u[h](\lambda x).
\end{equation}

\begin{lemma}\label{lem:real-complex-angular}
	Let $N\in\N$, $N\geq1$, $\rho\in\R$,
	$a\in C_c^\infty((0,\infty);\R)$ and
	$\psi\in C^\infty((0,\infty);\R)$. Then
	\begin{equation}\label{eq:real-complex-angular}
		\Norm{a(r)\cos\!\left(N\varphi+\psi(r)\right)}_{H^\rho(\R^2)}
		=
		2^{-\frac12}
		\Norm{a(r)e^{i(N\varphi+\psi(r))}}_{H^\rho(\R^2)}.
	\end{equation}
\end{lemma}

\begin{proof}
	Let
	\[
	F(r,\varphi):=a(r)e^{i(N\varphi+\psi(r))}.
	\]
	Then
	$a(r)\cos(N\varphi+\psi(r))=(F+\overline F)/2$.
	Since $(1-\Delta)^{\rho/2}$ commutes with rotations, it preserves
	angular Fourier modes. The functions $F$ and $\overline F$ belong to
	the distinct modes $N$ and $-N$, so they are orthogonal in
	$H^\rho$. Complex conjugation also preserves the $H^\rho$ norm.
	Consequently,
	\[
	\Norm{\frac{F+\overline F}{2}}_{H^\rho}^2
	=
	\frac14\left(
	\Norm{F}_{H^\rho}^2+\Norm{\overline F}_{H^\rho}^2
	\right)
	=
	\frac12\Norm{F}_{H^\rho}^2,
	\]
which gives \eqref{eq:real-complex-angular}.
\end{proof}

For $r>0$, define the radial and angular velocity components by
\begin{equation*}
	u_r[h]:=u[h]\cdot\er,
	\quad
	u_\varphi[h]:=u[h]\cdot\evarphi.
\end{equation*}

\begin{lemma}\label{lem:radial}
	Let $h\in C_c^\infty(\R^2)$ be radial. Then
	$\Lambda^{-1}h$ is radial and, for $r>0$,
	\[
	u[h](r,\varphi)
	=
	\partial_r\!\left(\Lambda^{-1}h\right)(r)\evarphi.
	\]
	In particular,
	\[
	u_r[h]\equiv0,
	\quad
	u_\varphi[h](r)
	=
	\partial_r\!\left(\Lambda^{-1}h\right)(r).
	\]
\end{lemma}

\begin{proof}
	Set $\psi_h:=\Lambda^{-1}h$. Since the Fourier symbol $|\xi|^{-1}$
	is radial, $\Lambda^{-1}$ preserves radial symmetry. Hence $\psi_h$
	is radial. For $r>0$,
	\[
	u[h]
	=
	\nabla^\perp\psi_h
	=
	\partial_r\psi_h(r)\nabla^\perp r
	=
	\partial_r\psi_h(r)\evarphi.
	\]
	It follows that $u_r[h]=u[h]\cdot\er=0$.
\end{proof}

\begin{lemma}[Preservation of angular periodicity]
	\label{lem:Nfold}
	Let $T>0$ and $N\in\N$, $N\geq2$. Let $\theta$ be the unique smooth
	solution of \eqref{eq:SQG} on $[0,T]$ with initial datum $\theta_0$.
	If
	\[
	\theta_0\left(r,\varphi+\frac{2\pi}{N}\right)
	=
	\theta_0(r,\varphi),
	\]
	then
	\[
	\theta\left(r,\varphi+\frac{2\pi}{N},t\right)
	=
	\theta(r,\varphi,t),
	\quad 0\leq t\leq T.
	\]
	Moreover, let $h$ be the unique smooth solution of
	\[
	\partial_t h+u[\theta]\cdot\nabla h=0,
	\quad
	h(\cdot,0)=h_0.
	\]
	If $h_0$ is periodic in $\varphi$ with period $2\pi/N$, then so is $h(t)$ for
	every $t\in[0,T]$.
\end{lemma}

\begin{proof}
	Define
	\[
	\theta^\sharp(r,\varphi,t)
	:=
	\theta\!\left(r,\varphi+\frac{2\pi}{N},t\right).
	\]
	By rotational invariance, $\theta^\sharp$ solves \eqref{eq:SQG} with
	initial datum
	$\theta_0(r,\varphi+2\pi/N)=\theta_0(r,\varphi)$. Uniqueness therefore
	gives $\theta^\sharp=\theta$ on $[0,T]$.

Since $\theta^\sharp=\theta$, the polar components of
$u[\theta]$ are $2\pi/N$-periodic in angle. Rotational invariance then shows that $h(r,\varphi+2\pi/N,t)$ solves the transport equation with initial datum $h_0$. Uniqueness gives the stated periodicity of $h$.
\end{proof}

\subsection{Basic velocity estimates}

We denote the kernel in \eqref{Biot} by
\[
\cK(z):=c\frac{z^\perp}{|z|^3}.
\]
The first lemma gives a logarithmic $L^\infty$ bound for the Riesz transform $u[q]$; the second quantifies the decay of its nonlocal effect away
from the support.

\begin{lemma}
	\label{lem:log-CZ}
	Let $0<a<b$ and $R>0$. Let
	$q=q(x)\in C_c^1(\R^2)$ be nonzero and satisfy
	\[
	\supp q
	\subset
	\bigl\{x\in\R^2:aR<|x|<bR\bigr\}.
	\]
	Then
	\begin{equation}\label{eq:log-CZ}
		\Norm{u[q]}_\infty
		\leq
		C\Norm{q}_\infty
		\left[
		1+\log\left(
		2+\frac{R\Norm{\nabla q}_\infty}{\Norm{q}_\infty}
		\right)
		\right],
	\end{equation}
	where $C$ depends only on $a$ and $b$.
\end{lemma}

\begin{proof}
Define $\widetilde q(y):=q(Ry)$. Then $\widetilde q$ is supported in
$\{a<|y|<b\}$ and, by \eqref{eq:Riesz-scaling},
\[
\Norm{\widetilde q}_\infty=\Norm{q}_\infty,
\quad
\Norm{\nabla\widetilde q}_\infty
=R\Norm{\nabla q}_\infty,
\quad
\Norm{u[\widetilde q]}_\infty=\Norm{u[q]}_\infty.
\]
These scaling relations reduce the proof to the case $R=1$; below we
relabel $\widetilde q$ as $q$. Set
	\[
	d:=
	\min\left\{
	\frac12,
	\frac{\Norm{q}_\infty}{\Norm{\nabla q}_\infty}
	\right\}.
	\]
Split the integral around the singularity at $z=0$ into the regions
$|z|<d$, $d\leq|z|\leq1$, and $|z|>1$.
Since $\cK$ is odd,
\[
\pv\int_{|z|<d}\cK(z)\,\dd z=0,
\]
so we may replace $q(x-z)$ by $q(x-z)-q(x)$ in the first region.
Thus,
	\begin{align*}
		u[q](x)
		={}&
		\pv\int_{|z|<d}
		\cK(z)\bigl(q(x-z)-q(x)\bigr)\,\dd z\\
		&+
		\int_{d\leq|z|\leq1}\cK(z)q(x-z)\,\dd z
		+
		\int_{|z|>1}\cK(z)q(x-z)\,\dd z.
	\end{align*}
	Since $|\cK(z)|\leq C|z|^{-2}$, the three terms on the right are
	bounded respectively by
	\[
	C\Norm{\nabla q}_\infty d,
	\quad
	C\Norm{q}_\infty\log\frac1d,
	\quad
	C\Norm{q}_1.
	\]
	The choice of $d$, together with the fixed support size at unit scale,
	therefore gives
	\[
	\Norm{u[q]}_\infty
	\leq
	C\Norm{q}_\infty
	\left(
	1+\log\frac1d
	\right).
	\]
	Finally,
	\[
	d^{-1}
	\leq
	2+\frac{\Norm{\nabla q}_\infty}{\Norm{q}_\infty},
	\]
	which proves \eqref{eq:log-CZ}.
\end{proof}

\begin{lemma}\label{lem:far-field}
	Let $E\subset\R^2$ and suppose that, for some $D>0$,
	\[
	\operatorname{dist}(E,\supp f)\geq D.
	\]
	If $f\in L^1(\R^2)$, then, for every integer $j\geq0$,
	\begin{equation}\label{eq:far-field-L1}
		\Norm{\nabla^j u[f]}_{L^\infty(E)}
		\leq
		CD^{-2-j}\Norm{f}_{L^1(\R^2)}.
	\end{equation}
	If $f\in L^2(\R^2)$ and $|E|<\infty$, then
	\begin{equation}\label{eq:far-field-L2}
		\Norm{u[f]}_{L^2(E)}
		\leq
		C|E|^{\frac12}D^{-1}\Norm{f}_{L^2(\R^2)}.
	\end{equation}
\end{lemma}

\begin{proof}
	For $x\in E$ and $y\in\supp f$, we have $|x-y|\geq D$. Since
	\[
	|\nabla^j\cK(z)|
	\leq C|z|^{-2-j},
	\]
	it follows that
	\[
	|\nabla^j u[f](x)|
	\leq
	\int_{\R^2}
	|\nabla^j\cK(x-y)|\,|f(y)|\,\dd y
	\leq
	CD^{-2-j}\Norm{f}_{L^1(\R^2)}.
	\]
	Taking the supremum over $x\in E$ proves
	\eqref{eq:far-field-L1}.
	
	If $f\in L^2(\R^2)$, the Cauchy--Schwarz inequality gives, for every
	$x\in E$,
	\[
	|u[f](x)|
	\leq
	C\Norm{f}_{L^2(\R^2)}
	\left(
	\int_{|z|\geq D}|z|^{-4}\,\dd z
	\right)^{\frac12}
	\leq
	CD^{-1}\Norm{f}_{L^2(\R^2)}.
	\]
	Taking the $L^2$ norm over $E$ proves
	\eqref{eq:far-field-L2}.
\end{proof}

\subsection{Exponential decay of the radial velocity}
For an $N$-fold angularly periodic function $q$, the low angular modes of the radial component of $u[q]$ vanish. The next lemma shows that this radial velocity is exponentially small in $N$ on annuli separated from $\supp q$.
\begin{lemma}[Exponential decay on separated annuli]
	\label{lem:radial-velocity-decay}
	Let $0<a<b$, $N\in\N_+$, and $R>0$. Let
	$E\subset\R^2\setminus\{0\}$ be a bounded open annulus such that
	\[
	d_*:=
	\inf_{\substack{x\in E\\ a\leq\rho\leq b}}
	\left|\log\frac{|x|}{\rho}\right|>0,
	\]
	and set $E_R:=\{Rx:x\in E\}$. Suppose $q\in L^\infty(\R^2)$ is
	$2\pi/N$-periodic in the angular variable and supported in
	$\bigl\{x\in\R^2:aR<|x|<bR\bigr\}$.
	Then, for every integer $j\geq0$ and every $0<\alpha<1$,
	\begin{equation}\label{eq:radial-velocity-decay}
		\Norm{\nabla^j u_r[q]}_{L^\infty(E_R)}
		+
		R^\alpha
		[\nabla^j u_r[q]]_{C^\alpha(E_R)}
		\leq
		C_{j}R^{-j}(1+N)^{j+2}e^{-\frac{d_*N}{3} }
		\Norm{q}_\infty.
	\end{equation}
	Here $C_j>0$ depends only on $a$, $b$, $E$ and $j$,
	and is independent of $N$ and $R$.
\end{lemma}

\begin{proof}
We first prove the estimate when $R=1$. Let
$I_E:=\{|x|:x\in\overline E\}$. Substituting
$x=r(\cos\varphi,\sin\varphi)$ and
$y=\rho(\cos\psi,\sin\psi)$ into \eqref{Biot} and taking the radial
component, we obtain
	\begin{equation}\label{eq:radial-angular-kernel}
		u_r[q](r,\varphi)
		=
		\int_a^b\int_{-\pi}^{\pi}
		\mathscr K(r,\rho,\varphi-\psi)q(\rho,\psi)
		\,\dd\psi\dd\rho,
	\end{equation}
	where
	\[
	\mathscr K(r,\rho,\beta)
	=
	-c
	\frac{\rho^2\sin\beta}
	{\left(r^2+\rho^2-2r\rho\cos\beta\right)^{3/2}}.
	\]
	
To estimate the angular Fourier coefficients of $\mathscr K$, we
extend the angular variable to the complex plane. For $z\in\mathbb C$,
\[
r^2+\rho^2-2r\rho\cos z
=
2r\rho\left(
\cosh\left(\log\frac r\rho\right)-\cos z
\right).
\]
Hence
\[
r^2+\rho^2-2r\rho\cos z=0
\quad\Longleftrightarrow\quad
z=2\pi n\pm i\left|\log\frac r\rho\right|,
\quad n\in\mathbb Z.
\]
By the definition of $d_*$, the denominator
$r^2+\rho^2-2r\rho\cos z$ does not vanish when
$|\operatorname{Im}z|<d_*$. The strip is simply connected, so we
choose a single-valued holomorphic branch of its $\frac32$-power
that is positive on the real axis. Thus $z\mapsto\mathscr K(r,\rho,z)$
is holomorphic in this strip and $2\pi$-periodic.
	
	Compactness and periodicity imply that, for every integer
	$\ell\geq0$,
	\[
	\sup_{\substack{r\in I_E,\ \rho\in[a,b]\\
			|\operatorname{Im}z|\leq d_*/2}}
	\left|\partial_r^\ell\mathscr K(r,\rho,z)\right|
	\leq C_\ell.
	\]
	Define the angular Fourier coefficients of the kernel by
	\[
	\widehat{\mathscr K}_m(r,\rho)
	:=
	\frac{1}{2\pi}
	\int_{-\pi}^{\pi}
	\mathscr K(r,\rho,\beta)e^{-im\beta}\,\dd\beta.
	\]
	For $m\neq0$, Cauchy's theorem and periodicity allow us to shift the
	path of integration from $[-\pi,\pi]$ to
	\[
	z=\beta-i\operatorname{sgn}(m)\frac{d_*}{2},
	\quad -\pi\leq\beta\leq\pi.
	\]
	Along this path,
	\[
	\left|e^{-imz}\right|=e^{-d_*|m|/2}.
	\]
The contour shift therefore gives
	\[
	\left|
	\partial_r^{\ell_1}\widehat{\mathscr K}_m(r,\rho)
	\right|
	\leq
	C_{\ell_1}e^{-d_*|m|/2},
	\quad m\neq0.
	\]
	It follows that, for all nonnegative integers $\ell_1$ and $\ell_2$,
	\begin{equation}\label{eq:kernel-Fourier-decay}
		\left|
		(im)^{\ell_2}
		\partial_r^{\ell_1}
		\widehat{\mathscr K}_m(r,\rho)
		\right|
		\leq
		C_{\ell_1}
		(1+|m|)^{\ell_2}e^{-d_*|m|/2}.
	\end{equation}
	
	Since $q$ is only assumed to belong to $L^\infty$, we use its angular
	periodicity without relying on pointwise convergence of its Fourier
	series. To this end, define the averaged kernel
	\[
	\mathscr K_N(r,\rho,\beta)
	:=
	\frac1N\sum_{h=0}^{N-1}
	\mathscr K\!\left(r,\rho,\beta+\frac{2\pi h}{N}\right).
	\]
	For each $h$, shifting the angular integration variable in
	\eqref{eq:radial-angular-kernel} by $2\pi h/N$ leaves $q$ unchanged.
	Averaging these identities gives
	\[
	u_r[q](r,\varphi)
	=
	\int_a^b\int_{-\pi}^{\pi}
	\mathscr K_N(r,\rho,\varphi-\psi)q(\rho,\psi)
	\,\dd\psi\dd\rho.
	\]
	The finite sum satisfies
	\[
	\frac1N\sum_{h=0}^{N-1}e^{2\pi imh/N}
	=
	\begin{cases}
		1, & m\in N\mathbb Z,\\
		0, & m\notin N\mathbb Z.
	\end{cases}
	\]
Thus the absolutely and uniformly convergent Fourier expansion of the kernel yields
	\[
	\mathscr K_N(r,\rho,\beta)
	=
	\sum_{k\in\mathbb Z\setminus\{0\}}
	\widehat{\mathscr K}_{kN}(r,\rho)e^{ikN\beta},
	\]
	where the zero mode vanishes because
	$\mathscr K(r,\rho,\cdot)$ is odd. 
	
	On $E$, Cartesian derivatives are expressed in polar coordinates as
	\begin{equation}\label{eq:Cartesian-polar-derivatives}
		\partial_{x_1}
		=
		\cos\varphi\,\partial_r
		-\frac{\sin\varphi}{r}\,\partial_\varphi,
		\quad
		\partial_{x_2}
		=
		\sin\varphi\,\partial_r
		+\frac{\cos\varphi}{r}\,\partial_\varphi.
	\end{equation}
Since $r$ is bounded away from zero on $E$,
\eqref{eq:Cartesian-polar-derivatives} and
\eqref{eq:kernel-Fourier-decay} give absolute and uniform convergence
after differentiation through order $j$. We therefore differentiate
under the integral to obtain
	\begin{align*}
		\Norm{\nabla^j u_r[q]}_{L^\infty(E)}
		&\leq
		C_j\Norm{q}_\infty
		\sum_{k\neq0}
		(1+|k|N)^j e^{-d_*|k|N/2}\\
		&\leq
		C_j(1+N)^{j+1}e^{-\frac{d_*N}{3} }\Norm{q}_\infty.
	\end{align*}

	Interpolation on the fixed annulus $E$, together with the same
	estimate at order $j+1$, gives
	\[
	\begin{aligned}
		[\nabla^j u_r[q]]_{C^\alpha(E)}
		&\leq
		C_{j}
		\Norm{\nabla^j u_r[q]}_{L^\infty(E)}^{1-\alpha}
		\Norm{\nabla^{j+1}u_r[q]}_{L^\infty(E)}^\alpha\\
		&\leq
		C_{j}(1+N)^{j+1+\alpha}
		e^{-\frac{d_*N}{3} }\Norm{q}_\infty.
	\end{aligned}
	\]
	Since $j+1+\alpha<j+2$, these two bounds prove
	\eqref{eq:radial-velocity-decay} when $R=1$.
	
	For general $R>0$, set $\widetilde q(y):=q(Ry)$. Then
	$\widetilde q$ is supported in $\{a<|y|<b\}$, has the same angular
	periodicity and satisfies
	$\Norm{\widetilde q}_\infty=\Norm{q}_\infty$.
	By \eqref{eq:Riesz-scaling}, taking radial components gives
	$u_r[\widetilde q](y)=u_r[q](Ry)$. Hence
	\[
	\begin{aligned}
		\Norm{\nabla^j u_r[q]}_{L^\infty(E_R)}
		&=
		R^{-j}
		\Norm{\nabla^j u_r[\widetilde q]}_{L^\infty(E)},\\
		[\nabla^j u_r[q]]_{C^\alpha(E_R)}
		&=
		R^{-j-\alpha}
		[\nabla^j u_r[\widetilde q]]_{C^\alpha(E)}.
	\end{aligned}
	\]
	Multiplying the second identity by $R^\alpha$ and applying the
	$R=1$ estimate to $\widetilde q$ proves
	\eqref{eq:radial-velocity-decay}.
\end{proof}

\section{Construction of the approximate solution}
\label{sec:pseudo}

An outer radial profile generates a nondegenerate angular shear on the inner annulus. We transport an oscillatory profile by this flow to obtain the approximate solution and the desired Sobolev growth.

\subsection{Profiles and nondegenerate angular shear}

Choose real-valued radial profiles $f,g\in C_c^\infty(\R^2)$ satisfying
\begin{equation}\label{eq:profiles}
0\leq f\in C_c^\infty((4,5)),\quad g\in C_c^\infty((1,2)),\quad f,g\not\equiv0.
\end{equation}
Set
\[
\psi_f:=\Lambda^{-1}f,
\quad
u[f]:=\nabla^\perp\psi_f.
\]
By \cref{lem:radial}, $\psi_f$ is radial and, for $r>0$,
\[
u[f](r,\varphi)
=
\partial_r\psi_f(r)\evarphi,
\quad
u_r[f](r)=0,
\quad
u_\varphi[f](r)=\partial_r\psi_f(r).
\]
The velocity generated by $f$ is purely angular. We therefore define its angular speed by
\begin{equation}\label{eq:angular-velocity}
\Omega(r):=\frac{u_\varphi[f](r)}{r}=\frac{\partial_r\psi_f(r)}{r}.
\end{equation}

\begin{lemma}[Nondegenerate convex shear on the inner annulus]
	\label{lem:nondegenerate}
	Let $f$ and $g$ satisfy \eqref{eq:profiles}, and let $\Omega$ be
	defined by \eqref{eq:angular-velocity}. Set
	$\mathcal I:=\{x\in\R^2:1/2<|x|<3\}$. Then there exists a constant
	$c_f>0$, depending only on $f$, such that
	\begin{equation}\label{eq:nondegenerate}
		\Omega'(r)\geq c_f,
		\quad
		\Omega''(r)\geq c_f,
		\quad \frac12\leq r\leq3.
	\end{equation}
	In particular, both bounds hold for $r\in\supp g$.
\end{lemma}

\begin{proof}
	Part~\textup{(i)} of
	\cref{lem:appendix-SQG-shear} gives
	\[
	\Omega'(r)
	\geq
	8c_Sa_2M_2r,
	\quad
	\Omega''(r)
	\geq
	8c_Sa_2M_2,
	\quad 0<r<4,
	\]
	where
	\[
	M_2:=\int_4^5 f(\rho)\rho^{-4}\,\dd\rho.
	\]
	Since $f\geq0$ and $f\not\equiv0$, we have $M_2>0$. Set
	\[
	c_f:=4c_Sa_2M_2>0.
	\]
	For $\frac12\leq r\leq3$, the preceding estimate gives
	\[
	\Omega'(r)\geq 2c_fr\geq c_f,
	\quad
	\Omega''(r)\geq 2c_f,
	\]
	which proves \eqref{eq:nondegenerate}. The last assertion follows
	from $\supp g\subset(1,2)\subset(1/2,3)$.
\end{proof}

\subsection{The initial datum}

Fix a small constant $c_0>0$ and assume $1<s<2$. Let $\lambda\geq1$
and $N\in\mathbb N$, $N\geq2$ be large. For
$x=r(\cos\varphi,\sin\varphi)$, define
\begin{equation}\label{eq:initial-datum}
	\begin{aligned}
		\theta_{\rad,0}(x)
		&:=c_0\lambda^{1-s}f(\lambda r),\\
		\theta_{\osc,0}(x)
		&:=c_0\lambda^{1-s}N^{-s}g(\lambda r)\cos(N\varphi),\\
		\theta_0(x)
		&:=\theta_{\rad,0}(x)+\theta_{\osc,0}(x).
	\end{aligned}
\end{equation}
The support conditions on $f$ and $g$ give
\[
\supp\theta_{\rad,0}
\subset
\left\{x\in\R^2:
4\lambda^{-1}<|x|<5\lambda^{-1}\right\},
\]
and
\[
\supp\theta_{\osc,0}
\subset
\left\{x\in\R^2:
\lambda^{-1}<|x|<2\lambda^{-1}\right\}.
\]
Thus the two components are supported in disjoint annuli, with the
radial component outside the oscillatory one.

\begin{lemma}\label{lem:angular-Sobolev}
	Let $0<a<b<\infty$ and let
	\[
	h\in C_c^\infty((a,b);\mathbb C),
	\quad
	h\not\equiv0.
	\]
	For $\lambda\geq1$, an integer $N\geq2$ and $0\leq\rho\leq2$,
	\begin{equation}\label{eq:angular-Sobolev}
		\Norm{h(\lambda r)e^{iN\varphi}}_{H^\rho(\R^2)}
		\simeq_{h,\rho}
		\lambda^{\rho-1}N^\rho.
	\end{equation}
	If $h$ is real-valued, \eqref{eq:angular-Sobolev} also holds with
	$e^{iN\varphi}$ replaced by $\cos(N\varphi)$.
\end{lemma}

\begin{proof}
	Set
	\[
	F_N(r,\varphi):=h(r)e^{iN\varphi}.
	\]
	Since $F_N$ is supported in a fixed annulus away from the
	origin, its Sobolev norms in Cartesian and polar coordinates
	are equivalent, with constants independent of $N$.
	Moreover, $F_N$ has only the angular Fourier mode $N$, so the
	Fourier characterization on 
	$\R\times\T$ gives
	\[
	\Norm{F_N}_{H^\rho(\R^2)}^2
	\simeq
	\int_{\R}
	(1+\xi^2+N^2)^\rho
	|\widehat h(\xi)|^2\,\dd\xi
	\simeq_{h,\rho}
	N^{2\rho},
	\]
	where $h$ is extended by zero outside $(a,b)$. For $\rho>0$, the
	Poincar\'e inequality on the fixed annulus also yields
	\[
	\Norm{F_N}_{\dot H^\rho}
	\simeq
	\Norm{F_N}_{H^\rho}
	\simeq_{h,\rho}
	N^\rho.
	\]
	
	Under the rescaling $x\mapsto\lambda x$, the polar variables transform
	as $(r,\varphi)\mapsto(\lambda r,\varphi)$. Hence
	\[
	F_N(\lambda x)=h(\lambda r)e^{iN\varphi}.
	\]
	Combining the $L^2$ scaling identity with
	\eqref{eq:Sobolev-scaling}, we obtain, for $\rho>0$,
	\begin{align*}
		\Norm{F_N(\lambda\cdot)}_{H^\rho}^2
		&\simeq
		\lambda^{-2}\Norm{F_N}_2^2
		+
		\lambda^{2\rho-2}
		\Norm{F_N}_{\dot H^\rho}^2\\
		&\simeq_{h,\rho}
		\lambda^{-2}
		+
		\lambda^{2\rho-2}N^{2\rho}\\
		&\simeq
		\lambda^{2\rho-2}N^{2\rho},
	\end{align*}
		where in the last step we used $\lambda\ge1$ and $N\ge2$ to discard the $\lambda^{-2}$ term. If $h$ is real-valued, applying \cref{lem:real-complex-angular} with $\psi=0$ gives the corresponding estimate for $\cos(N\varphi)$.
\end{proof}

\begin{corollary}\label{cor:initial-small}
For $\theta_{\rad,0}, \theta_{\osc,0}$ and $\theta_0$ as defined in \eqref{eq:initial-datum}, we have
	\begin{align}
		\|\theta_0\|_{H^s} &\lesssim_{f,g,s} c_0, \label{eq:initial-small}\\
	\|\theta_{\rad,0}\|_2 \simeq_f c_0\lambda^{-s},\quad&	\|\theta_{\osc,0}\|_2 \simeq_g c_0\lambda^{-s}N^{-s}. \label{eq:osc-L2-size}
	\end{align}
\end{corollary}

\begin{proof}
	For $\rho=s$, Sobolev scaling and \cref{lem:angular-Sobolev} yield
	\[
	\Norm{\theta_{\rad,0}}_{H^s}
	\lesssim_f
	c_0\lambda^{1-s}\lambda^{s-1}
	=c_0,
	\]
	and
	\[
	\Norm{\theta_{\osc,0}}_{H^s}
	\lesssim_{g,s}
	c_0\lambda^{1-s}N^{-s}\lambda^{s-1}N^s
	=c_0.
	\]
Since $\theta_0=\theta_{\rad,0}+\theta_{\osc,0}$, the triangle inequality gives \eqref{eq:initial-small}.
The case $\rho=0$ follows directly from $L^2$ scaling:
	\[
	\|\theta_{\rad,0}\|_2 \simeq_f c_0\lambda^{-s},\quad 	\Norm{\theta_{\osc,0}}_2
	\simeq_g
	c_0\lambda^{1-s}N^{-s}\lambda^{-1}
	=
	c_0\lambda^{-s}N^{-s},
	\]
	which proves \eqref{eq:osc-L2-size}.
\end{proof}

\subsection{The approximate solution}
By \cref{lem:radial}, \eqref{eq:Riesz-scaling}, and the definition of $\theta_{\rad,0}$, the background velocity is
\[
u[\theta_{\rad,0}](r,\varphi)
=
c_0\lambda^{1-s}u_\varphi[f](\lambda r)\evarphi
=
c_0\lambda^{2-s}\Omega(\lambda r)r\evarphi.
\]
Set
\begin{equation}\label{eq:K-definition}
	K:=c_0\lambda^{2-s}.
\end{equation}
Then
\[
u[\theta_{\rad,0}](r,\varphi)=Kr\Omega(\lambda r)\evarphi, ~~
u[\theta_{\rad,0}]\cdot\nabla=K\Omega(\lambda r)\partial_\varphi.
\]

The approximate solution $\overline\theta$ satisfies the transport equation driven by the stationary background velocity:
\begin{equation*}
	\begin{cases}
		\partial_t\overline\theta+u[\theta_{\rad,0}]\cdot\nabla\overline\theta=0,\\
		\overline\theta(\cdot,0)=\theta_0.
	\end{cases}
\end{equation*}
Recall from \eqref{eq:initial-datum} that
$\theta_0=\theta_{\rad,0}+\theta_{\osc,0}$.
Linearity decomposes the approximate solution as
$\overline\theta=\overline\theta_{\rad}+\overline\theta_{\osc}$,
with initial data $\theta_{\rad,0}$ and $\theta_{\osc,0}$, respectively.
Since $\theta_{\rad,0}$ is radial, it is stationary under the flow
generated by $u[\theta_{\rad,0}]$, and hence
\[
\overline\theta_{\rad}(r,t)
=
\theta_{\rad,0}(r)
=
c_0\lambda^{1-s}f(\lambda r).
\]
Moreover,
$u[\theta_{\rad,0}]\cdot\nabla
=K\Omega(\lambda r)\partial_\varphi$, so the oscillatory part satisfies
\begin{equation}\label{eq:approximate-oscillatory}
	\left(
	\partial_t+K\Omega(\lambda r)\partial_\varphi
	\right)\overline\theta_{\osc}=0,
	\qquad
	\overline\theta_{\osc}(\cdot,0)=\theta_{\osc,0}.
\end{equation}
Solving \eqref{eq:approximate-oscillatory} by characteristics gives
\begin{align}
	\overline\theta(r,\varphi,t)
	&=
	\theta_0\bigl(
	r,\varphi-Kt\Omega(\lambda r)
	\bigr)\notag\\
	&=
	\underbrace{
		c_0\lambda^{1-s}f(\lambda r)
	}_{\overline\theta_{\rad}(r,t)}
	+
	\underbrace{
		c_0\lambda^{1-s}N^{-s}g(\lambda r)
		\cos\bigl(N\varphi-NKt\Omega(\lambda r)\bigr)
	}_{\overline\theta_{\osc}(r,\varphi,t)}.
	\label{eq:explicit-approximate}
\end{align}
Let
\[
\Phi_t(r,\varphi)
:=
N\varphi-NKt\Omega(\lambda r)
\]
denote the oscillatory phase. On
$\supp\overline\theta_{\osc}(t)$, we have
$\lambda r\in\supp g$. Hence \cref{lem:nondegenerate} gives
\[
|\partial_r\Phi_t(r,\varphi)|
=
\lambda NKt\Omega'(\lambda r)
\simeq_f
\lambda NKt,
\quad t>0.
\]
We next convert the growth of the phase gradient into a Sobolev estimate. For an ansatz $a(x)e^{i\mu\phi(x)}$ with $\mu\gg1$, a uniform Wentzel--Kramers--Brillouin (WKB) estimate gives a lower bound away from $t=0$ and an upper bound on the full time interval.

\begin{lemma}[WKB estimate for the oscillatory component]
	\label{lem:WKB}
	Let $0\leq\rho\leq2$ and fix $0<\tau<T$.
	For $\lambda$ sufficiently large depending on $\tau$ and $T$,
	we have
	\begin{equation}\label{eq:WKB}
		\Norm{\overline\theta_{\osc}(t)}_{H^\rho}
		\simeq_{f,g,\rho}
		\Norm{\theta_{\osc,0}}_2
		(\lambda NKt)^\rho,
		\quad t\in[\tau,T].
	\end{equation}
	In view of \eqref{eq:K-definition} and
	\eqref{eq:osc-L2-size}, this estimate can also be written as
	\begin{equation}\label{eq:pseudo-growth-explicit}
		\Norm{\overline\theta_{\osc}(t)}_{H^\rho}
		\simeq_{f,g,\rho}
		c_0^{1+\rho}t^\rho
		\lambda^{\rho(3-s)-s}N^{\rho-s},
		\quad t\in[\tau,T].
	\end{equation}
	On the full interval $[0,T]$, one also has
	\begin{equation}\label{eq:WKB-upper}
		\Norm{\overline\theta_{\osc}(t)}_{H^\rho}
		\lesssim_{f,g,\rho} \Norm{\theta_{\osc,0}}_2
		\bigl(\lambda N(1+Kt)\bigr)^\rho,
		\quad 0\leq t\leq T.
	\end{equation}
\end{lemma}

\begin{proof}
	We first establish the oscillatory estimate at a fixed spatial scale. Let
	$U\subset\R^2$ be open and let
	\[
	a\in C_c^2(U;\mathbb C),
	\quad
	a\not\equiv0.
	\]
	Suppose that $\phi:\mathbb R^2\to\mathbb R$ satisfies
	$\phi|_U\in C^3(U;\mathbb R)$ and
	\[
	\Norm{\phi}_{C^3(U)}\leq C_0,
	\quad
	\inf_{\supp a}|\nabla\phi|\geq m>0.
	\]
	We claim that, for every $0\leq\rho\leq2$, there exists $\mu_0\geq1$
	such that
	\begin{equation}\label{eq:uniform-WKB}
		\Norm{ae^{i\mu\phi}}_{H^\rho(\R^2)}
		\simeq_{a,U,C_0,m,\rho}
		\mu^\rho,
		\quad
		\mu\geq\mu_0.
	\end{equation}
The implicit constants and $\mu_0$ depend only on
$a$, $U$, $C_0$, $m$, and $\rho$.
	
	We first consider $0<\rho<1$ and set $F_\mu:=ae^{i\mu\phi}$.
	The pointwise estimate
	\[
	|F_\mu(y+z)-F_\mu(y)|
	\lesssim
	\min\{1,\mu|z|\}
	\]
	follows from the uniform bounds on $a$ and $\phi$.
	The difference $F_\mu(\cdot+z)-F_\mu$ is supported in
	$\supp a\cup(\supp a-z)$.
	Hence
	\[
	\Norm{F_\mu(\cdot+z)-F_\mu}_2^2
	\leq C\min\{1,\mu^2|z|^2\}.
	\]
	The Gagliardo characterization then gives
	\[
	[ae^{i\mu\phi}]_{\dot H^\rho}^2
	\lesssim
	\mu^2\int_0^{\mu^{-1}}r^{1-2\rho}\,\dd r
	+
	\int_{\mu^{-1}}^\infty r^{-1-2\rho}\,\dd r
	\lesssim
	\mu^{2\rho}.
	\]
	
	For the lower bound, choose a compact set $E\subset U$ of positive
	measure on which the amplitude does not vanish: for some $a_0>0$,
	\[
	|a(y)|\geq a_0
	\quad\text{for every }y\in E.
	\]
	For each \(y\in E\), set
	\[
	v_y:=\nabla\phi(y),
	\quad
	\zeta_y:=\frac{\pi}{2}\frac{v_y}{|v_y|^2}.
	\]
Since \(m\leq |v_y|\leq C_0\), we may choose
\(\delta=\delta(C_0,m)>0\) sufficiently small and set
\(\mathcal Z_y:=B(\zeta_y,\delta)\) so that, for every
\(y\in E\) and \(\zeta\in\mathcal Z_y\),
\[
\frac{\pi}{3}
\leq v_y\cdot\zeta
\leq\frac{2\pi}{3},
\quad
|\zeta|\simeq_{C_0,m}1.
\]
Since $E$ is compact and the sets $\mathcal Z_y$ are uniformly
bounded, for all sufficiently large $\mu$ we have
$
y+\frac{\zeta}{\mu}\in U
$
uniformly for $y\in E$ and $\zeta\in\mathcal Z_y$. For \(z=\zeta/\mu\), Taylor expansion gives
\[
\mu\bigl(\phi(y+z)-\phi(y)\bigr)
=
v_y\cdot\zeta+O(\mu^{-1}),
\quad
a(y+z)=a(y)+O(\mu^{-1}).
\]
The phase increment stays away from a multiple of \(2\pi\) on \(\mathcal Z_y\). Together with \(|a(y)|\geq a_0\) on \(E\), this yields a constant \(c>0\) and a sufficiently large \(\mu_1\ge1\) such that
\[
\left|
a\Bigl(y+\frac{\zeta}{\mu}\Bigr)
e^{i\mu\phi(y+\zeta/\mu)}
-
a(y)e^{i\mu\phi(y)}
\right|
\geq c
\]
for all \(y\in E\), \(\zeta\in\mathcal Z_y\), and \(\mu\ge\mu_1\).

	Restricting the Gagliardo integral to $y\in E$ and
	$z\in\mu^{-1}\mathcal Z_y$, we obtain
	\[
	[ae^{i\mu\phi}]_{\dot H^\rho}^2
	\gtrsim
	\int_E\int_{\mu^{-1}\mathcal Z_y}
	\frac{1}{|z|^{2+2\rho}}\,\dd z\dd y
	\gtrsim
	\mu^{2\rho}.
	\]
	Here we used $|z|\simeq\mu^{-1}$ and
	$|\mu^{-1}\mathcal Z_y|\simeq\mu^{-2}$ uniformly in $y$. Together
	with the preceding upper bound, this proves
	\eqref{eq:uniform-WKB} for $0<\rho<1$.
	
	For $1<\rho<2$, write $\rho=1+\beta$ with $0<\beta<1$.
	For $j=1,2$, differentiation gives
	\begin{equation}\label{eq:WKB-first-derivative}
		\partial_j(ae^{i\mu\phi})
		=
		i\mu(a\,\partial_j\phi)e^{i\mu\phi}
		+
		(\partial_j a)e^{i\mu\phi}.
	\end{equation}
	Set $b:=a\nabla\phi$. On the fixed set $E$ chosen above,
	$|b|\geq a_0m$, and the $C^1$ norm of $b$ is uniformly bounded.
	Applying the same difference-quotient argument to the vector
	amplitude $b$ gives
	\[
	c\mu^\beta
	\leq\Norm{be^{i\mu\phi}}_{H^\beta}
	\leq C\mu^\beta.
	\]
	The upper bound also gives
	$\Norm{(\nabla a)e^{i\mu\phi}}_{H^\beta}\leq C\mu^\beta$.
	Thus \eqref{eq:WKB-first-derivative} and the triangle inequality yield
	\[
	\Norm{\nabla(ae^{i\mu\phi})}_{H^\beta}
	\geq c\mu^{1+\beta}-C\mu^\beta
	\geq\frac{c}{2}\mu^{1+\beta}
	\]
	for sufficiently large $\mu$, uniformly in the allowed phases.
	Together with the corresponding upper bound, this gives
	\[
	\sum_{j=1}^2
	\Norm{\partial_j(ae^{i\mu\phi})}_{H^\beta}
	\simeq \mu^{1+\beta}.
	\]
	Together with the $L^2$ bound, this proves \eqref{eq:uniform-WKB}
for $1<\rho<2$. The cases $\rho=0$ and $\rho=1$ follow directly
	from the $L^2$ norm and \eqref{eq:WKB-first-derivative}, respectively.
	
For $\rho=2$, differentiating twice gives, for $j,k=1,2$,
\[
\partial_j\partial_k(ae^{i\mu\phi})
=
-\mu^2a\,\partial_j\phi\,\partial_k\phi\,e^{i\mu\phi}
+R_{\mu,jk},
\quad
\Norm{R_{\mu,jk}}_2\lesssim\mu.
\]
Since $|\nabla\phi|$ is bounded above and below by positive
constants on $\supp a$ and $a\not\equiv0$, the leading terms
on the right-hand side have total $L^2$ norm comparable to $\mu^2$.
Thus, for sufficiently large $\mu$, the triangle inequality gives
\[
\sum_{j,k=1}^2
\Norm{\partial_j\partial_k(ae^{i\mu\phi})}_2
\simeq \mu^2.
\]
Together with the $L^2$ and first derivative estimates, this proves
	\eqref{eq:uniform-WKB} for $\rho=2$ and completes the proof of
the auxiliary WKB estimate \eqref{eq:uniform-WKB}.

At the fixed scale, set $H_t(R,\varphi):=e^{iN(\varphi-Kt\Omega(R))}$. For the lower bound, choose a smooth cutoff $\chi$, supported in a coordinate chart on which a smooth branch of the polar angle
$\varphi$ is defined, such that $\chi g\not\equiv0$. For $t\in[\tau,T]$ and $K\tau\geq1$, the local phase $\psi_t:=\frac{\varphi}{Kt}-\Omega(R)$ has uniformly bounded $C^3$ norm and satisfies
\[
|\nabla\psi_t|^2
=\frac{1}{K^2t^2R^2}+|\Omega'(R)|^2
\geq c_f^2
\]
on $\supp(\chi g)$. Applying \eqref{eq:uniform-WKB} to $\chi gH_t$ with $\mu=NKt$, and using boundedness of multiplication by $\chi$ on $H^\rho$, gives
\[
\|gH_t\|_{H^\rho}\gtrsim_{f,g,\rho}(NKt)^\rho
\]
for sufficiently large $\lambda$.

For the upper bound, $N\in\mathbb N$ makes $gH_t$ globally smooth. Direct differentiation gives
\[
\|gH_t\|_2=\|g\|_2,\qquad
\|gH_t\|_{H^2}\lesssim_{f,g}[N(1+Kt)]^2.
\]
Interpolation and $1+Kt\lesssim Kt$ on $[\tau,T]$ yield
\begin{equation}\label{eq:fixed-scale-growth}
	\|gH_t\|_{H^\rho}
	\simeq_{f,g,\rho}(NKt)^\rho,
	\qquad t\in[\tau,T].
\end{equation}
The same interpolation on $[0,T]$ gives
\begin{equation}\label{eq:fixed-scale-upper}
	\|gH_t\|_{H^\rho}
	\lesssim_{f,g,\rho}[N(1+Kt)]^\rho,
	\qquad 0\leq t\leq T.
\end{equation}
By \cref{lem:real-complex-angular}, both estimates hold for $g\operatorname{Re}H_t$.	

Rescaling from $y=\lambda x$ to $x$, we obtain, for $\rho>0$,
\begin{equation}\label{eq:osc-Sobolev-scaling}
	\Norm{\overline\theta_{\osc}(t)}_{H^\rho}
	\simeq_\rho
	c_0\lambda^{1-s}N^{-s}
	\left(
	\lambda^{-1}\Norm{gH_t}_2
	+
	\lambda^{\rho-1}\Norm{gH_t}_{\dot H^\rho}
	\right).
\end{equation}
Since $gH_t$ is supported in a fixed annulus, the Poincar\'e inequality makes its $\dot H^\rho$ and $H^\rho$ norms equivalent. Applying \eqref{eq:fixed-scale-growth} in \eqref{eq:osc-Sobolev-scaling}, together with \eqref{eq:osc-L2-size}, proves \eqref{eq:WKB} on $[\tau,T]$. Using \eqref{eq:fixed-scale-upper} instead proves \eqref{eq:WKB-upper} on $[0,T]$. For $\rho=0$, both bounds follow directly from $L^2$ scaling. Finally, substituting $K=c_0\lambda^{2-s}$ and \eqref{eq:osc-L2-size} into \eqref{eq:WKB} gives \eqref{eq:pseudo-growth-explicit}.
\end{proof}

\section{Norm inflation across a Sobolev regularity gap}\label{sec:norm-inflation}

The approximate solution has the required Sobolev growth on
$[\tau,T]$, but we must still show that the exact solution exists
up to $T$. For this, we use coordinates defined by the flow of $u[\overline\theta_{\rad}]$: the large background transport disappears from the $C^{1,\alpha}$ estimate, while the remaining velocity is governed by a conjugated operator whose bound depends polynomially on $Kt$. The parameter choice below makes the time integral of this velocity small enough to close the continuation bootstrap. We then compare the exact and approximate solutions to transfer the Sobolev lower bound.

\subsection{Coordinates adapted to the background shear}
\label{sec:pullback}

Throughout this section, we work with the rescaled variable
\[
y=\lambda x
=
R(\cos\varphi,\sin\varphi),
\quad
R=|y|=\lambda r.
\]
Under this rescaling, the transport velocity becomes
$\lambda u[\cdot]$. Let $\cS_t$ denote the flow generated by $Ku[f]$.
Then
\begin{equation}\label{eq:background-shear-flow}
	\cS_t(R,\varphi)
	=
	\bigl(R,\varphi+Kt\Omega(R)\bigr),
	\quad
	\cS_t^{-1}(R,\varphi)
	=
	\bigl(R,\varphi-Kt\Omega(R)\bigr).
\end{equation}
Although $\cS_t$ is written in polar coordinates, all spatial
derivatives and H\"older norms below are taken with respect to the
Cartesian variable $y$. In particular, $D\cS_t(y)$ denotes the
Jacobian matrix in the standard Cartesian basis.

Both maps in \eqref{eq:background-shear-flow} preserve Lebesgue
measure and the radial coordinate.
By \eqref{eq:background-shear-flow}, we have
\begin{equation}\label{eq:background-flow-Linfty}
	\Norm{D\cS_t}_\infty
	+
	\Norm{D(\cS_t^{-1})}_\infty
	\leq C_f(1+Kt),
	\quad t\geq0,
\end{equation}
where both norms are taken over $\R^2$.
Differentiating $\cS_t^{-1}\circ\cS_t=\operatorname{Id}$ gives
\[
(D\cS_t(y))^{-1}
=
D(\cS_t^{-1})(\cS_t(y)).
\]
Consequently, the derivative of the inverse has the same supremum norm
as the pointwise matrix inverse:
\[
\Norm{(D\cS_t)^{-1}}_\infty
=
\Norm{D(\cS_t^{-1})}_\infty.
\]

The derivative bounds give the composition estimate needed below.
\begin{lemma}\label{lem:composition}
	Let $0<\alpha<1$, $t\geq0$, and
	$\Phi\in\{\cS_t,\cS_t^{-1}\}$, with the maps defined in
	\eqref{eq:background-shear-flow}. For every scalar or vector field
	$h\in C^{1,\alpha}(\R^2)$,
	\begin{equation}\label{eq:composition-C1a}
		\Norm{h\circ\Phi}_{C^{1,\alpha}(\R^2)}
		\leq
		C_{f,\alpha}(1+Kt)^{1+\alpha}
		\Norm{h}_{C^{1,\alpha}(\R^2)}.
	\end{equation}
\end{lemma}

\begin{proof}
	Write $\Phi=\cS_t^\varepsilon$ with $\varepsilon\in\{-1,1\}$ and
	set $b(R):=\varepsilon Kt\Omega(R)$. For $R=|y|$, the Cartesian
	components of $\Phi$ are
	\begin{equation*}
	\begin{aligned}
		\Phi_1(y)&=y_1\cos b(R)-y_2\sin b(R),\\
		\Phi_2(y)&=y_1\sin b(R)+y_2\cos b(R).
	\end{aligned}
	\end{equation*}
	For $R>0$, differentiation in the Cartesian variable $y_i$ gives
	\begin{equation}\label{eq:shear-first-derivatives}
	\begin{aligned}
		\partial_i\Phi_1
		&=\delta_{i1}\cos b(R)-\delta_{i2}\sin b(R)
		-b'(R)\frac{y_i}{R}\Phi_2,\\
		\partial_i\Phi_2
		&=\delta_{i1}\sin b(R)+\delta_{i2}\cos b(R)
		+b'(R)\frac{y_i}{R}\Phi_1.
	\end{aligned}
	\end{equation}
	Here $\delta_{ij}$ denotes the Kronecker delta. From
	$\partial_iR=\frac{y_i}{R}$,
	$\partial_j\partial_iR=\frac{\delta_{ij}}{R}-\frac{y_iy_j}{R^3}$,
	and \eqref{eq:shear-first-derivatives}, we obtain, for
	$i,j,k\in\{1,2\}$,
	\begin{equation}\label{eq:shear-derivative-bounds}
		\begin{aligned}
			|\partial_i\Phi_k(y)|
			&\leq 1+KtR|\Omega'(R)|,\\
			|\partial_i\partial_j\Phi_k(y)|
			&\leq
			CKt\bigl(|\Omega'(R)|+R|\Omega''(R)|\bigr)
			+C(Kt)^2R|\Omega'(R)|^2.
		\end{aligned}
	\end{equation}
	
	By part (iii) of \cref{lem:appendix-SQG-shear}, both 
	$\Omega(|y|)$ and $\Phi$ are smooth at the origin.
	Moreover, \eqref{eq:appendix-Omega-global} gives
	\[
	|\Omega'(R)|\leq C_f(1+R)^{-4},
	\quad
	|\Omega''(R)|\leq C_f(1+R)^{-5}.
	\]
Consequently,
	\[
	\begin{aligned}
		\Norm{\partial_i\Phi_k}_\infty
		&\leq C_f(1+Kt),\\
		\Norm{\partial_i\partial_j\Phi_k}_\infty
		&\leq C_f\bigl(Kt+(Kt)^2\bigr)
		\leq C_f(1+Kt)^2.
	\end{aligned}
	\]
	Interpolation yields
	\[
	\begin{aligned}
		[\partial_i\Phi_k]_{C^\alpha}
		&\leq
		C_\alpha
		\Norm{\partial_i\Phi_k}_\infty^{1-\alpha}
		\left(
		\sum_{j=1}^2
		\Norm{\partial_j\partial_i\Phi_k}_\infty
		\right)^\alpha\\
		&\leq C_{f,\alpha}(1+Kt)^{1+\alpha}.
	\end{aligned}
	\]
	Since $\Phi$ is a diffeomorphism, the chain rule and the
	preceding bounds give
\[
\Norm{h\circ\Phi}_\infty=\Norm{h}_\infty,
\quad
\Norm{\partial_i(h\circ\Phi)}_\infty
\leq C_f(1+Kt)
\sum_{k=1}^2\Norm{\partial_k h}_\infty.
\]
The H\"older product estimate gives
	\[
	\begin{aligned}
		[\partial_i(h\circ\Phi)]_{C^\alpha}
		&\leq
		\sum_{k=1}^2
		\Bigl(
		[(\partial_k h)\circ\Phi]_{C^\alpha}
		\Norm{\partial_i\Phi_k}_\infty
		+
		\Norm{\partial_k h}_\infty
		[\partial_i\Phi_k]_{C^\alpha}
		\Bigr)\\
		&\leq
		C_{f,\alpha}(1+Kt)^{1+\alpha}
		\Norm{h}_{C^{1,\alpha}(\R^2)}.
	\end{aligned}
	\]
	Summing over $i$ and using the $L^\infty$ bound proves
	\eqref{eq:composition-C1a}.
\end{proof}

By \eqref{eq:background-shear-flow}, both $\cS_t$ and
$\cS_t^{-1}$ preserve the radial coordinate. We work on the
invariant annulus
\[
\mathcal A:=\left\{y\in\R^2:\frac12<|y|<6\right\}.
\]
We write $y=R\er(\varphi)$, where
\[
\er(\varphi)=(\cos\varphi,\sin\varphi),\qquad
\evarphi(\varphi)=(-\sin\varphi,\cos\varphi).
\]
These are the same polar unit vectors used for $(r,\varphi)$;
only the radial variable has changed from $r$ to $R$. Thus
$u_\nu[F](y)=u[F](y)\cdot e_\nu(\varphi)$ for
$\nu\in\{r,\varphi\}$. 
For $H\in C_c^{1,\alpha}(\R^2)$ with
$\supp H\subset\mathcal A$, define
\begin{equation}\label{eq:pullback-operator}
	\cB_t[H](y)
	:=(D\cS_t(y))^{-1}
	\lambda u[H\circ\cS_t^{-1}](\cS_t(y)).
\end{equation}
The profile $f$ is supported in $4<R<5$. By
\cref{lem:appendix-SQG-shear}, its angular velocity satisfies
$\Omega'>0$ on $0<R<4$. We therefore single out the inner annulus
\[
\mathcal I:=\left\{y\in\R^2:\frac12<|y|<3\right\}
\subset\mathcal A,
\qquad \inf_{\overline{\mathcal I}}\Omega'>0.
\]

On $\mathcal I$, the weighted radial identity cancels the $Kt$
term arising when the polar velocity components are differentiated.
On $\mathcal A$, this term is not canceled and yields an additional factor of $1+Kt$.
\begin{proposition}[Bounds for $\cB_t$]
	\label{prop:pullback-velocity}
	Let $0<\alpha<1$, $t\geq0$, and
	$H\in C_c^{1,\alpha}(\R^2)$ satisfy $\supp H\subset\mathcal A$.
	Then
	\begin{equation}\label{eq:pullback-velocity-bound}
		\Norm{\cB_t[H]}_{C^{1,\alpha}(\mathcal A)}
		\leq C_{f,\alpha}\lambda(1+Kt)^{2+2\alpha}
		\Norm H_{C^{1,\alpha}(\mathcal A)}.
	\end{equation}
	If, in addition, $\supp H\subset\mathcal I$, the sharper local bound
	\begin{equation}\label{eq:pullback-velocity-bound-inner}
		\Norm{\cB_t[H]}_{C^{1,\alpha}(\mathcal I)}
		\leq C_{f,\alpha}\lambda(1+Kt)^{1+2\alpha}
		\Norm H_{C^{1,\alpha}(\mathcal I)}
	\end{equation}
	holds. The constants are independent of $K,t,\lambda$, and $H$.
\end{proposition}

\begin{proof}
	Set $h:=H\circ\cS_t^{-1}$. Differentiating
	\eqref{eq:background-shear-flow} in the moving polar frames gives
	\begin{align*}
		D\cS_t\,\er(\varphi)
		&=\er(\varphi+Kt\Omega)
		+KtR\Omega'\evarphi(\varphi+Kt\Omega),\\
		D\cS_t\,\evarphi(\varphi)
		&=\evarphi(\varphi+Kt\Omega).
	\end{align*}
	In the corresponding polar bases, $(D\cS_t)^{-1}$ has matrix
	$\bigl(\begin{smallmatrix}1&0\\-KtR\Omega'&1\end{smallmatrix}\bigr)$.
	Substituting this into \eqref{eq:pullback-operator}, we have
	\begin{equation}\label{eq:pullback-polar-identity}
		\begin{aligned}
			\cB_t[H]
			={}&\lambda\er\,\bigl(u_r[h]\circ\cS_t\bigr)\\
			&+\lambda\evarphi\Bigl(
			u_\varphi[h]\circ\cS_t
			-KtR\Omega'\bigl(u_r[h]\circ\cS_t\bigr)
			\Bigr).
		\end{aligned}
	\end{equation}
	We first estimate the scalar fields $u_\nu[h]\circ\cS_t$,
	$\nu\in\{r,\varphi\}$.
	
	Let $O_\beta$ denote rotation by $\beta$. Rotation covariance gives
	$u[F\circ O_\beta](y)=O_{-\beta}u[F](O_\beta y)$.
	Since $O_\beta e_\nu(\varphi)=e_\nu(\varphi+\beta)$, taking polar
	components and differentiating at $\beta=0$, we get
	\begin{equation}\label{eq:polar-angular-commutation}
		\partial_\varphi u_\nu[F]=u_\nu[\partial_\varphi F],
		\qquad \nu\in\{r,\varphi\}.
	\end{equation}
	The flow $\cS_t$ changes only the angular coordinate. Hence
	\eqref{eq:polar-angular-commutation} also gives
	\begin{equation}\label{eq:local-angular-cancellation}
		\partial_\varphi\bigl(u_\nu[h]\circ\cS_t\bigr)
		=u_\nu[(\partial_\varphi H)\circ\cS_t^{-1}]
		\circ\cS_t.
	\end{equation}
	
	We next differentiate in $R$. On $\mathcal A$, the chain rule gives
	\[
	(\partial_R+Kt\Omega'\partial_\varphi)h
	=(\partial_RH)\circ\cS_t^{-1}.
	\]
	Applying $\partial_R+Kt\Omega'\partial_\varphi$ to $u_\nu[h]$
	and commuting this differential operator with $u_\nu$, we obtain
	\begin{equation}\label{eq:pullback-radial-commutation}
		\partial_R\bigl(u_\nu[h]\circ\cS_t\bigr)
		=\Bigl(u_\nu[(\partial_RH)\circ\cS_t^{-1}]+[\partial_R+Kt\Omega'\partial_\varphi,u_\nu]h
		\Bigr)\circ\cS_t.
	\end{equation}
	The last commutator retains a factor of $Kt$, but it acts on $h$
	without differentiating it.
	
For functions supported in $\mathcal I$, we can instead use
	$\Omega'^{-1}\partial_R$. Indeed, $\Omega'$ is bounded away from zero on $\mathcal I$, and
	\[
	\Omega'^{-1}\partial_Rh
	=(\Omega'^{-1}\partial_RH)\circ\cS_t^{-1}
	-Kt\partial_\varphi h.
	\]
	Differentiating the outer composition contributes
	$Kt\partial_\varphi u_\nu[h]$. By
	\eqref{eq:polar-angular-commutation}, these two $Kt$ terms cancel.
	Thus
	\begin{equation}\label{eq:pullback-weighted-radial-commutation}
		\Omega'^{-1}\partial_R
		\bigl(u_\nu[h]\circ\cS_t\bigr)
		=\Bigl(u_\nu[
		(\Omega'^{-1}\partial_RH)\circ\cS_t^{-1}]+[\Omega'^{-1}\partial_R,u_\nu]h
		\Bigr)\circ\cS_t.
	\end{equation}

	We claim the following $C^\alpha$ bounds for the polar velocity
	components and their commutators. For $\nu\in\{r,\varphi\}$ and
	$F$ supported in $\mathcal A$,
	\begin{equation}\label{eq:pullback-first-commutator}
		\Norm{u_\nu[F]}_{C^\alpha(\mathcal A)}
		+\Norm{[\partial_R,u_\nu]F}_{C^\alpha(\mathcal A)}
		+\Norm{[\Omega'\partial_\varphi,u_\nu]F}
		_{C^\alpha(\mathcal A)}
		\leq C_{f,\alpha}\Norm F_{C^\alpha(\mathcal A)}.
	\end{equation}
	For $F$ supported in $\mathcal I$, we also claim
	\begin{equation}\label{eq:local-fixed-commutator}
		\Norm{u_\nu[F]}_{C^\alpha(\mathcal I)}
		+\Norm{[\Omega'^{-1}\partial_R,u_\nu]F}
		_{C^\alpha(\mathcal I)}
		\leq C_{f,\alpha}\Norm F_{C^\alpha(\mathcal I)}.
	\end{equation}
	
	Both bounds follow from the same kernel calculation. Since
	$\partial_R=\er\cdot\nabla$ and
	$\partial_\varphi=R\evarphi\cdot\nabla$, the relevant vector fields
	are $\er$ and $R\Omega'\evarphi$ on $\mathcal A$, and
	$\Omega'^{-1}\er$ on $\mathcal I$. Multiplying each by a fixed
	smooth radial cutoff equal to one near the corresponding closed
	annulus gives a smooth, compactly supported extension. For the
	last field, this uses $\inf_{\overline{\mathcal I}}\Omega'>0$.

	Fix one such extension $a$ and let
	$F\in C_c^{1,\alpha}(\R^2)$ be supported in the corresponding
	annulus. In the following kernel estimates, all output norms
	are taken on this annulus, while norms of $F$ are taken over
	$\R^2$. For $\epsilon>0$, set
	\[
	D_\epsilon(y):=\{y':|y-y'|>\epsilon\},\qquad
	u_{\nu,\epsilon}[F](y)
	:=e_\nu(y)\cdot\int_{D_\epsilon(y)}
	\cK(y-y')F(y')\,\dd y'.
	\]
	Integration by parts in $u_{\nu,\epsilon}[a\cdot\nabla F]$ gives
	\begin{align*}
		[a\cdot\nabla,u_{\nu,\epsilon}]F(y)
		={}&(a(y)\cdot\nabla e_\nu(y))\cdot
		\int_{D_\epsilon(y)}\cK(y-y')F(y')\,\dd y'\\
		&+\int_{D_\epsilon(y)}e_\nu(y)\cdot
		\Bigl\{(a(y)-a(y'))\cdot\nabla\cK(y-y')
		+(\operatorname{div}a)(y')\cK(y-y')\Bigr\}
		F(y')\,\dd y'\\
		&+B_{\nu,\epsilon}(y),
	\end{align*}
	where the difference of the two inner boundary terms is
	\[
	B_{\nu,\epsilon}(y)
	=\int_{|w|=\epsilon}e_\nu(y)\cdot\cK(w)
	\bigl(a(y)-a(y-w)\bigr)\cdot\frac w\epsilon
	F(y-w)\,\dd\sigma(w).
	\]
	There is no boundary term at infinity because $F$ and $a$ are
	compactly supported. Setting $w=\epsilon\omega$ and expanding $a$, we obtain
	\[
	B_{\nu,\epsilon}(y)
	=F(y)\int_{\mathbb S^1}e_\nu(y)\cdot\cK(\omega)
	\bigl(Da(y)\omega\cdot\omega\bigr)\,\dd\sigma(\omega)
	+O\!\left(\epsilon^\alpha[F]_{C^\alpha}
	+\epsilon\Norm F_\infty\right).
	\]
The integral vanishes because $\cK$ is odd and
$Da(y)\omega\cdot\omega$ is even. Thus
$B_{\nu,\epsilon}\to0$ uniformly on the observation annulus,
and the commutator satisfies
	\begin{equation}\label{eq:pullback-fixed-kernel-commutator}
		\begin{aligned}
			[a\cdot\nabla,u_\nu]F(y)
			={}&(a(y)\cdot\nabla e_\nu(y))\cdot u[F](y)\\
			&+\pv\int_{\R^2}e_\nu(y)\cdot
			\Bigl((a(y)-a(y'))\cdot\nabla\cK(y-y')\Bigr)
			F(y')\,\dd y'\\
			&+\pv\int_{\R^2}e_\nu(y)\cdot
			\Bigl((\operatorname{div}a)(y')\cK(y-y')\Bigr)
			F(y')\,\dd y'.
		\end{aligned}
	\end{equation}
	The first term vanishes for the radial choices
	$a=\er$ on $\mathcal A$ and
	$a=\Omega'^{-1}\er$ on $\mathcal I$.
	For $a=R\Omega'\evarphi$ on $\mathcal A$, it equals
	$\Omega'u_\varphi[F]$ when $\nu=r$ and
	$-\Omega'u_r[F]$ when $\nu=\varphi$.
	
	We estimate the two principal value integrals in
	\eqref{eq:pullback-fixed-kernel-commutator} together. Set
	$w=y-y'$ and denote their combined kernel by
	\[
	\mathcal L_{\nu,a}(y,w)
	:=e_\nu(y)\cdot\Bigl(
	(a(y)-a(y-w))\cdot\nabla\cK(w)
	+(\operatorname{div}a)(y-w)\cK(w)
	\Bigr).
	\]
	Taylor expansion at $y$ gives, for $0<|w|<1$,
	\[
	\mathcal L_{\nu,a}(y,w)
	=\sum_{i,j,m=1}^2
	e_{\nu,m}(y)\partial_j a_i(y)\,
	w_j\partial_i\cK_m(w)+(\operatorname{div}a)(y)
	\sum_{m=1}^2e_{\nu,m}(y)\cK_m(w)
	+E_{\nu,a}(y,w),
	\]
	where
	\[
	|E_{\nu,a}(y,w)|\leq C_f|w|^{-1},
	\qquad
	|\nabla_yE_{\nu,a}(y,w)|+|\nabla_wE_{\nu,a}(y,w)|
	\leq C_f|w|^{-2}.
	\]
	Choose a fixed radial cutoff $\eta$ supported in $B_1$ and equal
	to one near the origin. For
	$k\in\{\cK_m,w_j\partial_i\cK_m: i,j,m\in\{1,2\}\}$, set
	\[
	T_kF(y):=\pv\int\eta(w)k(w)F(y-w)\,\dd w.
	\]
	Thus \eqref{eq:pullback-fixed-kernel-commutator} becomes
	\[
	\begin{aligned}
		[a\cdot\nabla,u_\nu]F(y)
		={}&(a(y)\cdot\nabla e_\nu(y))\cdot u[F](y)+\sum_{i,j,m=1}^2
		e_{\nu,m}(y)\partial_j a_i(y)\,
		T_{w_j\partial_i\cK_m}F(y)\\
		&+(\operatorname{div}a)(y)
		\sum_{m=1}^2e_{\nu,m}(y)T_{\cK_m}F(y)
		+\mathcal E_{\nu,a}F(y),
	\end{aligned}
	\]
	where
	\[
	\mathcal E_{\nu,a}F(y)
	:=\int_{\R^2}
	\Bigl[\eta(w)E_{\nu,a}(y,w)
	+(1-\eta(w))\mathcal L_{\nu,a}(y,w)\Bigr]
	F(y-w)\,\dd w.
	\]

	Each $k$ is odd and homogeneous of degree $-2$. Its zero angular
	mean gives
	\[
	T_kF(y)
	=\int\eta(w)k(w)\bigl(F(y-w)-F(y)\bigr)\,\dd w,
	\qquad
	\Norm{T_kF}_\infty\leq C_\alpha\Norm F_{C^\alpha}.
	\]

	To estimate the H\"older seminorm, let $d=|y-\bar y|<1/6$
	and split each integral at radius $3d$ about its center.
	For each $x\in\{y,\bar y\}$, we have
	\[
	\left|\int_{|w|<3d}\eta(w)k(w)
	\bigl(F(x-w)-F(x)\bigr)\,\dd w\right|
	\leq C_\alpha d^\alpha[F]_{C^\alpha}.
	\]
Because both the cutoff $\mathbf 1_{\{|w|\geq 3d\}}$ and $\eta(w)$
are radial, while $k$ has zero angular mean, the integral of the
truncated kernel against a constant vanishes. Hence, by comparing the
two truncated kernels and using the pointwise bounds for $k$, $\nabla k$,
$\eta$ and $\nabla\eta$, we obtain
\[
\begin{aligned}
	&\Bigg|
	\int_{|y-z|\geq3d}\eta(y-z)k(y-z)F(z)\,\dd z
	-\int_{|\bar y-z|\geq3d}\eta(\bar y-z)k(\bar y-z)F(z)\,\dd z
	\Bigg|\\
	&=
	\Bigg|
	\int_{|y-z|\geq3d}\eta(y-z)k(y-z)
	\bigl(F(z)-F(y)\bigr)\,\dd z
	-\int_{|\bar y-z|\geq3d}\eta(\bar y-z)k(\bar y-z)
	\bigl(F(z)-F(y)\bigr)\,\dd z
	\Bigg|\\
	&\leq
	C[F]_{C^\alpha}
	\left(
	d\int_{3d}^1 r^{\alpha-2}\,\dd r+d^\alpha
	\right)
	+Cd\Norm F_{C^\alpha}.
\end{aligned}
\]
	Since $
	d\int_{3d}^1r^{\alpha-2}\,\dd r
	\leq C_\alpha d^\alpha,$
	the near and far bounds give
	\[
	|T_kF(y)-T_kF(\bar y)|
	\leq C_\alpha d^\alpha\Norm F_{C^\alpha}.
	\]
		The size and derivative bounds
	for the local kernel of $\mathcal E_{\nu,a}$, together with
	the smoothness of its far kernel, give
	\[
	\begin{aligned}
		\Norm{\mathcal E_{\nu,a}F}_\infty
		&\leq C_f\Norm F_\infty,\\
		|\mathcal E_{\nu,a}F(y)-\mathcal E_{\nu,a}F(\bar y)|
		&\leq C_f\Norm F_\infty
		\left(
		\int_0^{4d}r^{-1}r\,\dd r
		+d\int_{4d}^1r^{-2}r\,\dd r+d
		\right)\\
		&\leq C_{f,\alpha}d^\alpha\Norm F_\infty.
	\end{aligned}
	\]
	For $d\geq1/6$, the supremum bound gives the same H\"older
	estimate. Moreover, for each Cartesian component,
	\[
	(u[F])_m(y)
	=T_{\cK_m}F(y)
	+\int_{\R^2}(1-\eta(w))\cK_m(w)F(y-w)\,\dd w.
	\]
	The second kernel is smooth away from $w=0$. Thus the estimate
	for $T_kF$, the bound above for $\mathcal E_{\nu,a}F$, and the
	commutator decomposition give
	\[
	\Norm{[a\cdot\nabla,u_\nu]F}_{C^\alpha}
	\leq C_{f,\alpha}
	\left(
	\Norm{u[F]}_{C^\alpha}
	+\sum_k\Norm{T_kF}_{C^\alpha}
	+\Norm{\mathcal E_{\nu,a}F}_{C^\alpha}
	\right)
	\leq C_{f,\alpha}\Norm F_{C^\alpha}.
	\]
	Here $F$ is supported in the relevant annulus, so its zero
	extension has global $C^\alpha$ norm controlled by its annular
	norm. This proves \eqref{eq:pullback-first-commutator} and
	\eqref{eq:local-fixed-commutator}.
	
	By \eqref{eq:background-flow-Linfty}, composition with either
	$\cS_t$ or $\cS_t^{-1}$ satisfies
	\[
	\Norm{G\circ\cS_t^{\pm1}}_{C^\alpha}
	\leq C_{f,\alpha}(1+Kt)^\alpha\Norm G_{C^\alpha}
	\]
	on each invariant annulus. When $H$ is supported in that
	annulus, \eqref{eq:local-angular-cancellation} and the fixed
	velocity bound yield
	\[
	\Norm{u_\nu[h]\circ\cS_t}_{C^\alpha}
	+\Norm{\partial_\varphi(u_\nu[h]\circ\cS_t)}_{C^\alpha}
	\leq C_{f,\alpha}(1+Kt)^{2\alpha}
	\Norm H_{C^{1,\alpha}}.
	\]
For the radial derivative,
\eqref{eq:pullback-radial-commutation} and
\eqref{eq:pullback-first-commutator} lead to
	\[
	\begin{aligned}
		\Norm{\partial_R(u_\nu[h]\circ\cS_t)}
		_{C^\alpha(\mathcal A)}
		&\leq C_{f,\alpha}(1+Kt)^\alpha
		\left(
		\Norm{(\partial_RH)\circ\cS_t^{-1}}_{C^\alpha(\mathcal A)}
		+(1+Kt)\Norm h_{C^\alpha(\mathcal A)}
		\right)\\
		&\leq C_{f,\alpha}(1+Kt)^{2\alpha}
		\left(
		\Norm H_{C^{1,\alpha}(\mathcal A)}
		+Kt\Norm H_{C^\alpha(\mathcal A)}
		\right).
	\end{aligned}
	\]
	If $\supp H\subset\mathcal I$, then
	\eqref{eq:pullback-weighted-radial-commutation} and
	\eqref{eq:local-fixed-commutator} instead give
	\[
	\begin{aligned}
		\Norm{\partial_R(u_\nu[h]\circ\cS_t)}
		_{C^\alpha(\mathcal I)}
		&\leq C_{f,\alpha}(1+Kt)^\alpha
		\left(
		\Norm{(\Omega'^{-1}\partial_RH)\circ\cS_t^{-1}}
		_{C^\alpha(\mathcal I)}
		+\Norm h_{C^\alpha(\mathcal I)}
		\right)\\
		&\leq C_{f,\alpha}(1+Kt)^{2\alpha}
		\Norm H_{C^{1,\alpha}(\mathcal I)}.
	\end{aligned}
	\]
	Since
	$\nabla=\er\partial_R+R^{-1}\evarphi\partial_\varphi$
	on the fixed annuli, we obtain
	\begin{align}
		\Norm{u_\nu[h]\circ\cS_t}_{C^{1,\alpha}(\mathcal A)}
		&\leq C_{f,\alpha}(1+Kt)^{2\alpha}
		\left(
		\Norm H_{C^{1,\alpha}(\mathcal A)}
		+Kt\Norm H_{C^\alpha(\mathcal A)}
		\right),
		\label{eq:polar-components-pullback-full}\\
		\Norm{u_\nu[h]\circ\cS_t}_{C^{1,\alpha}(\mathcal I)}
		&\leq C_{f,\alpha}(1+Kt)^{2\alpha}
		\Norm H_{C^{1,\alpha}(\mathcal I)},
		\qquad \supp H\subset\mathcal I.
		\notag
	\end{align}
	
	The coefficients in \eqref{eq:pullback-polar-identity} satisfy,
	on either annulus,
	\[
	\Norm{\er}_{C^{1,\alpha}}
	+\Norm{\evarphi}_{C^{1,\alpha}}
	+Kt\Norm{R\Omega'\evarphi}_{C^{1,\alpha}}
	\leq C_{f,\alpha}(1+Kt).
	\]
	The product estimate therefore gives
	\[
	\begin{aligned}
		\Norm{\cB_t[H]}_{C^{1,\alpha}(\mathcal A)}
		&\leq C_{f,\alpha}\lambda(1+Kt)^{1+2\alpha}
		\left(
		\Norm H_{C^{1,\alpha}(\mathcal A)}
		+Kt\Norm H_{C^\alpha(\mathcal A)}
		\right)\\
		&\leq C_{f,\alpha}\lambda(1+Kt)^{2+2\alpha}
		\Norm H_{C^{1,\alpha}(\mathcal A)},\\
		\Norm{\cB_t[H]}_{C^{1,\alpha}(\mathcal I)}
		&\leq C_{f,\alpha}\lambda(1+Kt)^{1+2\alpha}
		\Norm H_{C^{1,\alpha}(\mathcal I)},
		\qquad \supp H\subset\mathcal I,
	\end{aligned}
	\]
which proves \eqref{eq:pullback-velocity-bound} and
	\eqref{eq:pullback-velocity-bound-inner}.
\end{proof}

The same commutators give integer Sobolev bounds. Composition with
$\cS_t^{\pm1}$ is isometric in $L^2$ on either annulus. 

\begin{lemma}
	\label{lem:pullback-H2}
	Let $t\geq0$, $m\in\{0,1,2\}$, and
	$H\in H^2(\R^2)$ satisfy $\supp H\subset\mathcal A$.
	Then
	\begin{equation}\label{eq:pullback-velocity-H2}
		\Norm{\cB_t[H]}_{H^m(\mathcal A)}
		\leq C_f\lambda(1+Kt)^{m+1}
		\Norm{H}_{H^m(\mathcal A)}.
	\end{equation}
	If $\supp H\subset\mathcal I$, then
	\begin{equation}\label{eq:pullback-velocity-H2-inner}
		\Norm{\cB_t[H]}_{H^m(\mathcal I)}
		\leq C_f\lambda(1+Kt)
		\Norm{H}_{H^m(\mathcal I)}.
	\end{equation}
	The constants are independent of $K,t,\lambda$, and $H$.
\end{lemma}

\begin{proof}
	By density, it suffices to consider smooth $H$ with the stated
	support. Set $h:=H\circ\cS_t^{-1}$.
	
	Let $a$ be one of the fixed extensions used in
	\cref{prop:pullback-velocity}, or a linear combination of the
	two extensions on $\mathcal A$ with uniformly bounded
	coefficients. Commuting
	\eqref{eq:pullback-fixed-kernel-commutator} once more gives
	kernels of the form
	\[
	w^\beta\partial^\mu\cK(w),
	\qquad w=y-y',
	\qquad |\beta|=|\mu|\leq2,
	\]
	with smooth bounded coefficients, and a remainder bounded by
	$C_f|w|^{-1}$ near zero. The principal kernels are odd and
	homogeneous of degree $-2$, so their principal-value operators
	are bounded on $L^2$ \cite{Stein}. 	The remainder kernel is bounded by $C_f|y-y'|^{-1}$
	near $y=y'$ and is bounded elsewhere on the fixed annuli.
	Consequently, its integrals in either variable are uniformly
	bounded, and the corresponding operator is bounded on $L^2$.
	Thus, for $\nu\in\{r,\varphi\}$,
	$\mathcal D\in\{\mathcal A,\mathcal I\}$, and
	$F\in C_c^\infty(\mathcal D)$,
	\[
	\begin{aligned}
		&\Norm{u_\nu[F]}_{L^2(\mathcal D)}
		+\Norm{[a\cdot\nabla,u_\nu]F}_{L^2(\mathcal D)}\\
		&\qquad
		+\Norm{[a\cdot\nabla,[a\cdot\nabla,u_\nu]]F}_{L^2(\mathcal D)}
		\leq C_f\Norm{F}_{L^2(\mathcal D)}.
	\end{aligned}
	\]

	On $\mathcal A$, take the extension of
	$a=(\er+KtR\Omega'\evarphi)/(1+Kt)$.
	The first and second commutators with
	$\partial_R+Kt\Omega'\partial_\varphi$ are therefore bounded
	on $L^2$ by $C_f(1+Kt)$ and $C_f(1+Kt)^2$, respectively.
	On $\mathcal I$, $\partial_\varphi$ commutes with both
	$u_\nu$ and $\Omega'^{-1}\partial_R$, by
	\eqref{eq:polar-angular-commutation}. Hence the first two
	commutators with
	$\Omega'^{-1}\partial_R+Kt\partial_\varphi$
	are independent of $Kt$.
	
	Iterating \eqref{eq:pullback-radial-commutation} and
	\eqref{eq:pullback-weighted-radial-commutation}, and using
	the measure-preserving property of $\cS_t$, we obtain,
	for integers $j,\ell\geq0$ with $j+\ell\leq2$,
	\[
	\begin{aligned}
		&\Norm{
			(\Omega'^{-1}\partial_R)^j\partial_\varphi^\ell
			(u_\nu[h]\circ\cS_t)
		}_{L^2(\mathcal I)}
		\leq C_f\sum_{k=0}^j
		\Norm{
			(\Omega'^{-1}\partial_R)^k\partial_\varphi^\ell H
		}_{L^2(\mathcal I)},\\
		&\Norm{
			\partial_R^j\partial_\varphi^\ell
			(u_\nu[h]\circ\cS_t)
		}_{L^2(\mathcal A)}
		\leq C_f\sum_{k=0}^j(1+Kt)^{j-k}
		\Norm{\partial_R^k\partial_\varphi^\ell H}_{L^2(\mathcal A)}.
	\end{aligned}
	\]
	Here the first estimate assumes $\supp H\subset\mathcal I$.
	Since polar and Cartesian Sobolev norms are equivalent on the
	fixed annuli and $\inf_{\overline{\mathcal I}}\Omega'>0$,
	these estimates give
	\begin{align}
		\Norm{u_\nu[h]\circ\cS_t}_{H^m(\mathcal I)}
		&\leq C_f\Norm{H}_{H^m(\mathcal I)},
		\label{eq:pullback-components-H2-inner}\\
		\Norm{u_\nu[h]\circ\cS_t}_{H^m(\mathcal A)}
		&\leq C_f\sum_{j=0}^m(1+Kt)^{m-j}
		\Norm{H}_{H^j(\mathcal A)}.
		\label{eq:pullback-components-H2-full}
	\end{align}
	
	The bounded derivatives of $R\Omega'$ and the polar basis
	vectors, together with \eqref{eq:pullback-polar-identity},
	yield
	\[
	\Norm{\cB_t[H]}_{H^m(\mathcal D)}
	\leq C_f\lambda(1+Kt)
	\sum_{\nu\in\{r,\varphi\}}
	\Norm{u_\nu[h]\circ\cS_t}_{H^m(\mathcal D)},
	\qquad
	\mathcal D\in\{\mathcal A,\mathcal I\}.
	\]
	Combining this with \eqref{eq:pullback-components-H2-inner}
	proves \eqref{eq:pullback-velocity-H2-inner}.
	On $\mathcal A$, \eqref{eq:pullback-components-H2-full} gives
	\begin{equation}\label{eq:pullback-velocity-H2-tame}
		\Norm{\cB_t[H]}_{H^m(\mathcal A)}
		\leq C_f\lambda(1+Kt)
		\sum_{j=0}^m(1+Kt)^{m-j}
		\Norm{H}_{H^j(\mathcal A)}.
	\end{equation}
	Since $\Norm{H}_{H^j(\mathcal A)}
	\leq\Norm{H}_{H^m(\mathcal A)}$ for $j\leq m$,
	this implies \eqref{eq:pullback-velocity-H2}.
\end{proof}

\subsection{The exact solution in the adapted coordinates}
\label{sec:block}

We now apply the coordinates introduced in \cref{sec:pullback} to the
exact solution $\theta$ of \eqref{eq:SQG} with the initial datum defined
in \eqref{eq:initial-datum}.

For each $\ell\in\{\rad,\osc\}$, let \(\theta_\ell\) be the solution of
\begin{equation*}
	\left(
	\partial_t+u[\theta]\cdot\nabla_x
	\right)\theta_\ell=0,
	\quad
	\theta_\ell(\cdot,0)=\theta_{\ell,0}.
\end{equation*}
Thus $\theta=\theta_{\rad}+\theta_{\osc}$. Since $u[\theta]$ is
divergence-free, transport preserves each component's $L^p$ norm for
$1\leq p\leq\infty$. By \cref{lem:Nfold}, both components also remain
$2\pi/N$-periodic in $\varphi$.

Rescale by \(\lambda\):
\[
\Theta(y,t)
:=
\theta\left(\frac{y}{\lambda},t\right),
\quad
\Theta_\ell(y,t)
:=
\theta_\ell\left(\frac{y}{\lambda},t\right),
\quad
\ell\in\{\rad,\osc\}.
\]
Define the deviation of the rescaled radial component from its initial
profile by
\[
	\Theta_{\rad,\err}(y,t)
:=
	\Theta_{\rad}(y,t)-c_0\lambda^{1-s}f(R).
\]
Then
\[
\Theta
=
c_0\lambda^{1-s}f(R)+\Theta_{\osc}+	\Theta_{\rad,\err}.
\]
The rescaled SQG equation is
\[
\partial_t\Theta+\lambda u[\Theta]\cdot\nabla_y\Theta=0.
\]
Using
\[
\lambda u[c_0\lambda^{1-s}f]=Ku[f],
\qquad
u[f]\cdot\nabla f=0,
\]
we obtain
\begin{align}
	\left(
	\partial_t+Ku[f]\cdot\nabla
	+\lambda u[\Theta_{\osc}+	\Theta_{\rad,\err}]\cdot\nabla
	\right)\Theta_{\osc}
	&=0,
	\label{eq:q-fixed-scale}\\
	\left(
	\partial_t+Ku[f]\cdot\nabla
	+\lambda u[\Theta_{\osc}+	\Theta_{\rad,\err}]\cdot\nabla
	\right)	\Theta_{\rad,\err}
	&=
	-Ku[\Theta_{\osc}+	\Theta_{\rad,\err}]\cdot\nabla f(R).
	\label{eq:W-fixed-scale}
\end{align}

We now compose $\Theta_{\osc}$ and $\Theta_{\rad,\err}$ with the
background shear flow $\cS_t$ and set
\[
\widetilde\Theta_{\osc}(y,t)
:=
\Theta_{\osc}(\cS_t(y),t),
\quad
\widetilde\Theta_{\rad,\err}(y,t)
:=
\Theta_{\rad,\err}(\cS_t(y),t).
\]
It follows that
\[
\Theta_{\osc}
=
\widetilde\Theta_{\osc}\circ\cS_t^{-1},
\quad
\Theta_{\rad,\err}
=
\widetilde\Theta_{\rad,\err}\circ\cS_t^{-1}.
\]
 Since $\cS_t$ preserves the radial
coordinate, we have $f\circ\cS_t=f$.

For any time-dependent scalar function $h$, the chain rule gives
\begin{align*}
	\partial_t(h\circ\cS_t)
	&=
	\left(
	\partial_t h+Ku[f]\cdot\nabla h
	\right)\circ\cS_t,\\
	\nabla(h\circ\cS_t)
	&=
	(D\cS_t)^{\mathsf T}(\nabla h)\circ\cS_t.
\end{align*}
Differentiating $f\circ\cS_t=f$ gives
\[
(D\cS_t)^{-\mathsf T}\nabla f=(\nabla f)\circ\cS_t.
\]
Applying these identities to \eqref{eq:q-fixed-scale} and
\eqref{eq:W-fixed-scale}, and using the definition of $\cB_t$, yields
\begin{align}
	\partial_t\widetilde\Theta_{\osc}
	+V\cdot\nabla\widetilde\Theta_{\osc}
	&=0,
	\label{eq:Q-equation}\\
	\partial_t\widetilde\Theta_{\rad,\err}
	+V\cdot\nabla\widetilde\Theta_{\rad,\err}
	&=
	-c_0\lambda^{1-s}V\cdot\nabla f(R),
	\label{eq:Wtilde-equation}\\
	V
	&:=
	\cB_t[\widetilde\Theta_{\osc}]
	+\cB_t[\widetilde\Theta_{\rad,\err}].
	\label{eq:V}
\end{align}
The background transport $Ku[f]\cdot\nabla$ has therefore been
absorbed into the coordinates. The remaining transport field $V$ is
generated only by $\widetilde\Theta_{\osc}$ and $\widetilde\Theta_{\rad,\err}$ through
the operator $\cB_t$.

\subsection{Bootstrap closure on a fixed interval}

In this subsection, we close the $C^{1,\alpha}$ and support
bootstrap using \cref{prop:pullback-velocity,lem:radial-velocity-decay}.
Fix a smooth radial cutoff $\chi\in C_c^\infty(\mathcal I)$ such that
\[
\chi(y)=1
\quad\text{for } \frac34\leq |y| \leq\frac94.
\]
Fix $\delta>0$ and choose $\alpha$ such that
\begin{equation}\label{eq:alpha-choice}
	0<\alpha<\frac{\delta(s-1)}{1+2s-s^2+\delta},
\end{equation}
and set
\begin{align}
	\Gamma
	&:=K(1+KT)^{1+2\alpha},
	\notag\\
	\ell\label{eq:ell-def}
	&:=1+\log(2+N(1+KT)).
\end{align}

Since $\cS_0=\operatorname{Id}$, the definition of $\theta_{\osc,0}$
gives the following bound for some $A=A(g,\alpha)\geq1$,
independent of $c_0$, $\lambda$ and $N$:
\begin{align}\label{initial-osc}
	\Norm{\widetilde\Theta_{\osc}(0)}_{C^1}
\leq
A c_0\lambda^{1-s}N^{1-s},
\quad
\Norm{\widetilde\Theta_{\osc}(0)}_{C^{1,\alpha}}
\leq
A c_0\lambda^{1-s}N^{1+\alpha-s}.
\end{align}
Moreover,
\[
\begin{gathered}
\widetilde {\Theta}_{\rad,\err}(0)={\Theta}_{\rad,\err}(0)=0,\\
\supp\Theta_{\osc}(0)
\subset\left\{y\in\R^2:1<|y|<2\right\},
\quad
\supp\Theta_{\rad}(0)
\subset\left\{y\in\R^2:4<|y|<5\right\}.
\end{gathered}
\]

Let $T_{\max}>0$ denote the maximal lifespan of $\theta$.
Let $\kappa>0$ be a common exponential decay rate in
\cref{lem:radial-velocity-decay} for all the annuli used below.
 Define
\[
\tau_*:=
\sup\Bigl\{
\tau\in(0,\min\{T,T_{\max}\}):
\eqref{eq:osc-bootstrap}-\eqref{eq:support-bootstrap}
\text{ hold for every }t\in[0,\tau]
\Bigr\},
\]
where
\begin{align}
	&\Norm{\widetilde\Theta_{\osc}(t)}_{C^1}
	\leq
	2A c_0\lambda^{1-s}N^{1-s},
	&
	\Norm{\widetilde\Theta_{\osc}(t)}_{C^{1,\alpha}}
&	\leq
	2A c_0\lambda^{1-s}N^{1+\alpha-s},
	\label{eq:osc-bootstrap}\\
	&\Norm{\widetilde {\Theta}_{\rad,\err}}_{C^{1,\alpha}}
	\leq
	A c_0\lambda^{1-s}N^{-s}e^{-\kappa N/2}.
	\label{eq:Wtilde-bootstrap}\\
		&\supp\Theta_{\osc}(t)
	\subset
	\left\{y\in\R^2:\tfrac34<|y|<\tfrac94\right\},
	&
	\supp\Theta_{\rad}(t)
	&\subset
	\left\{y\in\R^2:\tfrac72<|y|<\tfrac{11}{2}\right\},
	\label{eq:support-bootstrap}
\end{align}
The initial bounds are strict at $t=0$. Continuity therefore gives a
short interval on which all bootstrap conditions hold, so $\tau_*>0$.

\subsubsection{Oscillatory component}

We first improve the $C^1$ and $C^{1,\alpha}$ bounds for $\widetilde\Theta_{\osc}$ in
\eqref{eq:osc-bootstrap}. 
\begin{proposition}
	\label{prop:Q-bound}
	Let \(\varepsilon>0\) be sufficiently small, depending only on \(f,g,\alpha\). Assume that
	\begin{equation}\label{eq:Q-smallness}
		T\Gamma N^{1+\alpha-s}\leq\varepsilon,
		\qquad (1+KT)e^{-\kappa N/2}\leq1.
	\end{equation}
	Then, for any $0\leq t<\tau_*$,
	\begin{equation}\label{eq:Q-improved}
		\Norm{\widetilde\Theta_{\osc}(t)}_{C^1}
		\leq\frac32A c_0\lambda^{1-s}N^{1-s},
		\quad
		\Norm{\widetilde\Theta_{\osc}(t)}_{C^{1,\alpha}}
		\leq\frac32A c_0\lambda^{1-s}N^{1+\alpha-s}.
	\end{equation}
\end{proposition}
\begin{proof}
By \eqref{eq:support-bootstrap} and the choice of $\chi$,
$\chi\widetilde\Theta_{\osc}=\widetilde\Theta_{\osc}$;
hence \eqref{eq:Q-equation} holds with $\chi V$ in place of $V$.
Applying \eqref{eq:pullback-velocity-bound-inner} and
\eqref{eq:pullback-velocity-bound} to the oscillatory and
radial error terms of \eqref{eq:V}, respectively, gives
\begin{align*}
	\Norm{\chi V(t)}_{C^{1,\alpha}}
	&\leq C\lambda(1+Kt)^{1+2\alpha}
	\Norm{\widetilde\Theta_{\osc}(t)}_{C^{1,\alpha}(\mathcal I)}
	+C\lambda(1+Kt)^{2+2\alpha}
	\Norm{\widetilde\Theta_{\rad,\err}(t)}_{C^{1,\alpha}(\mathcal A)}\\
	&\leq C\Gamma\bigl(
	N^{1+\alpha-s}+(1+KT)N^{-s}e^{-\kappa N/2}\bigr)\\
	&\leq C\Gamma N^{1+\alpha-s}=:B,
\end{align*}
where the last inequality uses
$N\geq C_*\log(2+KT)$ in \eqref{eq:Q-smallness} to absorb the
exponentially small term. The other condition in
\eqref{eq:Q-smallness} then gives
$
TB\leq C\varepsilon.
$
	The $C^1$ and $C^{1,\alpha}$ a priori estimates for the transport
	equation \eqref{eq:Q-equation} yield
	\begin{align*}
		\frac{\dd}{\dd t}
		\Norm{\widetilde\Theta_{\osc}(t)}_{C^1}
		&\leq
		C\Norm{\nabla(\chi V(t))}_\infty
		\Norm{\widetilde\Theta_{\osc}(t)}_{C^1},\\
		\frac{\dd}{\dd t}
		\Norm{\widetilde\Theta_{\osc}(t)}_{C^{1,\alpha}}
		&\leq
		C_\alpha\Norm{\chi V(t)}_{C^{1,\alpha}}
		\Norm{\widetilde\Theta_{\osc}(t)}_{C^{1,\alpha}}.
	\end{align*}
By \eqref{initial-osc} and Gronwall's inequality, we obtain
	\begin{align*}
		\Norm{\widetilde\Theta_{\osc}(t)}_{C^1}
		&\leq
		A c_0\lambda^{1-s}N^{1-s}\exp(CBt),\\
		\Norm{\widetilde\Theta_{\osc}(t)}_{C^{1,\alpha}}
		&\leq
		A c_0\lambda^{1-s}N^{1+\alpha-s}\exp(C_\alpha Bt).
	\end{align*}
	Choose $\varepsilon$ sufficiently small that
	\[
	\exp(C\varepsilon)\leq\frac32,
	\quad
	\exp(C_\alpha\varepsilon)\leq\frac32.
	\]
	This proves \eqref{eq:Q-improved}.
\end{proof}

\subsubsection{Radial deformation}

We now improve the estimate for $\widetilde\Theta_{\rad,\err}$ in \eqref{eq:Wtilde-bootstrap}. 

\begin{proposition}\label{prop:Wtilde-bound}
	Assume that \eqref{eq:Q-smallness} holds with $\varepsilon>0$ chosen
	as in \cref{prop:Q-bound}. Suppose in addition that
	\begin{equation}\label{eq:Gamma-small}
		KT\leq\varepsilon N,
		\qquad
		N\geq C_0\log(2+K(1+T)),
	\end{equation}
	where $C_0>0$ is sufficiently large and depends only on
	$f,g,\alpha$.
	Then, for any $0\leq t<\tau_*$,
	\begin{equation}\label{eq:Wtilde-improved}
		\Norm{\widetilde {\Theta}_{\rad,\err}}_{C^{1,\alpha}}
		\leq A
		c_0\lambda^{1-s}N^{-s}e^{-3\kappa N/4}.
	\end{equation}
\end{proposition}

\begin{proof}
	Set
	\[
	Z(t):=\Norm{\widetilde\Theta_{\rad,\err}(t)}_{C^{1,\alpha}}.
	\]
	The bootstrap assumptions, \eqref{eq:V}, and
	\cref{prop:pullback-velocity} give
	\[
	\begin{aligned}
		\Norm{V(t)}_{C^{1,\alpha}(\mathcal A)}
		&\le C\lambda(1+Kt)^{2+2\alpha}
		\bigl(
		\Norm{\widetilde\Theta_{\osc}(t)}_{C^{1,\alpha}}
		+\Norm{\widetilde\Theta_{\rad,\err}(t)}_{C^{1,\alpha}}
		\bigr)\\
		&\le CK(1+Kt)^{2+2\alpha}
		\bigl(N^{1+\alpha-s}+N^{-s}e^{-\kappa N/2}\bigr).
	\end{aligned}
	\]
	Since \(N^{-s}e^{-\kappa N/2}\le N^{1+\alpha-s}\),
	\eqref{eq:Q-smallness} and \(KT\le\varepsilon N\) imply
	\[
	\int_0^t\Norm{V(\xi)}_{C^{1,\alpha}(\mathcal A)}\,\dd\xi
	\le C(1+KT)T\Gamma N^{1+\alpha-s}
	\le C\varepsilon(1+\varepsilon N).
	\]
	
	We estimate the radial velocity in the source of
	\eqref{eq:Wtilde-equation} separately. The support bootstrap
	separates \(\Theta_{\osc}(t)\) from \(\{3<R<6\}\), while
	\cref{lem:Nfold} preserves its \(N\)-fold symmetry. Conservation
	of its \(L^\infty\) norm and
	\cref{lem:radial-velocity-decay} yield
	\[
	\begin{aligned}
		\Norm{u_r[\Theta_{\osc}(t)]}_{C^{1,\alpha}(\{3<R<6\})}
		&\le C(1+N)^3e^{-\kappa N}
		\Norm{\Theta_{\osc}(t)}_\infty\\
		&\le Cc_0\lambda^{1-s}N^{-s}
		(1+N)^3e^{-\kappa N}.
	\end{aligned}
	\]
	The radial identity in \eqref{eq:pullback-polar-identity}
	and composition with \(\cS_t\) therefore give
	\[
	\Norm{(\cB_t[\widetilde\Theta_{\osc}])_r}
	_{C^{1,\alpha}(\{3<R<6\})}
	\le CK(1+Kt)^{1+\alpha}N^{-s}
	(1+N)^3e^{-\kappa N}.
	\]
	
	For the radial error, the same identity and
	\eqref{eq:polar-components-pullback-full} give
	\[
	\begin{aligned}
		\Norm{(\cB_t[\widetilde\Theta_{\rad,\err}])_r}
		_{C^{1,\alpha}(\mathcal A)}
		&\le C\lambda(1+Kt)^{2\alpha}
		\bigl(
		\Norm{\widetilde\Theta_{\rad,\err}}_{C^{1,\alpha}(\mathcal A)}
		+Kt\Norm{\widetilde\Theta_{\rad,\err}}_{C^\alpha(\mathcal A)}
		\bigr)\\
		&\le C\lambda(1+Kt)^{1+2\alpha}Z(t).
	\end{aligned}
	\]
	The radial identity is essential here because the tangential
	$KtR\Omega'$ term in \eqref{eq:pullback-polar-identity} does not
	enter the estimate.
	
	Since \(\nabla f(R)=f'(R)\er\), the source term is
	\[
	-c_0\lambda^{1-s}f'(R)
	\bigl(
	(\cB_t[\widetilde\Theta_{\osc}])_r
	+(\cB_t[\widetilde\Theta_{\rad,\err}])_r
	\bigr).
	\]
	The fixed support of \(f'\) gives
	\[
	\Norm{f'V_r}_{C^{1,\alpha}}
	\le C_f\left(
	\Norm{(\cB_t[\widetilde\Theta_{\osc}])_r}
	_{C^{1,\alpha}(\{3<R<6\})}
	+\Norm{(\cB_t[\widetilde\Theta_{\rad,\err}])_r}
	_{C^{1,\alpha}(\mathcal A)}
	\right).
	\]
	Moreover, \eqref{eq:support-bootstrap} keeps
	\(\widetilde\Theta_{\rad,\err}\) in a fixed compact
	subannulus of \(\mathcal A\). A fixed cutoff there reduces
	the transport estimate to the norm of \(V\) on \(\mathcal A\).
	Consequently, for every \(t<\tau_*\),
	\[
	\begin{aligned}
		Z'(t)
		&\le C\Norm{V(t)}_{C^{1,\alpha}(\mathcal A)}Z(t)
		+Cc_0\lambda^{1-s}\Norm{f'V_r}_{C^{1,\alpha}}\\
		&\le C\bigl(
		\Norm{V(t)}_{C^{1,\alpha}(\mathcal A)}
		+K(1+Kt)^{1+2\alpha}\bigr)Z(t)\\
		&\quad+Cc_0\lambda^{1-s}K(1+Kt)^{1+\alpha}
		N^{-s}(1+N)^3e^{-\kappa N}
	\end{aligned}
	\]
	with $Z(0)=0.$
	The first condition in \eqref{eq:Q-smallness} also gives
	\[
	T\Gamma\le\varepsilon N^{s-1-\alpha}=o(N),
	\qquad 0<s-1-\alpha<1.
	\]
	Gronwall's inequality now yields
	\[
	Z(t)\le Cc_0\lambda^{1-s}N^{-s}
	TK(1+KT)^{1+\alpha}(1+N)^3
	\exp\!\left(
	-\kappa N+C\varepsilon(1+\varepsilon N)
	+C\varepsilon N^{s-1-\alpha}
	\right).
	\]
Choosing \(\varepsilon\)
	sufficiently small and then \(N\) sufficiently large gives
	\eqref{eq:Wtilde-improved}.
\end{proof}

\subsubsection{Support separation}

We now improve the support inclusions in
\eqref{eq:support-bootstrap} under
\eqref{eq:osc-bootstrap}--\eqref{eq:support-bootstrap}. 

\begin{proposition}
	\label{prop:support}
Assume the hypotheses of
\cref{prop:Q-bound,prop:Wtilde-bound}. Suppose in addition that
	\begin{equation}\label{eq:support-smallness}
		T\left(KN^{-s}\ell+Ke^{-\kappa N/2}\right)
		\leq\varepsilon.
	\end{equation}
	Then, for any $0\leq t<\tau_*$,
	\begin{equation}\label{eq:support-improved}
		\supp\Theta_{\osc}(t)
		\subset
		\left\{\tfrac78<|y|<\tfrac{17}{8}\right\},
		\quad
		\supp\Theta_{\rad}(t)
		\subset
		\left\{\tfrac{31}{8}<|y|<\tfrac{41}{8}\right\}.
	\end{equation}
\end{proposition}
\begin{proof}
For \(y_0\in\R^2\), let \(Y(t;y_0)\) denote the characteristic curve defined by
	\[
	\partial_tY(t;y_0)
	=
	\lambda u[\Theta](Y(t;y_0),t),
	\quad
	Y(0;y_0)=y_0.
	\]
Since the flow transports each component, we have
	\[
	Y\bigl(t;\supp\Theta_\ell(0)\bigr)
	=
	\supp\Theta_\ell(t),
	\quad
	\ell\in\{\rad,\osc\}.
	\]
	Set $\rho(t;y_0):=|Y(t;y_0)|$. Since $f$ is radial, the radial component of \(u[c_0\lambda^{1-s}f]\) vanishes, so \(u_r[\Theta_{\rad}]=u_r[	\Theta_{\rad,\err}]\). Thus
	\begin{equation}\label{eq:radial-trajectory}
		\partial_t\rho(t;y_0)
		=
		\lambda
		\bigl(
		u_r[\Theta_{\osc}]+u_r[	\Theta_{\rad,\err}]
		\bigr)(Y(t;y_0),t).
	\end{equation}
	
	Exact transport and $g\not\equiv0$ give
	\[
	\|\Theta_{\osc}(t)\|_\infty
	=
	\|\Theta_{\osc}(0)\|_\infty
	\simeq_g c_0\lambda^{1-s}N^{-s}.
	\]
	From \eqref{eq:Q-improved} and \eqref{eq:background-shear-flow}, we also have
	\[
	\|\nabla\Theta_{\osc}(t)\|_\infty\le Cc_0\lambda^{1-s}N^{1-s}(1+Kt).
	\]
	Together with the support bound \eqref{eq:support-bootstrap}, \cref{lem:log-CZ} gives
	\[
	\lambda\Norm{u[\Theta_{\osc}(t)]}_\infty
	\leq
	CKN^{-s}
	\left[1+\log\bigl(2+N(1+Kt)\bigr)\right]
	\leq
	CKN^{-s}\ell.
	\]
	Similarly, \eqref{eq:Wtilde-improved}, \cref{lem:composition}, and the boundedness of the Riesz transforms on \(C^{1,\alpha}\) yield
	\[
	\lambda\Norm{u[{\Theta}_{\rad,\err}(t)]}_\infty
	\leq
	CKN^{-s}(1+Kt)^{1+\alpha}e^{-3\kappa N/4}
	\leq
	CKe^{-\kappa N/2},
	\]
where the last step uses
\(N\geq C_0\log(2+K(1+T))\).

Integrating \eqref{eq:radial-trajectory} yields, for \(0\leq t<\tau_*\),
	\[
	\left|\rho(t;y_0)-\rho(0;y_0)\right|
	\leq
	Ct\left(
	KN^{-s}\ell+Ke^{-\kappa N/2}
	\right),
	\quad
	y_0\in
	\supp\Theta_{\osc}(0)\cup\supp\Theta_{\rad}(0),
	\]
By \eqref{eq:support-smallness}, this displacement is less than \(1/8\) (after taking \(\varepsilon\) smaller if necessary). Since
\[
\operatorname{supp}\Theta_{\osc}(0)\subset\{1<|y|<2\},\quad
\operatorname{supp}\Theta_{\rad}(0)\subset\{4<|y|<5\},
\]
the improved support inclusions in \eqref{eq:support-improved} follow.
\end{proof}

\subsubsection{Continuation}
The bounds \eqref{eq:Q-improved}, \eqref{eq:Wtilde-improved}, and
\eqref{eq:support-improved} strictly improve the bootstrap assumptions
under the hypotheses of
\cref{prop:Q-bound,prop:Wtilde-bound,prop:support}. We verify these
conditions in \cref{thm:fixed-block} and then apply the continuation
criterion to obtain a smooth solution on $[0,T]$.

\begin{theorem}[Closure of the bootstrap]
	\label{thm:fixed-block}
	Let \(1<s<2\), \(T>0\), \(c_0>0\), and \(\delta>0\).
	Choose \(\alpha\) according to \eqref{eq:alpha-choice} and set
	\[
	a_\delta:=\frac{2(2-s)+\delta}{s-1}.
	\]
	Choose \(\varepsilon>0\) sufficiently small for
	\cref{prop:Q-bound,prop:Wtilde-bound,prop:support}. With \(K\) and
	\(N\) as in \eqref{eq:K-definition} and
	\eqref{eq:idea-parameter-choice}, namely,
	\[
	K:=c_0\lambda^{2-s},\quad N:=\lfloor\lambda^{a_\delta}\rfloor,
	\]
	for all sufficiently large \(\lambda\), the following inequalities hold:
	\begin{align}
		TK(1+KT)^{1+2\alpha}N^{1+\alpha-s}
		&\leq\varepsilon,
		\label{eq:block-cond1}\\
		TK
		&\leq\varepsilon N,
		\label{eq:block-cond2}\\
		T\left(
		KN^{-s}\ell+Ke^{-\kappa N/2}
		\right)
		&\leq\varepsilon.
		\label{eq:block-cond3}
	\end{align}
For every such $\lambda$, the solution with initial datum
\eqref{eq:initial-datum} extends smoothly to $[0,T]$, and the estimates
\eqref{eq:Q-improved}, \eqref{eq:Wtilde-improved} and
\eqref{eq:support-improved} hold
throughout this interval.
\end{theorem}

\begin{proof}
First, from \eqref{eq:alpha-choice}, we have
	\begin{equation}\label{eq:adelta-bootstrap}
		a_\delta
		>
		\frac{(2-s)(2+2\alpha)}{s-1-\alpha}.
	\end{equation}

With \(K=c_0\lambda^{2-s}\) and
\(N=\lfloor\lambda^{a_\delta}\rfloor\), for fixed \(c_0\) and \(T\),
\eqref{eq:ell-def} gives
\(\ell\lesssim_{c_0,T}1+\log\lambda\) for large $\lambda$.
Substitution into the left-hand sides of
\eqref{eq:block-cond1}--\eqref{eq:block-cond3} yields
	\begin{equation}\label{eq:block-parameter-bounds}
		\begin{aligned}
			TK(1+KT)^{1+2\alpha}N^{1+\alpha-s}
			&\lesssim_{c_0,T}
			\lambda^{(2-s)(2+2\alpha)-a_\delta(s-1-\alpha)},\\
			\frac{TK}{N}
			&\lesssim_{c_0,T}
			\lambda^{2-s-a_\delta},\\
			TKN^{-s}\ell
			&\lesssim_{c_0,T}
			\lambda^{2-s-a_\delta s}\log\lambda,\\
			TKe^{-\kappa N/2}
			&\lesssim_{c_0,T}
			\lambda^{2-s}e^{-c\lambda^{a_\delta}}.
		\end{aligned}
	\end{equation}
	Since \(0<s-1-\alpha<1\) and
	\(s(2+2\alpha)>1\), \eqref{eq:adelta-bootstrap} implies
	\[
	(2-s)(2+2\alpha)-a_\delta(s-1-\alpha)<0,
	\quad
	2-s-a_\delta<0,
	\quad
	2-s-a_\delta s<0.
	\]
Therefore, all right-hand sides in \eqref{eq:block-parameter-bounds}
tend to zero as $\lambda\to\infty$. For sufficiently large $\lambda$,
conditions \eqref{eq:block-cond1}--\eqref{eq:block-cond3} hold.
The logarithmic conditions in \eqref{eq:Q-smallness} and
\eqref{eq:Gamma-small} also hold, since
$N\simeq\lambda^{a_\delta}$ with $a_\delta>0$. Hence
\cref{prop:Q-bound,prop:Wtilde-bound,prop:support} hold on
$[0,\tau_*)$.

The preceding three propositions give strict improvements of
\eqref{eq:osc-bootstrap}--\eqref{eq:support-bootstrap}
on $[0,\tau_*)$. If $\tau_*<\min\{T,T_{\max}\}$, the exact
solution exists beyond $\tau_*$. 
The strict bounds and continuity then ensure that all bootstrap assumptions hold
on $[0,\tau_*+h]$ for some $h>0$, contradicting the definition
of $\tau_*$. Hence
\[
\tau_*=\min\{T,T_{\max}\}.
\]	

The stationary radial part satisfies, for $\lambda\geq1$,
\begin{equation*}
	\Norm{\overline\theta_{\rad}}_{C_x^{1,\alpha}}
	=
	\Norm{\frac{K}{\lambda}f(\lambda\cdot)}_{C_x^{1,\alpha}}
	\leq C_{f,\alpha}K\lambda^\alpha.
\end{equation*}
The other two components in the $y$ variable are
$\widetilde\Theta_{\osc}\circ\cS_t^{-1}$ and
$\widetilde\Theta_{\rad,\err}\circ\cS_t^{-1}$.
Returning to $x$ and using \cref{lem:composition},
\eqref{eq:Q-improved} and \eqref{eq:Wtilde-improved}, we obtain
\begin{equation*}
	\begin{aligned}
		\Norm{\theta(t)}_{C_x^{1,\alpha}}
		&\leq \Norm{\overline\theta_{\rad}}_{C_x^{1,\alpha}}
		+C_{f,\alpha}\lambda^{1+\alpha}(1+Kt)^{1+\alpha}
		\Bigl(
		\Norm{\widetilde\Theta_{\osc}(t)}_{C_y^{1,\alpha}}
		+\Norm{\widetilde\Theta_{\rad,\err}(t)}_{C_y^{1,\alpha}}
		\Bigr)\\
		&\leq C_{f,g,\alpha}K\lambda^\alpha
		\Bigl[1+(1+Kt)^{1+\alpha}
		\bigl(N^{1+\alpha-s}+N^{-s}e^{-3\kappa N/4}\bigr)\Bigr].
	\end{aligned}
\end{equation*}
This bound is uniform for $0\leq t<\min\{T,T_{\max}\}$.
Moreover, $\Norm{\theta(t)}_2=\Norm{\theta_0}_2$ by the
divergence-free transport equation. The continuation criterion in
the local $C^{1,\alpha}\cap L^2$ theory \cite{WuHolder} therefore
rules out $T_{\max}\leq T$. Hence
$T_{\max}>T$, and so $\tau_*=T$. This completes the proof.
\end{proof}

The bootstrap gives smooth existence and support separation on
$[0,T]$. We next estimate the $L^2$ difference between the exact and
approximate solutions, which will transfer the lower bound to the
exact solution.

\subsection{Error estimates in \texorpdfstring{$L^2$}{L2}}

We estimate the radial and oscillatory errors separately. Define
\[
\theta_{\rad,\mathrm{err}}
:=
\theta_{\rad}-\overline\theta_{\rad},
\quad
\theta_{\osc,\mathrm{err}}
:=
\theta_{\osc}-\overline\theta_{\osc}.
\]
Both the exact and approximate component equations preserve $L^p$
norms. In particular,
\begin{equation}\label{eq:component-L2-exact}
	\Norm{\theta_{\rad}(t)}_2
	=
	\Norm{\overline\theta_{\rad}(t)}_2
	=
	\Norm{\theta_{\rad,0}}_2,
	\quad
	\Norm{\theta_{\osc}(t)}_2
	=
	\Norm{\overline\theta_{\osc}(t)}_2
	=
	\Norm{\theta_{\osc,0}}_2.
\end{equation}

Returning to the variable \(x\),
we combine \eqref{eq:explicit-approximate} with \eqref{eq:support-improved} to obtain
\begin{align*}
	\supp\overline\theta_{\osc}(t)
	&\subset
	\left\{
	\lambda^{-1}<|x|<2\lambda^{-1}
	\right\},
	&
	\supp\theta_{\osc}(t)
	&\subset
	\left\{
	\tfrac78\lambda^{-1}<|x|<
	\tfrac{17}{8}\lambda^{-1}
	\right\},\\
	\supp\overline\theta_{\rad}(t)
	&\subset
	\left\{
	4\lambda^{-1}<|x|<5\lambda^{-1}
	\right\},
	&
	\supp\theta_{\rad}(t)
	&\subset
	\left\{
	\tfrac{31}{8}\lambda^{-1}<|x|<
	\tfrac{41}{8}\lambda^{-1}
	\right\}.
\end{align*}
Consequently,
\[
\supp\theta_{\osc,\mathrm{err}}(t)
\subset
\left\{
\tfrac78\lambda^{-1}<|x|<
\tfrac{17}{8}\lambda^{-1}
\right\},
\quad
\supp\theta_{\rad,\mathrm{err}}(t)
\subset
\left\{
\tfrac{31}{8}\lambda^{-1}<|x|<
\tfrac{41}{8}\lambda^{-1}
\right\}.
\]
Since $\frac{17}{8}<\frac{31}{8}$, the two error supports are disjoint. Hence
\begin{equation}\label{eq:exact-approx-L2}
	\Norm{\theta(t)-\overline\theta(t)}_2^2
	=
	\Norm{\theta_{\rad,\mathrm{err}}(t)}_2^2
	+
	\Norm{\theta_{\osc,\mathrm{err}}(t)}_2^2.
\end{equation}

\begin{lemma}[Radial error]
	\label{lem:W-L2}
	Under the hypotheses of \cref{thm:fixed-block},
for every $0\leq t\leq T$,
	\begin{equation}\label{eq:W-L2}
		\Norm{\theta_{\rad,\mathrm{err}}(t)}_2
		\leq
		C\lambda^{-s}N^{-s}e^{-\kappa N/2+CKt}.
	\end{equation}
\end{lemma}

\begin{proof}
	Since $\overline\theta_{\rad}$ is radial and stationary,
	$u[\overline\theta_{\rad}]
	\cdot\nabla\overline\theta_{\rad}=0$. 
	Hence $\theta_{\rad}$ and
$\overline\theta_{\rad}$ satisfy, respectively,
	\[
	\left(\partial_t+u[\theta]\cdot\nabla\right)\theta_{\rad}=0,
	\quad
	\left(
	\partial_t+u[\overline\theta_{\rad}]\cdot\nabla
	\right)\overline\theta_{\rad}=0.
	\]
Since
	$\theta_{\rad}(0)=\overline\theta_{\rad}(0)=\theta_{\rad,0}$,
	we subtract the two equations and use $\theta-\overline\theta_{\rad}
	=\theta_{\osc}+\theta_{\rad,\mathrm{err}}$ to obtain
	\[
	\left(\partial_t+u[\theta]\cdot\nabla\right)
	\theta_{\rad,\mathrm{err}}
	=
	-\left(
	u[\theta_{\osc}]
	+
	u[\theta_{\rad,\mathrm{err}}]
	\right)
	\cdot\nabla\overline\theta_{\rad},
	\quad
	\theta_{\rad,\mathrm{err}}(0)=0.
	\]
	The $L^2$ energy estimate then yields
	\begin{equation}\label{eq:radial-error-L2-energy}
	\frac{\dd}{\dd t}
	\Norm{\theta_{\rad,\mathrm{err}}(t)}_2
	\leq
	\Norm{
		u_r[\theta_{\osc}(t)]
		\partial_r\overline\theta_{\rad}
	}_2
	+
	\Norm{u[\theta_{\rad,\mathrm{err}}(t)]}_2
	\Norm{\nabla\overline\theta_{\rad}}_\infty.
	\end{equation}
	
By \cref{lem:Nfold}, $\theta_{\osc}(t)$ is $2\pi/N$-periodic. The
support separation follows from \eqref{eq:support-improved}, and exact
transport gives
$\|\theta_{\osc}(t)\|_\infty=\|\theta_{\osc,0}\|_\infty$.
Scaling and \cref{lem:radial-velocity-decay}, with decay rate
$\kappa$ and $(1+N)^2e^{-\kappa N}\leq Ce^{-\kappa N/2}$, imply
	\begin{align*}
		\Norm{u_r[\theta_{\osc}(t)]}_{
			L^\infty(\supp\nabla\overline\theta_{\rad})}
		&\leq
		Cc_0\lambda^{1-s}N^{-s}e^{-\kappa N/2},\\
		\Norm{\nabla\overline\theta_{\rad}}_\infty
		&\leq CK,\\
		|\supp\nabla\overline\theta_{\rad}|^{1/2}
		&\leq C\lambda^{-1}.
	\end{align*}
	These estimates and the $L^2$ boundedness of the Riesz transforms yield
	\begin{align*}
		\Norm{
			u_r[\theta_{\osc}(t)]
			\partial_r\overline\theta_{\rad}
		}_2
		&\leq
		Cc_0K\lambda^{-s}N^{-s}e^{-\kappa N/2},\\
		\Norm{u[\theta_{\rad,\mathrm{err}}(t)]}_2
		\Norm{\nabla\overline\theta_{\rad}}_\infty
		&\leq
		CK\Norm{\theta_{\rad,\mathrm{err}}(t)}_2.
	\end{align*}
	Substituting these bounds into \eqref{eq:radial-error-L2-energy}
	and applying Gronwall's inequality to
	$\Norm{\theta_{\rad,\mathrm{err}}(t)}_2$, with $\theta_{\rad,\mathrm{err}}(0)=0$,
	proves \eqref{eq:W-L2}.
\end{proof}

\begin{lemma}[Oscillatory error]
	\label{lem:E-L2}
	Under the hypotheses of \cref{thm:fixed-block},
	for every $0\leq t\leq T$,
	\begin{equation}\label{eq:E-L2}
	\Norm{\theta_{\osc,\mathrm{err}}(t)}_2
		\leq
		C	\lambda^{-s} N^{-s}\left(
		Kt(1+Kt)N^{1-s}
		+
		e^{-\kappa N/2+CKt}
		\right).
	\end{equation}
\end{lemma}

	\begin{proof}
		Subtracting the exact and approximate oscillatory equations gives
		\[
		\left(
		\partial_t
		+
		u[\theta]\cdot\nabla
		\right)
		\theta_{\osc,\mathrm{err}}
		=
		-u[\theta_{\osc}+\theta_{\rad,\mathrm{err}}]
		\cdot\nabla\overline\theta_{\osc},
		\quad
		\theta_{\osc,\mathrm{err}}(0)=0.
		\]
		Since $u[\theta]$ is divergence-free, the $L^2$ energy estimate yields
		\[
		\frac{\dd}{\dd t}
		\Norm{\theta_{\osc,\mathrm{err}}(t)}_2
	\leq
	\left(
	\Norm{u[\theta_{\osc}(t)]}_2
	+
		\Norm{u[\theta_{\rad,\mathrm{err}}(t)]}_2
		\right)
		\Norm{\nabla\overline\theta_{\osc}(t)}_\infty.
		\]
		The explicit formula \eqref{eq:explicit-approximate} gives
		\[
		\Norm{\nabla\overline\theta_{\osc}(t)}_\infty
		\leq
		CKN^{1-s}(1+Kt).
		\]
		By \eqref{eq:component-L2-exact}, \cref{lem:W-L2}, and the $L^2$
		boundedness of the Riesz transforms,
		\begin{align*}
			\Norm{u[\theta_{\osc}(t)]}_2
			&\leq
			C\Norm{\theta_{\osc,0}}_2\leq C \lambda^{-s}N^{-s},\\
			\Norm{u[\theta_{\rad,\mathrm{err}}(t)]}_2
			&\leq
			C\lambda^{-s}N^{-s}e^{-\kappa N+CKt}.
		\end{align*}
		The parameter choice in \cref{thm:fixed-block} gives
		$1+KT\le N$, and hence
		$Kt(1+Kt)N^{1-s}\le N^{3-s}\lesssim e^{\kappa N/2}$.
		Integrating the energy inequality with the preceding $L^2$ bounds
		then yields \eqref{eq:E-L2}.
	\end{proof}

Combining \eqref{eq:exact-approx-L2}, \eqref{eq:W-L2} and
\eqref{eq:E-L2} with the initial $L^2$ estimates
\eqref{eq:osc-L2-size} in \cref{cor:initial-small}, we obtain
\begin{equation}\label{eq:total-L2}
	\frac{
		\Norm{\theta(t)-\overline\theta(t)}_2
	}{
		\Norm{\theta_{\osc,0}}_2
	}
	\leq
	C\left(
	Kt(1+Kt)N^{1-s}
	+
	e^{-\kappa N/2+CKt}
	\right),
	\quad
	0\leq t\leq T.
\end{equation}

\subsection{Proof of Theorem~\ref{thm:block-inflation}}
\label{sec:optimization}

To transfer the growth of the approximate oscillatory component
to the modes $\pm N$ of the exact solution, we first project
onto these angular modes. For $1<\sigma<2$, we then use the
$L^2$ error estimate \eqref{eq:total-L2} and the WKB bounds in
\cref{lem:WKB}, interpolating at an exponent below $\sigma$
to obtain an $H^\sigma$ lower bound for the exact solution.

\begin{proof}[Proof of \cref{thm:block-inflation}]
Since $H^\sigma(\R^2)\hookrightarrow H^{\widehat\sigma}(\R^2)$
whenever $\sigma\geq\widehat\sigma$, it suffices to prove \eqref{eq:block-inflation}
for $\sigma_*(s)<\sigma<2$.

For $\delta\geq0$, set
\[
a_\delta:=\frac{2(2-s)+\delta}{s-1},
~\sigma_\delta(s):=\frac{s(1+a_\delta)}{3-s+a_\delta}
=\frac{s(3-s+\delta)}{1+2s-s^2+\delta}.
\]
The function $\sigma_\delta(s)$ is continuous and strictly increasing
in $\delta$, with
\[
\sigma_0(s)
=
\frac{s(3-s)}{1+2s-s^2}
=
\sigma_*(s).
\]
We may therefore choose $\delta>0$ sufficiently small that
$\sigma_*(s)<\sigma_\delta(s)<\sigma$.
With this $\delta$ fixed, choose $\alpha$ according to
\eqref{eq:alpha-choice}. Set
\[
\eta:=\frac{\sigma_\delta(s)+\sigma}{2},
~\sigma_\delta(s)<\eta<\sigma.
\]

Given $\eps>0$ and $M>0$, we fix $c_0>0$ small enough, via \eqref{eq:initial-small}, so that
	\[
	\Norm{\theta_0}_{H^s}\leq\eps.
	\]
	For the chosen $\delta$, let $\lambda$ be sufficiently large and set
	\[
	K:=c_0\lambda^{2-s},
	\quad
	N:=\left\lfloor\lambda^{a_\delta}\right\rfloor.
	\]
By \cref{thm:fixed-block}, the corresponding solution extends smoothly to
$[0,T]$ and satisfies the bootstrap bounds
\eqref{eq:osc-bootstrap}--\eqref{eq:support-bootstrap}
throughout this interval.

	Set
	\[
	d_\lambda
	:=
	C\left(
	KT(1+KT)N^{1-s}
	+
	e^{-\kappa N/2+CKT}
	\right),
	\quad
	\mathfrak F_\lambda
	:=
	\lambda NK.
	\]
	Then \eqref{eq:total-L2} gives
	\[
	\sup_{0\leq t\leq T}
	\frac{
		\Norm{\theta(t)-\overline\theta(t)}_2
	}{
		\Norm{\theta_{\osc,0}}_2
	}
	\leq d_\lambda.
	\]
	Since
	$a_\delta(s-1)=2(2-s)+\delta$ and $a_\delta>2-s$, the parameter
	choice above also gives
	\[
    d_\lambda\longrightarrow0,
	\quad
	\mathfrak F_\lambda\longrightarrow\infty
	\quad\text{as }\lambda\longrightarrow\infty.
	\]
	
	To isolate the oscillatory modes, use the angular Fourier expansion
	\[
	h(r,\varphi)
	=
	\sum_{m\in\mathbb Z}h_m(r)e^{im\varphi},
	~h_m(r)
	:=
	\frac{1}{2\pi}\int_0^{2\pi}
	h(r,\psi)e^{-im\psi}\,\dd\psi.
	\]
	Define the projection onto the modes $\pm N$ by
	\[
	\Pi_Nh
	:=
	h_N(r)e^{iN\varphi}
	+
	h_{-N}(r)e^{-iN\varphi}.
	\]
	The rotation invariance of $(I-\Delta)^{\rho/2}$ implies that it
	commutes with $\Pi_N$. Since $\Pi_N$ is an orthogonal projection on
	$L^2(\R^2)$, we obtain
	\begin{equation}\label{eq:PiN-contraction}
		\Norm{\Pi_Nh}_{H^\rho}
		=
		\Norm{\Pi_N(I-\Delta)^{\rho/2}h}_2
		\leq
		\Norm{h}_{H^\rho},
		~\rho\geq0.
	\end{equation}
	By \eqref{eq:explicit-approximate}, $\overline\theta_{\rad}$ contains
	only the angular mode $0$, while $\overline\theta_{\osc}$ contains
	only the modes $\pm N$. Thus
	$\Pi_N\overline\theta=\overline\theta_{\osc}$, and
	\[
	\Pi_N\theta-\overline\theta_{\osc}
	=
	\Pi_N(\theta-\overline\theta).
	\]
	Applying \eqref{eq:PiN-contraction} with $\rho=0$ and using
	\eqref{eq:total-L2} gives
	\begin{equation}\label{eq:PiN-L2-error}
		\Norm{\Pi_N\theta(t)-\overline\theta_{\osc}(t)}_2
		\leq
		d_\lambda\Norm{\theta_{\osc,0}}_2,
		~0\leq t\leq T.
	\end{equation}

By \cref{lem:WKB}, uniformly for $t\in[\tau,T]$,
\begin{align}
	\Norm{\overline\theta_{\osc}(t)}_{H^\eta}
	&\geq
	c_\eta\Norm{\theta_{\osc,0}}_2
	(t\mathfrak F_\lambda)^\eta,
	\label{eq:direct-WKB-lower}\\
	\Norm{\overline\theta_{\osc}(t)}_{H^\sigma}
	&\leq
	C_\sigma\Norm{\theta_{\osc,0}}_2
	(t\mathfrak F_\lambda)^\sigma.
	\label{eq:direct-WKB-upper}
\end{align}
Since
$\Norm{\overline\theta_{\osc}(t)}_2=\Norm{\theta_{\osc,0}}_2$,
\eqref{eq:PiN-L2-error} and the triangle inequality give
\[
\Norm{\Pi_N\theta(t)}_2
\leq
(1+d_\lambda)\Norm{\theta_{\osc,0}}_2
\leq
2\Norm{\theta_{\osc,0}}_2
\]
for sufficiently large $\lambda$.
	
	Fix $t\in[\tau,T]$. The smoothness and compact support given by
	\cref{thm:fixed-block,prop:support} ensure that
	$\Pi_N\theta(t)\in H^\sigma$. Interpolation between $L^2$ and
	$H^\sigma$ gives
	\begin{equation}\label{eq:PiN-Heta}
		\Norm{\Pi_N\theta(t)}_{H^\eta}
		\leq
		C
		\Norm{\theta_{\osc,0}}_2^{1-\eta/\sigma}
		\Norm{\Pi_N\theta(t)}_{H^\sigma}^{\eta/\sigma}.
	\end{equation}
The triangle inequality and \eqref{eq:direct-WKB-upper} give
\[
\Norm{\Pi_N\theta(t)-\overline\theta_{\osc}(t)}_{H^\sigma}
\leq
\Norm{\Pi_N\theta(t)}_{H^\sigma}
+
C_\sigma\Norm{\theta_{\osc,0}}_2
(t\mathfrak F_\lambda)^\sigma.
\]
Interpolation between $L^2$ and $H^\sigma$, followed by
\eqref{eq:PiN-L2-error} and the inequality
$(a+b)^{\eta/\sigma}\leq a^{\eta/\sigma}+b^{\eta/\sigma}$, yields
\begin{align}
	\Norm{\Pi_N\theta(t)-\overline\theta_{\osc}(t)}_{H^\eta}
	&\leq
	\Norm{\Pi_N\theta(t)-\overline\theta_{\osc}(t)}_2^{1-\eta/\sigma}
	\Norm{\Pi_N\theta(t)-\overline\theta_{\osc}(t)}_{H^\sigma}^{\eta/\sigma}
	\notag\\
	&\leq
	Cd_\lambda^{1-\eta/\sigma}
	\Norm{\theta_{\osc,0}}_2^{1-\eta/\sigma}
	\Norm{\Pi_N\theta(t)}_{H^\sigma}^{\eta/\sigma}
	\notag\\
	&\quad+
	Cd_\lambda^{1-\eta/\sigma}
	\Norm{\theta_{\osc,0}}_2
	(t\mathfrak F_\lambda)^\eta.
	\label{eq:error-Heta}
\end{align}

	Collecting \eqref{eq:direct-WKB-lower},
\eqref{eq:PiN-Heta} and \eqref{eq:error-Heta} yields
\begin{align}
	c_\eta\Norm{\theta_{\osc,0}}_2
	(t\mathfrak F_\lambda)^\eta
	&\leq
	C\left(
	1+d_\lambda^{\,1-\eta/\sigma}
	\right)
	\Norm{\theta_{\osc,0}}_2^{1-\eta/\sigma}
	\Norm{\Pi_N\theta(t)}_{H^\sigma}^{\eta/\sigma}\nonumber\\\label{eq:PiN-lower-preabsorption}
	&\quad+
	Cd_\lambda^{\,1-\eta/\sigma}
	\Norm{\theta_{\osc,0}}_2
	(t\mathfrak F_\lambda)^\eta.
\end{align}
Since \(d_\lambda\to0\), for sufficiently large \(\lambda\),
\[
Cd_\lambda^{\,1-\eta/\sigma}\leq\frac{c_\eta}{2},
\quad
1+d_\lambda^{\,1-\eta/\sigma}\leq2.
\]
We absorb the last term in \eqref{eq:PiN-lower-preabsorption}
into the left-hand side and divide by
$\Norm{\theta_{\osc,0}}_2^{1-\eta/\sigma}$ to obtain
\[
\Norm{\Pi_N\theta(t)}_{H^\sigma}^{\eta/\sigma}
\geq
c_{\eta,\sigma}
\Norm{\theta_{\osc,0}}_2^{\eta/\sigma}
(t\mathfrak F_\lambda)^\eta.
\]
Raising both sides to the power $\sigma/\eta$ and applying
\eqref{eq:PiN-contraction}, we conclude that
\begin{equation}\label{eq:true-lower}
	\Norm{\theta(t)}_{H^\sigma}
	\geq
	\Norm{\Pi_N\theta(t)}_{H^\sigma}
	\geq
	c_{\eta,\sigma}\Norm{\theta_{\osc,0}}_2
	(t\mathfrak F_\lambda)^\sigma,
	~t\in[\tau,T].
\end{equation}
	
	By \eqref{eq:osc-L2-size} and the definitions of
	$K$, $N$ and $\mathfrak F_\lambda$,
	\[
	\Norm{\theta_{\osc,0}}_2
	(t\mathfrak F_\lambda)^\sigma
	\gtrsim_{g,\sigma}
	\tau^\sigma c_0^{1+\sigma}
	\lambda^{
		\sigma(3-s)-s+a_\delta(\sigma-s)
	}.
	\]
	Since
	\(\sigma>\sigma_\delta(s):=\frac{s(1+a_\delta)}
	{3-s+a_\delta}\), we have
	\(\sigma(3-s+a_\delta)-s(1+a_\delta)>0\).
	Thus the displayed lower bound tends to $\infty$ as
	$\lambda\to\infty$, uniformly for $t\in[\tau,T]$, and is therefore
	at least $M$ for sufficiently large $\lambda$, proving
	\eqref{eq:block-inflation}.
\end{proof}

\section{Gluing and instantaneous loss}
\label{sec:gluing}

\subsection{Lower regularity and spatial decoupling}

For the gluing and uniqueness arguments, we need each block to be
small in $H^{1+\gamma}$ uniformly on $[0,T]$.
Fix
\begin{equation}\label{eq:gamma-choice}
 \gamma:=\frac{s-1}{2(1+2s-s^2)}.
\end{equation}
Then $0<\gamma<\sigma_*(s)-1$ by \eqref{eq:sigma-above-one}.

\begin{lemma}
	\label{lem:low-block}
	For arbitrarily small $\zeta>0$, the initial datum can be chosen so that
	the conclusions of \cref{thm:block-inflation} hold and
	\begin{equation}\label{eq:low-block}
		\sup_{0\leq t\leq T}
		\Norm{\theta(t)}_{H^{1+\gamma}}
		\leq \zeta,
		\quad
		\sup_{0\leq t\leq T}
		\Norm{u[\theta(t)]}_\infty
		\leq \zeta.
	\end{equation}
\end{lemma}

\begin{lemma}[Arbitrary smallness in $H^{1+\gamma}$]
	\label{lem:low-block}
	For arbitrarily small $\zeta>0$, the initial datum can be chosen so
	that the conclusions of \cref{thm:block-inflation} hold and
	\begin{equation}\label{eq:low-block}
		\sup_{0\leq t\leq T}\Norm{\theta(t)}_{H^{1+\gamma}}
		\leq\zeta,
		\qquad
		\sup_{0\leq t\leq T}\Norm{u[\theta(t)]}_\infty
		\leq\zeta.
	\end{equation}
\end{lemma}

\begin{proof}
	We first estimate the oscillatory component in $H^2$. Set
	$X(t):=\Norm{\widetilde\Theta_{\osc}(t)}_{H^2}$.
	By \eqref{eq:support-improved}, $\chi=1$ on
	$\supp\widetilde\Theta_{\osc}(t)$. The inner bound in
	\cref{lem:pullback-H2}, support separation, and the bounds on
	$\cS_t$ give
	\[
	\begin{aligned}
		\Norm{\chi\cB_t[\widetilde\Theta_{\osc}]}_{H^2}
		&\leq C\lambda(1+Kt)X(t),\\
		\sum_{j=0}^{2}
		\Norm{\nabla^j u[\Theta_{\rad,\err}(t)]}
		_{L^\infty(\supp\chi)}
		&\leq C\Norm{\Theta_{\rad,\err}(t)}_1
		\leq C\Norm{\widetilde\Theta_{\rad,\err}(t)}_2,\\
		\Norm{\chi\cB_t[\widetilde\Theta_{\rad,\err}]}_{H^2}
		&\leq C\lambda(1+Kt)^3
		\Norm{\widetilde\Theta_{\rad,\err}(t)}_2.
	\end{aligned}
	\]
	Thus the $H^2$ estimate for \eqref{eq:Q-equation} yields
	\[
	\begin{aligned}
		X'(t)\leq{}&
		C\Bigl(\Norm{\chi V(t)}_{C^1}
		+\lambda(1+Kt)
		\Norm{\nabla\widetilde\Theta_{\osc}(t)}_\infty\Bigr)X(t)\\
		&+C\lambda(1+Kt)^3
		\Norm{\widetilde\Theta_{\rad,\err}(t)}_2
		\Norm{\nabla\widetilde\Theta_{\osc}(t)}_\infty.
	\end{aligned}
	\]
	By \eqref{eq:Q-improved}, \eqref{eq:Q-smallness}, and
	\eqref{eq:Wtilde-improved}, we have
	\[
	\begin{aligned}
		\int_0^T\Bigl(
		\Norm{\chi V(t)}_{C^1}
		+\lambda(1+Kt)
		\Norm{\nabla\widetilde\Theta_{\osc}(t)}_\infty
		\Bigr)\,\dd t
		&\leq C\varepsilon,\\
		\int_0^T\lambda(1+Kt)^3
		\Norm{\widetilde\Theta_{\rad,\err}(t)}_2
		\Norm{\nabla\widetilde\Theta_{\osc}(t)}_\infty\,\dd t
		&\leq Cc_0\lambda^{1-s}N^{2-s}
		TK(1+KT)^3N^{-s-1}e^{-3\kappa N/4}\\
		&=o\bigl(c_0\lambda^{1-s}N^{2-s}\bigr).
	\end{aligned}
	\]
	Since $X(0)\leq Cc_0\lambda^{1-s}N^{2-s}$,
	Gronwall's inequality gives
	\begin{equation}\label{eq:osc-adapted-H2}
		\sup_{0\leq t\leq T}
		\Norm{\widetilde\Theta_{\osc}(t)}_{H^2}
		\leq Cc_0\lambda^{1-s}N^{2-s}.
	\end{equation}
	
	We next estimate the radial error in the original rescaled variable.
	Its support is separated from $\Theta_{\osc}$. 
	Conservation of
	$\Norm{\Theta_{\osc}}_1$ and
	\cref{lem:far-field,lem:radial-velocity-decay} give
	\[
	\begin{aligned}
		\lambda\Norm{u[\Theta_{\osc}(t)]}
		_{C^2(\{7/2<R<11/2\})}
		&\leq CKN^{-s},\\
		\Norm{Kf'u_r[\Theta_{\osc}(t)]}_{H^2}
		&\leq Cc_0K\lambda^{1-s}N^{-s}
		(1+N)^4e^{-\kappa N}.
	\end{aligned}
	\]
	Furthermore, \eqref{eq:Wtilde-improved} and \cref{lem:composition} imply
	\[
	\lambda\Norm{\Theta_{\rad,\err}(t)}_{C^{1,\alpha}}
	\leq CK(1+Kt)^{1+\alpha}
	N^{-s}e^{-3\kappa N/4}.
	\]
	Set $Z_2(t):=\Norm{\Theta_{\rad,\err}(t)}_{H^2}$.
	The $H^2$ estimate for \eqref{eq:W-fixed-scale} now gives
	\[
	\begin{aligned}
		Z_2'(t)\leq{}&
		C\Bigl(K+KN^{-s}
		+\lambda\Norm{\Theta_{\rad,\err}(t)}_{C^{1,\alpha}}\Bigr)
		Z_2(t)\\
		&+Cc_0K\lambda^{1-s}N^{-s}
		(1+N)^4e^{-\kappa N}
	\end{aligned}
	\]
	with $ Z_2(0)=0.$
	The time integral of the coefficient of $Z_2$ is $o(N)$, since
	$KT/N\to0$. Hence
	\begin{equation}\label{eq:radial-error-H2}
		\begin{aligned}
			\sup_{0\leq t\leq T}Z_2(t)
			&\leq Cc_0\lambda^{1-s}N^{-s}
			TK(1+N)^4e^{-\kappa N+o(N)}\\
			&\leq Cc_0\lambda^{1-s}N^{-s}e^{-\kappa N/2}.
		\end{aligned}
	\end{equation}
	
	Measure preservation, \eqref{eq:shear-derivative-bounds}, and
	\eqref{eq:osc-adapted-H2} yield
	\[
	\Norm{\Theta_{\osc}(t)}_2
	\simeq_g c_0\lambda^{1-s}N^{-s},
	\qquad
	\Norm{\Theta_{\osc}(t)}_{H^2}
	\leq Cc_0\lambda^{1-s}(1+Kt)^2N^{2-s}.
	\]
	Interpolation and \eqref{eq:Sobolev-scaling} therefore give, for
	$0\leq q\leq2$,
	\begin{equation}\label{eq:exact-osc-Sobolev-upper}
		\sup_{0\leq t\leq T}\Norm{\theta_{\osc}(t)}_{H^q}
		\leq Cc_0\lambda^{q-s}(1+KT)^qN^{q-s}.
	\end{equation}
	Combining this with the fixed radial profile and
	\eqref{eq:radial-error-H2}, we obtain, for $q=1+\gamma$,
	\begin{equation}\label{eq:exact-low-Sobolev-upper}
		\sup_{0\leq t\leq T}\Norm{\theta(t)}_{H^q}
		\leq Cc_0\lambda^{q-s}
		\left[1+(1+KT)^qN^{q-s}
		+N^{-s}e^{-\kappa N/2}\right].
	\end{equation}
	No condition $\gamma<\alpha$ is needed. Since $q<s$ and
	\[
	q(3-s+a_\delta)-s(1+a_\delta)
	=-\frac12-\frac{s-1-\gamma}{s-1}\delta<0,
	\]
	the right-hand side of \eqref{eq:exact-low-Sobolev-upper} tends to
	zero as $\lambda\to\infty$. Finally, Sobolev embedding and the
	$H^{1+\gamma}$ boundedness of the Riesz transforms give
	\[
	\sup_{0\leq t\leq T}\Norm{u[\theta(t)]}_\infty
	\leq C_\gamma
	\sup_{0\leq t\leq T}\Norm{\theta(t)}_{H^{1+\gamma}}
	\longrightarrow0.
	\]
	Taking $\lambda$ sufficiently large proves \eqref{eq:low-block},
	while \eqref{eq:true-lower} retains the required norm inflation.
\end{proof}

To glue the blocks, we need their mutual interaction to be small.
The next lemma shows that widely separated smooth solutions are
approximately additive on $[0,T]$.

\begin{lemma}
	\label{lem:spatial-decoupling}
	Let $\phi$ and $\psi$ be smooth compactly supported solutions of
	\eqref{eq:SQG} on $[0,T]$. For any $\eta>0$, there exists $D_0>0$
	such that, whenever $|z|\geq D_0$, the smooth solution $\Theta_z$
	with initial datum
	\[
	\Theta_z(0)=\phi(0)+\psi(\cdot-z,0)
	\]
	exists on $[0,T]$ and satisfies
	\begin{equation}\label{eq:spatial-decoupling}
		\sup_{0\leq t\leq T}
		\Norm{\Theta_z(t)-\phi(t)-\psi(\cdot-z,t)}_{H^4}
		\leq\eta.
	\end{equation}
\end{lemma}

\begin{proof}
	Set
	\[
	\Psi_z(x,t):=\psi(x-z,t),
	\quad
	\Theta^{\mathrm{app}}:=\phi+\Psi_z.
	\]
	Since $\phi$ and $\psi$ are smooth and compactly supported on
	$[0,T]$, there exists $R_T<\infty$ such that
	\[
	\supp\phi(t)\cup\supp\psi(t)\subset B(0,R_T),
	\quad 0\leq t\leq T.
	\]
	If $|z|\geq4R_T$, set $D:=|z|-2R_T$. Then
	\[
	\operatorname{dist}
	\bigl(\supp\phi(t),\supp\Psi_z(t)\bigr)
	\geq D,
	\quad 0\leq t\leq T.
	\]

	Substituting $\Theta^{\mathrm{app}}$ into \eqref{eq:SQG}, we obtain
	the error term
	\[
	\mathcal E_z
	:=
	\partial_t\Theta^{\mathrm{app}}
	+u[\Theta^{\mathrm{app}}]\cdot
	\nabla\Theta^{\mathrm{app}}
	=
	u[\phi]\cdot\nabla\Psi_z
	+u[\Psi_z]\cdot\nabla\phi.
	\]
	We claim that
	\begin{equation}\label{eq:finite-residual}
		\sup_{0\leq t\leq T}
		\Norm{\mathcal E_z(t)}_{H^4}
		\leq C_{\phi,\psi}D^{-2},
	\end{equation}
	where $C_{\phi,\psi}$ may depend on $T$, $R_T$, $\phi$ and $\psi$,
	but is independent of $z$.

Indeed, for any multi-index $\beta$ with $|\beta|\leq4$,
\cref{lem:far-field} and the separation of the supports give
\[
\Norm{\partial^\mu u[\phi(t)]}_{L^\infty(\supp\Psi_z(t))}
\leq C_\mu D^{-2-|\mu|}\Norm{\phi(t)}_{L^1},
\quad \mu\leq\beta.
\]
Leibniz's rule therefore yields
\begin{equation}\label{eq:decoupling-first-interaction}
\begin{aligned}
\Norm{\partial^\beta\bigl(u[\phi]\cdot\nabla\Psi_z\bigr)}_2
&\leq C\sum_{\mu\leq\beta}
\Norm{\partial^\mu u[\phi]}_{L^\infty(\supp\Psi_z)}
\Norm{\partial^{\beta-\mu}\nabla\Psi_z}_2\\
&\leq C_{\phi,\psi}D^{-2}.
\end{aligned}
\end{equation}
Similarly,
\begin{equation}\label{eq:decoupling-second-interaction}
	\Norm{u[\Psi_z]\cdot\nabla\phi}_{H^4}
	\leq C_{\phi,\psi}D^{-2}.
\end{equation}
Summing \eqref{eq:decoupling-first-interaction} over $|\beta|\leq4$
and combining it with \eqref{eq:decoupling-second-interaction} proves
\eqref{eq:finite-residual}.

Local existence gives a solution $\Theta_z$ with initial datum
$\Theta^{\mathrm{app}}(0)$, defined on its maximal interval of
smooth existence $[0,T_z)$. The difference
$w:=\Theta_z-\Theta^{\mathrm{app}}$ satisfies
\begin{equation*}
	\partial_t w
	+u[\Theta^{\mathrm{app}}+w]\cdot\nabla w
	+u[w]\cdot\nabla\Theta^{\mathrm{app}}
	=-\mathcal E_z,
	~w(0)=0.
\end{equation*}
By the smoothness of $\phi$ and $\psi$ and translation invariance
of the Sobolev norms,
\[
\sup_{0\leq t\leq T}
\Norm{\Theta^{\mathrm{app}}(t)}_{H^5}
\leq
\sup_{0\leq t\leq T}
\left(
\Norm{\phi(t)}_{H^5}
+\Norm{\psi(t)}_{H^5}
\right)
<\infty,
\]	
uniformly in $z$. The standard product and commutator
estimates, together with Sobolev embedding, the boundedness of the
Riesz transforms and \eqref{eq:finite-residual}, yield
\begin{equation}\label{eq:finite-H4-energy}
	\frac{\dd}{\dd t}\Norm{w}_{H^4}
	\leq
	C_{\phi,\psi}
	\left(
	\Norm{w}_{H^4}
	+\Norm{w}_{H^4}^2
	+D^{-2}
	\right),
\end{equation}
for $0\leq t<\min\{T,T_z\}$.	
	
	Let $0<t_*<\min\{T,T_z\}$ and assume that
	\[
	\sup_{0\leq t\leq t_*}\Norm{w(t)}_{H^4}\leq1.
	\]
	Then \eqref{eq:finite-H4-energy} and Gronwall's inequality give
	\[
	\sup_{0\leq t\leq t_*}\Norm{w(t)}_{H^4}
	\leq C_{\phi,\psi}D^{-2}.
	\]
	Choose $D$ sufficiently large that
	$C_{\phi,\psi}D^{-2}\leq\min\{1/2,\eta\}$.
	This strictly improves the bootstrap assumption, so a continuity
	argument yields
	\begin{equation}\label{eq:finite-H4-small}
		\sup_{0\leq t<\min\{T,T_z\}}
		\Norm{w(t)}_{H^4}
		\leq C_{\phi,\psi}D^{-2}
		\leq\min\left\{\frac12,\eta\right\}.
	\end{equation}
	Since $\Theta_z=\Theta^{\mathrm{app}}+w$ is uniformly bounded in
	$H^4$ on $[0,\min\{T,T_z\})$, the continuation criterion implies
	$T_z>T$. Thus \eqref{eq:finite-H4-small} holds on $[0,T]$ by
	continuity, proving \eqref{eq:spatial-decoupling}.
\end{proof}

\subsection{Gluing of blocks and convergence}
Keep $\gamma$ fixed as in \eqref{eq:gamma-choice}; this gives one
common lower regularity space for the limiting solution. The H\"older
exponent is chosen separately for each block.
Set
\begin{equation}\label{eq:block-parameters}
	\sigma_j
	:=
	\sigma_*(s)+\frac{s-\sigma_*(s)}{j+1},
	~\tau_j:=\frac{T}{j+1}.
\end{equation}
Then $\sigma_j\downarrow\sigma_*(s)$ and $\tau_j\downarrow0$.
Hence each fixed $t\in(0,T]$ belongs to $[\tau_j,T]$ for all
sufficiently large $j$.

For each $j$, we follow the construction in \cref{sec:optimization},
choosing $\delta=\delta_j>0$ so that
$\sigma_{\delta_j}(s)<\sigma_j$, then choosing
$\alpha=\alpha_j>0$ to satisfy \eqref{eq:alpha-choice} with
$\delta=\delta_j$, and setting
$N=N_j:=\lfloor\lambda_j^{a_{\delta_j}}\rfloor$.
Fix all constants for the $j$th block before choosing $\lambda_j$.
The construction needs no uniform lower bound on $\alpha_j$.
By \cref{thm:block-inflation,lem:low-block}, we may take
$\lambda_j$ sufficiently large to obtain a smooth block solution
$\vartheta_j$ on $[0,T]$ such that
\begin{align}
	\Norm{\vartheta_j(0)}_{H^s}
	&\leq \eps 2^{-j-2}, \label{eq:block-list1}\\
	\sup_{0\leq t\leq T}\Norm{\vartheta_j(t)}_{H^{1+\gamma}}
	+\sup_{0\leq t\leq T}\Norm{u[\vartheta_j(t)]}_\infty
	&\leq 2^{-j-2}, \label{eq:block-list2}\\
	\Norm{\vartheta_j(t)}_{H^{\sigma_j}}
	&\geq 2^{2j},~t\in[\tau_j,T]. \label{eq:block-list3}
\end{align}
Let $\vartheta_{j,\osc}(0)$ denote the initial oscillatory component
of $\vartheta_j$. Increasing $\lambda_j$ if necessary also ensures that
\[ \supp\vartheta_j(t)\subset B(0,2^{-j-4}) \quad\text{for }0\leq t\leq T, \quad \Norm{\vartheta_{j,\osc}(0)}_2\leq1. \]

We now add the blocks one at a time, placing each new center far from
the preceding ones. The resulting smooth solutions will converge to
the solution with the desired initial datum. Set $\eta_1:=0$.
For $j\geq2$, define
\begin{equation*}
	\eta_j
	:=
	2^{-j-10}
	\min_{1\leq i\leq j}
	\Norm{\vartheta_{i,\osc}(0)}_2.
\end{equation*}
Fix $C_\gamma>0$ such that
\begin{equation}\label{Riesz-infty-H1+gamma}
	\Norm{u[h]}_\infty
	\leq C_\gamma\Norm{h}_{H^{1+\gamma}}.
\end{equation}
For $j\geq1$, define
\begin{equation}\label{eq:Mj-definition}
	M_j
	:=
	1+C_{\mathrm{uni}}
	\biggl(
	\sup_{0\leq t\leq T}\Norm{\vartheta_j(t)}_{H^4}
	+\sum_{n\geq\max\{2,j\}}\eta_n
	\biggr),
\end{equation}
where $C_{\mathrm{uni}}\geq1$ is a sufficiently large constant,
independent of $j$, to be fixed below.

Set $z_1:=0$ and $\theta^{(1)}:=\vartheta_1$.
For $J\geq1$, suppose that we have chosen the centers
$z_1,\dots,z_J$ and constructed a smooth compactly supported
solution $\theta^{(J)}$ on $[0,T]$ with initial datum
\begin{equation*}
	\theta^{(J)}(0)
	=
	\sum_{j=1}^J\vartheta_j(\cdot-z_j,0).
\end{equation*}
Let $D_{J+1}$ denote the constant $D_0$ in
\cref{lem:spatial-decoupling} corresponding to
$\phi=\theta^{(J)}$, $\psi=\vartheta_{J+1}$ and
$\eta=\eta_{J+1}$. Choose $z_{J+1}\in\R^2$ such that
\begin{equation}\label{eq:center-separation}
	|z_{J+1}|
	\geq
	\max\left\{
	J+1,\,
	D_{J+1},\,
	4+2^{-4}+TC_\gamma
	+\max_{1\leq i\leq J}|z_i|
	\right\}
\end{equation}
and
\begin{equation}\label{eq:uniqueness-separation}
	\begin{aligned}
		|z_{J+1}-z_i|
		\geq{}&
		2^{-4}+TC_\gamma\\
		&+\max\left\{
		2^{i+12}M_i e^{C_{\mathrm{uni}}M_iT},
		2^{J+13}M_{J+1}e^{C_{\mathrm{uni}}M_{J+1}T}
		\right\},
		~1\leq i\leq J.
	\end{aligned}
\end{equation}
Both conditions hold when $|z_{J+1}|$ is sufficiently large.
In particular, \eqref{eq:center-separation} gives
\[
|z_{J+1}-z_i|
\geq
4+2^{-4}+TC_\gamma,
~1\leq i\leq J.
\]
With this choice, \cref{lem:spatial-decoupling} yields a smooth
compactly supported solution $\theta^{(J+1)}$ on $[0,T]$ satisfying
\[
\theta^{(J+1)}(0)
=
\sum_{j=1}^{J+1}\vartheta_j(\cdot-z_j,0)
\]
and
\begin{equation}\label{eq:gluing-increment}
	\sup_{0\leq t\leq T}
	\Norm{
		\theta^{(J+1)}(t)-\theta^{(J)}(t)
		-\vartheta_{J+1}(\cdot-z_{J+1},t)
	}_{H^4}
	\leq\eta_{J+1}.
\end{equation}
This completes the induction. To record its error, set $R_1:=0$ and,
for $n\geq2$, define
\begin{equation*}
	R_n
	:=
	\theta^{(n)}-\theta^{(n-1)}-\vartheta_n(\cdot-z_n).
\end{equation*}

Since $\theta^{(1)}=\vartheta_1$ and $z_1=0$, we have
\[
	\theta^{(J)}(t)
	=
	\sum_{j=1}^J\vartheta_j(\cdot-z_j,t)+
	\sum_{j=2}^J
	\left[
	\theta^{(j)}(t)-\theta^{(j-1)}(t)
	-\vartheta_j(\cdot-z_j,t)
	\right].
\]
Translations preserve the $H^{1+\gamma}$ norm, and
$1+\gamma<4$. Thus \eqref{eq:block-list2} and
\eqref{eq:gluing-increment} give
\begin{equation}\label{eq:finite-Hgamma-bound}
		\sup_{0\leq t\leq T}
		\Norm{\theta^{(J)}(t)}_{H^{1+\gamma}}
		\leq
		\sum_{j=1}^J
		\sup_{0\leq t\leq T}
		\Norm{\vartheta_j(t)}_{H^{1+\gamma}}
		+\sum_{j=2}^J\eta_j\leq
		\frac14+2^{-10}
		<\frac12.
\end{equation}

For $1\leq j\leq J$, let $\theta_j^{(J)}$ solve
\[
\left\{
\begin{aligned}
	(\partial_t+u[\theta^{(J)}]\cdot\nabla)\theta_j^{(J)}
	&=0,\\
	\theta_j^{(J)}(0)
	&=\vartheta_j(\cdot-z_j,0).
\end{aligned}
\right.
\]
Uniqueness for the linear transport equation gives
\[
\theta^{(J)}=\sum_{j=1}^J\theta_j^{(J)}.
\]
By \eqref{eq:finite-Hgamma-bound} and
\eqref{Riesz-infty-H1+gamma},
\[
\sup_{0\leq t\leq T}
\Norm{u[\theta^{(J)}(t)]}_\infty
\leq\frac{C_\gamma}{2}.
\]
Each characteristic therefore moves by at most $tC_\gamma/2$.
Since
$\supp\theta_j^{(J)}(0)\subset B(z_j,2^{-j-4})$,
we obtain
\begin{equation}\label{eq:finite-component-support}
	\supp\theta_j^{(J)}(t)
	\subset
	B\left(z_j,2^{-5}+\frac{tC_\gamma}{2}\right)
	\subset
	B\left(z_j,2^{-5}+\frac{TC_\gamma}{2}\right),
	~0\leq t\leq T.
\end{equation}
Using this inclusion together with \eqref{eq:center-separation}, we obtain
\[
\operatorname{dist}
\left(\supp\theta_i^{(J)}(t),
\supp\theta_j^{(J)}(t)\right)
\geq4,
~1\leq i\neq j\leq J,
~0\leq t\leq T.
\]
The balls
$B\left(z_j,1+2^{-5}+\frac{TC_\gamma}{2}\right)$,
$j\geq1$, are also pairwise disjoint.

Define
\begin{equation}\label{eq:glued-data}
	\theta_0(x)
	:=
	\sum_{j=1}^\infty\vartheta_j(x-z_j,0).
\end{equation}
Since
$$
\supp\vartheta_j(\cdot-z_j,0)\subset B(z_j,2^{-j-4})
$$
and \(|z_j|\to\infty\), only finitely many terms in the series \eqref{eq:glued-data} are nonzero on any compact subset of \(\R^2\). By \eqref{eq:block-list1}, the series in \eqref{eq:glued-data}
converges in $H^s$, with
\[
\Norm{\theta_0}_{H^s}
\leq
\sum_{j=1}^\infty\Norm{\vartheta_j(0)}_{H^s}
\leq\eps.
\]

The summable block and interaction bounds now give convergence of
the finite solutions. By
\eqref{eq:block-list2} and \eqref{eq:gluing-increment},
\[
	\sup_{0\leq t\leq T}
	\Norm{\theta^{(J+1)}(t)-\theta^{(J)}(t)}_{H^{1+\gamma}}
	\leq
	\sup_{0\leq t\leq T}
	\Norm{\vartheta_{J+1}(t)}_{H^{1+\gamma}}
	+\eta_{J+1}\leq
	2^{-J-3}+\eta_{J+1}.
\]
The right-hand side is summable in $J$. Hence
$\{\theta^{(J)}\}$ is Cauchy in $L^\infty([0,T];H^{1+\gamma})$.
Consequently, there exists $\theta$ such that
\begin{equation}\label{eq:glued-limit}
	\theta^{(J)}\longrightarrow\theta
	\quad\text{in }
	L^\infty([0,T];H^{1+\gamma}(\R^2)).
\end{equation}
Finally, \eqref{eq:finite-Hgamma-bound} and
\eqref{Riesz-infty-H1+gamma} yield
\begin{equation}\label{eq:limit-Hgamma-bound}
	\sup_{0\leq t\leq T}
	\Norm{\theta(t)}_{H^{1+\gamma}}
	\leq\frac12,
	~\sup_{0\leq t\leq T}
	\Norm{u[\theta(t)]}_\infty
	\leq\frac{C_\gamma}{2}.
\end{equation}

To pass the equation to the limit pointwise, we also need local
$H^4$ convergence. Let $E\subset\R^2$ be
a bounded smooth domain. The support bounds for the blocks and
$|z_j|\to\infty$ give an index $J_E$ such that
\[
\supp\vartheta_j(\cdot-z_j,t)\cap E=\varnothing,
~j\geq J_E,
~0\leq t\leq T.
\]
Thus, for $M>J\geq J_E$, 
\[
\left.(\theta^{(M)}-\theta^{(J)})\right|_E
=
\left.
\sum_{n=J+1}^M
\left(
\theta^{(n)}-\theta^{(n-1)}
-\vartheta_n(\cdot-z_n)
\right)
\right|_E.
\]
By \eqref{eq:gluing-increment},
\[
\sup_{0\leq t\leq T}
\Norm{\theta^{(M)}(t)-\theta^{(J)}(t)}_{H^4(E)}
\leq
\sum_{n=J+1}^M\eta_n
\longrightarrow0
\]
as $J,M\to\infty$. Since each $\theta^{(J)}$ is smooth,
$\{\theta^{(J)}\}$ is Cauchy in $C([0,T];H^4(E))$. Let
$\theta_E$ denote its limit. Since $H^4(E)\hookrightarrow
H^{1+\gamma}(E)$, \eqref{eq:glued-limit} implies
$\theta_E=\theta$ on $E$. Therefore,
\begin{equation}\label{eq:local-H4-convergence}
	\theta^{(J)}
	\longrightarrow
	\theta
	\quad\text{in }C([0,T];H^4(E)).
\end{equation}
In particular,
$\theta^{(J)}\to\theta$ in $C([0,T];C^1(E))$.

The nonlocal velocity requires a separate local convergence
argument. Fix a compact set $G\subset\R^2$ and choose
$\chi\in C_c^\infty(\R^2)$ equal to
$1$ near $G$. Set $f_J:=\theta^{(J)}-\theta$ and write
\[
u[f_J]
=
u[\chi f_J]+u[(1-\chi)f_J].
\]
By \eqref{eq:local-H4-convergence}, the boundedness of the Riesz
transforms on $H^4$ and Sobolev embedding,
\[
\sup_{0\leq t\leq T}
\Norm{u[\chi f_J(t)]}_{C^1(G)}
\lesssim
\sup_{0\leq t\leq T}
\Norm{\chi f_J(t)}_{H^4}
\longrightarrow0.
\]
For the second term, set
$d:=\operatorname{dist}(G,\supp(1-\chi))>0$.
For $x\in G$ and $y\in\supp(1-\chi)$, we have $|x-y|\geq d$.
The bounds $|\nabla^\ell\cK(z)|\leq C_\ell|z|^{-2-\ell}$
give $\nabla^\ell\cK\in L^2(\{|z|\geq d\})$ for $\ell=0,1$.
Thus Cauchy--Schwarz and \eqref{eq:glued-limit} give
\[
\sup_{0\leq t\leq T}
\Norm{u[(1-\chi)f_J(t)]}_{C^1(G)}
\leq
C_{G,\chi}
\sup_{0\leq t\leq T}\Norm{f_J(t)}_2
\longrightarrow0.
\]
Since each $u[\theta^{(J)}]$ belongs to
$C([0,T];C^1(G))$, these estimates yield
\begin{equation}\label{eq:local-velocity-convergence}
	u[\theta^{(J)}]\longrightarrow u[\theta]
	\quad\text{in }C([0,T];C^1(G)).
\end{equation}

For each $J$, integrating the equation for $\theta^{(J)}$ in time gives
\[
\theta^{(J)}(t)-\theta^{(J)}(0)
=
-\int_0^t
u[\theta^{(J)}(\zeta)]
\cdot\nabla\theta^{(J)}(\zeta)\,\dd\zeta.
\]
By \eqref{eq:local-H4-convergence} and
\eqref{eq:local-velocity-convergence}, the integrand converges locally
uniformly on $\R^2\times[0,T]$. Passing to the
limit as $J\to\infty$, we obtain
\[
\theta(t)-\theta(0)
=
-\int_0^t
u[\theta](\zeta)\cdot\nabla\theta(\zeta)\,\dd\zeta.
\]
Both $\nabla\theta$ and $u[\theta]$ are continuous in space and
time, so $u[\theta]\cdot\nabla\theta$ is continuous. We may therefore differentiate
the integral identity to obtain
\[
\partial_t\theta = -u[\theta]\cdot\nabla\theta
\quad\text{in }(0,T).
\]
At $t=0$ and $t=T$, the relation holds with the right and left
derivatives, respectively. Thus $\theta$, $\nabla\theta$, and
$\partial_t\theta$ are continuous in the interior and extend
continuously to the endpoints. In particular, $\theta$ satisfies
\eqref{eq:SQG} pointwise on $\R^2\times[0,T]$.

Moreover, \eqref{eq:local-velocity-convergence} and
\eqref{eq:limit-Hgamma-bound} give
\[
u[\theta]\in
C([0,T];C^1_{\mathrm{loc}}(\R^2))
\cap L^\infty([0,T];L^\infty(\R^2)).
\]
Together with \eqref{eq:glued-limit} and \eqref{eq:glued-data},
these properties show that $\theta$ is a classical solution of
\eqref{eq:SQG} on $[0,T]$ with initial datum $\theta_0$ in the
sense of \cref{def:classical}.

\subsection{Instantaneous loss and uniqueness}

\begin{proof}[Proof of \cref{thm:main-loss}]

	Let $\theta$ be the solution constructed in the preceding subsection.
We first identify the transported blocks, then prove loss of regularity,
and finally establish uniqueness.
	
\medskip
\noindent\emph{Step 1. Transported blocks.}

Recall the fixed choice \eqref{eq:gamma-choice}, for which
$0<\gamma<\sigma_*(s)-1$ and
\[
\theta\in L^\infty([0,T];H^{1+\gamma}(\R^2)).
\]
Let $\Phi$ denote the flow generated by $u[\theta]$. Since
$u[\theta]\in C([0,T];C^1_{\text{loc}}(\R^2))$ and is uniformly bounded,
the flow is uniquely defined on $[0,T]$. By
\eqref{eq:limit-Hgamma-bound},
\[
|\Phi(x,t)-x|
\leq
\frac{tC_\gamma}{2}.
\]
For each $j$, let $\theta_j$ denote the transport of
$\vartheta_j(\cdot-z_j,0)$ by the flow $\Phi$, namely,
\[
\theta_j\bigl(\Phi(x,t),t\bigr)
:=
\vartheta_j(x-z_j,0).
\]
Since
$\supp\vartheta_j(\cdot-z_j,0)\subset B(z_j,2^{-j-4})$,
we have
\begin{equation}\label{eq:limit-component-support}
	\supp\theta_j(t)
	\subset
	B\left(z_j,2^{-5}+\frac{TC_\gamma}{2}\right),
	\quad
	0\leq t\leq T.
\end{equation}
By \eqref{eq:center-separation}, these supports remain mutually
separated and locally finite. Moreover, since
$\theta(\Phi(x,t),t)=\theta_0(x)$, the local finiteness of
\eqref{eq:glued-data} gives $
\theta=\sum_{j=1}^\infty\theta_j.$

\medskip
\noindent\emph{Step 2. Instantaneous Sobolev loss.}

Localizing near a transported block isolates its norm growth from
the other blocks. We then compare it with the corresponding isolated
solution. Choose a radial function $\chi\in C_c^\infty(\R^2)$ such that
\[
0\leq\chi\leq1,~
\chi=1\ \text{on }B\left(0,\frac12+2^{-5}+\frac{TC_\gamma}{2}\right),
~\supp\chi\subset B\left(0,1+2^{-5}+\frac{TC_\gamma}{2}\right).
\]
and set $\chi_j(x):=\chi(x-z_j)$. By
\eqref{eq:center-separation} and
\eqref{eq:limit-component-support}, for every $0\leq t\leq T$,
$\chi_j=1$ on $\supp\theta_j(t)$ and
$\chi_j=0$ on $\supp\theta_i(t)$ whenever $i\neq j$.

By \eqref{eq:block-parameters},
$\sigma_j\in[\sigma_*(s),(s+\sigma_*(s))/2]\subset(1,2)$.
Since $\chi_j$ is a translate of the fixed function $\chi$,
multiplication by $\chi_j$ is bounded on $H^{\sigma_j}$ uniformly in
$j$. Thus
\[
\Norm{\chi_j f}_{H^{\sigma_j}}
\leq
C\Norm{f}_{H^{\sigma_j}},
\]
where $C$ is independent of $j$.

For each $n\geq2$, \eqref{eq:finite-component-support} gives
\[
\supp R_n(t)
\subset
\bigcup_{i=1}^n
B\left(z_i,2^{-5}+\frac{TC_\gamma}{2}\right).
\]
Hence, if $n<j$, this support is disjoint from $\supp\chi_j$. Moreover,
$\chi_j\vartheta_n(\cdot-z_n)=0$ for $n\neq j$, while
$\chi_j\vartheta_j(\cdot-z_j)=\vartheta_j(\cdot-z_j)$. Therefore, for
$J\geq j\geq1$,
\begin{equation}\label{eq:localized-telescoping}
	\chi_j\theta^{(J)}
	-\vartheta_j(\cdot-z_j)
	=
	\chi_j\sum_{n=\max\{2,j\}}^J R_n.
\end{equation}
By \eqref{eq:gluing-increment} and
\eqref{eq:local-H4-convergence}, we may let $J\to\infty$ and obtain
\[
\chi_j\theta
-\vartheta_j(\cdot-z_j)
=
\chi_j\sum_{n\geq\max\{2,j\}}R_n.
\]
Since $\sigma_j<2<4$, \eqref{eq:gluing-increment} and the definition
of $\eta_n$ yield
\begin{equation*}
	\sup_{0\leq t\leq T}
	\Norm{
		\chi_j\theta(t)-\vartheta_j(\cdot-z_j,t)
	}_{H^{\sigma_j}}
	\leq
	C\sum_{n\geq\max\{2,j\}}\eta_n
	\leq
	C2^{-j-9}
	\Norm{\vartheta_{j,\osc}(0)}_2.
\end{equation*}
Thus, by \eqref{eq:block-list3} and
$\Norm{\vartheta_{j,\osc}(0)}_2\leq1$, for all sufficiently large $j$,
\[
\Norm{\chi_j\theta(t)}_{H^{\sigma_j}}
\geq
2^{2j}
-
C2^{-j-9}\Norm{\vartheta_{j,\osc}(0)}_2
\geq
2^j,
\quad
t\in[\tau_j,T].
\]

Fix $t\in(0,T]$ and $\sigma>\sigma_*(s)$. By
\eqref{eq:block-parameters}, for all sufficiently large $j$,
$t\geq\tau_j$ and $\sigma_j<\sigma$. If
$\theta(t)\in H^\sigma$, then the multiplier estimate gives
\[
2^j
\leq
\Norm{\chi_j\theta(t)}_{H^{\sigma_j}}
\leq
C\Norm{\theta(t)}_{H^\sigma},
\]
where $C$ is independent of $j$. Letting $j\to\infty$ gives a
contradiction. Therefore,
\[
\theta(t)\notin H^\sigma
\quad
\text{for every }t\in(0,T]
\text{ and every }\sigma>\sigma_*(s),
\]
which proves \eqref{eq:main-loss}.

\medskip
\noindent\emph{Step 3. Uniqueness.}

Set
\[
E_{j,0}
:=
\supp\vartheta_j(\cdot-z_j,0),
\quad
D_j
:=
\operatorname{dist}\biggl(
E_{j,0},
\bigcup_{i\neq j}E_{i,0}
\biggr).
\]
For each fixed $j$, choose a bounded smooth domain $E$ with
$\supp\chi_j\subset E$. Multiplying
\eqref{eq:local-H4-convergence} by $\chi_j$ gives
\[
\chi_j\theta^{(J)}
\longrightarrow
\chi_j\theta=\theta_j
\quad\text{in }L^\infty([0,T];H^4(\R^2)).
\]
Passing to the limit in \eqref{eq:localized-telescoping} using
\eqref{eq:local-H4-convergence}, and applying
\eqref{eq:gluing-increment}, we obtain
\begin{equation}\label{eq:uniqueness-block-H4}
	\begin{aligned}
		\sup_{0\leq t\leq T}
		\Norm{\theta_j(t)-\vartheta_j(\cdot-z_j,t)}_{H^4(\R^2)}&\leq
		C\sum_{n\geq\max\{2,j\}}
		\sup_{0\leq t\leq T}
		\Norm{R_n(t)}_{H^4(\R^2)}\\
		&\leq C\sum_{n\geq\max\{2,j\}}\eta_n,
	\end{aligned}
\end{equation}
where $C$ is independent of $j$ and the centers $z_j$. With $C_{\mathrm{uni}}$ in \eqref{eq:Mj-definition} chosen
sufficiently large to absorb the constants in
\eqref{eq:uniqueness-block-H4} and Sobolev embedding, we obtain
\begin{equation}\label{eq:uniqueness-gradient-bound}
	1+\sup_{0\leq t\leq T}
	\left(
	\Norm{\nabla\theta_j(t)}_\infty
	+\Norm{\nabla\theta_j(t)}_2
	\right)
	\leq M_j.
\end{equation}

By the separation condition \eqref{eq:uniqueness-separation} and
$E_{k,0}\subset B(z_k,2^{-5})$, we obtain
\[
D_j\geq
2^{j+12}M_j e^{C_{\mathrm{uni}}M_jT}.
\]
Consequently,
\begin{equation}\label{eq:uniqueness-distance}
	D_j\longrightarrow\infty
	\quad\text{as }j\to\infty,
	\quad
	M_j e^{C_{\mathrm{uni}}M_jT}D_j^{-1}
	\leq2^{-j-12}.
\end{equation}

Let $\widetilde\theta$ be another classical solution on $[0,T]$
with initial datum $\theta_0$. By \cref{def:classical},
\[
\widetilde\theta
\in L^\infty([0,T];H^{1+\widetilde\gamma}(\R^2))
\]
for some $\widetilde\gamma>0$, possibly different from $\gamma$.
Using \eqref{Riesz-infty-H1+gamma} with exponent
$\widetilde\gamma$ and \eqref{eq:limit-Hgamma-bound}, we have
\[
U
:=
\max\left\{
\Norm{u[\theta]}_{L^\infty([0,T];L^\infty)},
\Norm{u[\widetilde\theta]}_{L^\infty([0,T];L^\infty)}
\right\}
<\infty.
\]
Moreover,
\[
u[\widetilde\theta]
\in L^\infty([0,T];W^{1,2}(\R^2)\cap L^\infty(\R^2)),
\quad
\nabla\cdot u[\widetilde\theta]=0.
\]
The pointwise equation and the preceding bounds give
\[
\partial_t\widetilde\theta
=-u[\widetilde\theta]\cdot\nabla\widetilde\theta
\in L^\infty([0,T];L^2).
\]
By Fubini's theorem,
$\partial_t\widetilde\theta(x,\cdot)\in L^1(0,T)$
for almost every $x$. Since
$t\mapsto\widetilde\theta(x,t)$ is differentiable throughout
$(0,T)$, the fundamental theorem of calculus yields
\[
\widetilde\theta(t)-\theta_0
=-\int_0^t
u[\widetilde\theta(\zeta)]\cdot
\nabla\widetilde\theta(\zeta)\,\dd\zeta
\quad\text{in }L^2,\qquad t\in[0,T].
\]
Hence $\widetilde\theta\in W^{1,\infty}([0,T];L^2)$
and satisfies \eqref{eq:SQG} in the distributional sense,
with initial trace $\theta_0$ in $L^2$.

We retain the flow $\Phi$ and the components $\theta_j$ from
Step~1. By the Sobolev transport theory in
\cite{DiPernaLions,Crippa}, the velocity $u[\widetilde\theta]$
generates a unique regular Lagrangian flow $\widetilde\Phi$
preserving Lebesgue measure. We define $\widetilde\theta_j$ by
\[
\widetilde\theta_j(\widetilde\Phi(t,x),t)
=
\vartheta_j(x-z_j,0)
\quad\text{for almost every }x.
\]
Since both flows preserve Lebesgue measure, we have
\[
\Norm{\theta_j(t)}_2
=
\Norm{\widetilde\theta_j(t)}_2
=
\Norm{\vartheta_j(0)}_2.
\]
These norms are summable in $j$ by \eqref{eq:block-list1}.
Hence both series converge in $L^\infty([0,T];L^2)$,
and uniqueness for the corresponding transport equations yields
\[
\theta=\sum_{j\geq1}\theta_j,
\quad
\widetilde\theta=\sum_{j\geq1}\widetilde\theta_j
\quad\text{in }L^\infty([0,T];L^2(\R^2)).
\]

Set $W_j:=\theta_j-\widetilde\theta_j$,
$W:=\theta-\widetilde\theta$, and define
\[
\mathcal E(t):=\sum_{j\geq1}\Norm{W_j(t)}_2.
\]
Since
\[
\sum_{j\geq1}
\Norm{W_j}_{L^\infty([0,T];L^2)}
\leq
2\sum_{j\geq1}\Norm{\vartheta_j(0)}_2
<\infty,
\]
we have $\mathcal E\in L^\infty([0,T])$ and
\[
W=\sum_{j\geq1}W_j
\quad\text{in }L^\infty([0,T];L^2),
\quad
\Norm{W(t)}_2\leq\mathcal E(t)
\]
for almost every $t\in[0,T]$.

Subtracting the equations for the $j$th components of the SQG
solutions $\theta$ and $\widetilde\theta$, we obtain
\begin{equation}\label{eq:component-difference}
	\partial_tW_j
	+u[\widetilde\theta]\cdot\nabla W_j
	+u[W]\cdot\nabla\theta_j
	=0,
	\quad W_j(0)=0.
\end{equation}
Choose $j_0$ such that $D_j\geq4TU$ for all $j\geq j_0$.
Both flows move points by at most $TU$ on $[0,T]$. Hence,
for almost every $t\in[0,T]$,
\begin{equation}\label{eq:competitor-separation}
		\operatorname{dist}\biggl(
		\supp\theta_j(t),
		\bigcup_{i\neq j}
		\left(
		\supp\theta_i(t)\cup\supp\widetilde\theta_i(t)
		\right)
		\biggr)\geq D_j-2TU
		\geq\frac{D_j}{2},
		\quad j\geq j_0.
\end{equation}

For $j\geq j_0$, write
\[
F_j:=\sum_{i\neq j}W_i,
\quad
u[W]=u[W_j]+u[F_j].
\]
The $L^2$ boundedness of the Riesz transforms and
\eqref{eq:uniqueness-gradient-bound} give
\[
\Norm{u[W_j]\cdot\nabla\theta_j}_2
\leq
C_{\mathrm{uni}}M_j\Norm{W_j}_2.
\]
Let $E_j(t):=\supp\theta_j(t)$. By
\eqref{eq:competitor-separation},
\[
\operatorname{dist}\bigl(E_j(t),\supp F_j(t)\bigr)
\geq\frac{D_j}{2}.
\]
The kernel bound $|\cK(z)|\leq C|z|^{-2}$ and
Cauchy--Schwarz therefore imply
\[
	\Norm{u[F_j(t)]}_{L^\infty(E_j(t))}
	\leq
	C_{\mathrm{uni}}
	\left(
	\int_{|z|\geq D_j/2}|z|^{-4}\,\dd z
	\right)^{1/2}
	\Norm{F_j(t)}_2\leq
	C_{\mathrm{uni}}D_j^{-1}\mathcal E(t).
\]
Using \eqref{eq:uniqueness-gradient-bound} once more, we obtain
\[
	\Norm{u[F_j]\cdot\nabla\theta_j}_2
	\leq
	\Norm{u[F_j]}_{L^\infty(E_j(t))}
	\Norm{\nabla\theta_j}_2\leq
	C_{\mathrm{uni}}M_jD_j^{-1}\mathcal E(t).
\]
Since $u[\widetilde\theta]$ is divergence-free and $W_j(0)=0$,
we apply the $L^2$ energy estimate to
\eqref{eq:component-difference} and use the preceding bounds
to obtain, for $j\geq j_0$,
\begin{equation*}
	\Norm{W_j(t)}_2
	\leq
	C_{\mathrm{uni}}M_j
	\int_0^t\Norm{W_j(\zeta)}_2\,\dd\zeta
	+
	C_{\mathrm{uni}}M_jD_j^{-1}
	\int_0^t\mathcal E(\zeta)\,\dd\zeta.
\end{equation*}
Gronwall's inequality then gives
\[
\Norm{W_j(t)}_2
\leq
C_{\mathrm{uni}}M_jD_j^{-1}e^{C_{\mathrm{uni}}M_jT}
\int_0^t\mathcal E(\zeta)\,\dd\zeta
\leq
C_{\mathrm{uni}}2^{-j-12}
\int_0^t\mathcal E(\zeta)\,\dd\zeta,
\]
where the last inequality follows from
\eqref{eq:uniqueness-distance}.

For $1\leq j<j_0$, we instead use the global $L^2$ boundedness
of the Riesz transforms and $\Norm{W(t)}_2\leq\mathcal E(t)$.
The same energy estimate yields
\begin{equation*}
	\Norm{W_j(t)}_2
	\leq
	C_{\mathrm{uni}}M_j
	\int_0^t\mathcal E(\zeta)\,\dd\zeta.
\end{equation*}
The component estimates hold for almost every $t\in[0,T]$.
Their summability lets us add them over $j$ to obtain
\[
\mathcal E(t)
=
\sum_{j\geq1}\Norm{W_j(t)}_2
\leq
A_*\int_0^t\mathcal E(\zeta)\,\dd\zeta.
\]
where
\[
A_*:=
C_{\mathrm{uni}}
\left(
\sum_{j=1}^{j_0-1}M_j
+\sum_{j=j_0}^{\infty}2^{-j-12}
\right)
<\infty.
\]
Choose $h_0>0$ such that $A_*h_0\leq1/2$.
If $\mathcal E=0$ almost everywhere on $(0,t_0)$ for some
$t_0\in[0,T)$, the integral inequality above reduces to
\[
\mathcal E(t)
\leq
A_*\int_{t_0}^t\mathcal E(\zeta)\,\dd\zeta
\quad\text{for almost every }t\in[t_0,T].
\]
Set $h:=\min\{h_0,T-t_0\}$. Taking the essential supremum
over $(t_0,t_0+h)$, we obtain
\[
\Norm{\mathcal E}_{L^\infty(t_0,t_0+h)}
\leq
A_*h\Norm{\mathcal E}_{L^\infty(t_0,t_0+h)}
\leq
\frac12\Norm{\mathcal E}_{L^\infty(t_0,t_0+h)}.
\]
Since $\mathcal E\in L^\infty(0,T)$, this implies that
$\mathcal E=0$ almost everywhere on $(t_0,t_0+h)$.
Starting with $t_0=0$ and repeating this argument with the
same $h_0$, we conclude after finitely many steps that
$\mathcal E=0$ almost everywhere on $[0,T]$.

Hence $\theta(t)=\widetilde\theta(t)$ in $L^2(\R^2)$
for almost every $t\in[0,T]$.
By the continuity of both solutions in $L^2$, we conclude that
\[
\theta(t)=\widetilde\theta(t)
\quad\text{in }L^2(\R^2),~t\in[0,T].
\]
This proves uniqueness and hence completes the proof of
\cref{thm:main-loss}.
\end{proof}

\appendix
\section{Radial shell formulas for Euler and SQG}
\label{app:radial-kernel}

The shell formulas below justify the radial shear used in
\cref{lem:nondegenerate}. They also show why the positions of the
radial and oscillatory annuli differ for Euler and SQG.

For Euler, \cref{lem:appendix-euler-shell} shows that an outer radial
profile produces no velocity in the inner region. Thus a radial
background must lie inside the oscillatory annulus to generate shear.
For SQG, \cref{lem:appendix-SQG-shear} shows that a radial profile can
generate nondegenerate angular shear on a separated annulus on either
side of its support. We place the radial annulus outside the oscillatory one.

\begin{lemma}[Euler shell formula]
	\label{lem:appendix-euler-shell}
	For radial $\omega\in C_c^\infty(\R^2)$, the two-dimensional
	Euler Biot--Savart law gives
	$u_E[\omega]=u_{E,\varphi}[\omega](r)\evarphi$, where
	\begin{equation}\label{eq:appendix-euler-shell}
		u_{E,\varphi}[\omega](r)
		=
		\frac{c_E}{r}\int_0^r \rho\omega(\rho)\,\dd\rho,
		\quad r>0,
	\end{equation}
	where $c_E\neq0$ is independent of $r$ and $\omega$.
\end{lemma}

\begin{proof}
	In complex notation, write $x=re^{i\theta}$ and $y=\rho e^{i\phi}$. For $0\leq a<1$,
	\[
	\log|1-ae^{i\alpha}|
	=
	-\sum_{n=1}^\infty \frac{a^n}{n}\cos(n\alpha).
	\]
Integration in $\alpha$ gives
	\[
	\frac{1}{2\pi}\int_0^{2\pi}
	\log|1-ae^{i\alpha}|\,\dd\alpha
	=0.
	\]
After extracting the factor $\max\{r,\rho\}$, we obtain
	\begin{equation}\label{eq:appendix-euler-average}
		\frac{1}{2\pi}\int_0^{2\pi}
		\log|re^{i\theta}-\rho e^{i\phi}|\,\dd\phi
		=
		\log\max\{r,\rho\},
		\quad
		r\neq\rho.
	\end{equation}
	For a fixed constant $c_E\neq0$, the stream function associated with
	$\omega$ has the logarithmic representation
		\[
	\psi_\omega(r)
	=
	\frac{c_E}{2\pi}
	\int_0^\infty\int_0^{2\pi}
	\log
	\left|
	x-\rho(\cos\phi,\sin\phi)
	\right|
	\omega(\rho)\rho\,\dd\phi\dd\rho.
	\]
	Using \eqref{eq:appendix-euler-average} and splitting the radial
	integral at $\rho=r$, we find
	\[
	\psi_\omega(r)
	=
	c_E\left[
	\log r\int_0^r\rho\omega(\rho)\,\dd\rho
	+
	\int_r^\infty
	\log\rho\,\rho\omega(\rho)\,\dd\rho
	\right].
	\]
The boundary terms cancel; hence
	\[
	\psi_\omega'(r)
	=
	\frac{c_E}{r}
	\int_0^r\rho\omega(\rho)\,\dd\rho.
	\]
	Since $u_{E,\varphi}[\omega]=\psi_\omega'$, this proves
	\eqref{eq:appendix-euler-shell}.
\end{proof}

If a radial vorticity $\omega$ is supported outside $B_R$, then
\eqref{eq:appendix-euler-shell} gives
\[
u_E[\omega](x)=0,
\qquad
|x|<R.
\]
Thus an outer radial profile induces no velocity, and hence no shear,
in the inner region. On the other hand, if the profile is supported
inside the observation radius and
\[
\Gamma:=\int_0^\infty\rho\omega(\rho)\,\dd\rho\neq0,
\]
then
\[
u_{E,\varphi}[\omega](r)=\frac{c_E\Gamma}{r},
\quad
\frac{\dd}{\dd r}
\left(\frac{u_{E,\varphi}[\omega](r)}r\right)
=
-\frac{2c_E\Gamma}{r^3}\neq0.
\]

\begin{lemma}[SQG shell formula]
	\label{lem:appendix-SQG-shear}
	Let $f\in C_c^\infty(\R^2)$ be radial, nonnegative and not
	identically zero. Define
	\[
	\psi_f:=\Lambda^{-1}f,
	\quad
	\Omega_f(r):=\frac{\psi_f'(r)}{r},
	\quad r>0.
	\]
	For $k\geq0$, set
	\[
	a_k
	:=
	\left(
	\frac{\binom{2k}{k}}{4^k}
	\right)^2>0.
	\]
	Then, with a fixed dimensional constant $c_S>0$, the following
	statements hold.
	\begin{enumerate}[label=\textup{(\roman*)}]
		\item
		Let $0<a<b$ and suppose that $\supp f\subset(a,b)$. Set
		\[
		M_k:=\int_a^b f(\rho)\rho^{-2k}\,\dd\rho>0,
		\quad k\geq0.
		\]
		Then, for $0<r<a$,
		\begin{equation}\label{eq:appendix-outer-shear}
			\Omega_f'(r)
			=
			c_S\sum_{k=2}^\infty
			2k(2k-2)a_kM_kr^{2k-3}
			\geq
			8c_Sa_2M_2r>0,
		\end{equation}
		and
		\begin{equation}\label{eq:appendix-outer-curvature}
			\Omega_f''(r)
			=
			c_S\sum_{k=2}^\infty
			2k(2k-2)(2k-3)a_kM_kr^{2k-4}
			\geq
			8c_Sa_2M_2>0.
		\end{equation}
		\item
		Let $0<a<b$ and suppose that $\supp f\subset(a,b)$. Set
		\[
		\widetilde M_k
		:=
		\int_a^b f(\rho)\rho^{2k+1}\,\dd\rho>0,
		\quad k\geq0.
		\]
		Then, for $r>b$,
		\begin{equation}\label{eq:appendix-inner-shear}
			\Omega_f'(r)
			=
			c_S\sum_{k=0}^\infty
			(2k+1)(2k+3)a_k\widetilde M_kr^{-2k-4}
			>0,
		\end{equation}
		while
		\begin{equation}\label{eq:appendix-inner-shear-1}
			\Omega_f''(r)
			=
			-c_S\sum_{k=0}^\infty
			(2k+1)(2k+3)(2k+4)
			a_k\widetilde M_kr^{-2k-5}
			<0.
		\end{equation}
		
	\item
	Let $0<a<b$ and suppose that $\supp f\subset(a,b)$.
	Then the function $x\mapsto\Omega_f(|x|)$ extends smoothly to
	$\R^2$. Moreover,
		for every multi-index $\nu$ with $|\nu|\leq3$,
		\begin{equation}\label{eq:appendix-Omega-global}
			\left|
			\partial_x^\nu\bigl(\Omega_f(|x|)\bigr)
			\right|
			\leq
			C_{f,\nu}(1+|x|)^{-3-|\nu|},
			\quad x\in\R^2.
		\end{equation}
	\end{enumerate}
\end{lemma}

\begin{proof}
	Fix $0<a<b$ with $\supp f\subset(a,b)$.
	The kernel representation of $\Lambda^{-1}$ gives
	\begin{equation}\label{eq:appendix-radial-potential}
		\psi_f(r)
		=
		c_\Lambda
		\int_0^\infty\int_0^{2\pi}
		\frac{f(\rho)\rho}
		{\left(r^2+\rho^2-2r\rho\cos\beta\right)^{1/2}}
		\,\dd\beta\dd\rho,
		\quad c_\Lambda>0.
	\end{equation}
	Set $r_>:=\max\{r,\rho\}$, $r_<:=\min\{r,\rho\}$ and
	$\eta:=r_</r_>$. For $r\neq\rho$, we have $0\leq\eta<1$ and
	\[
	\frac{1}
	{\left(r^2+\rho^2-2r\rho\cos\beta\right)^{1/2}}
	=
	\frac{1}{r_>}
	(1-\eta e^{i\beta})^{-1/2}
	(1-\eta e^{-i\beta})^{-1/2}.
	\]
	Using
	\[
	(1-z)^{-1/2}
	=
	\sum_{n=0}^\infty b_nz^n,
	\quad
	b_n:=\frac{\binom{2n}{n}}{4^n},
	\]
	we obtain
	\[
	(1-\eta e^{i\beta})^{-1/2}
	(1-\eta e^{-i\beta})^{-1/2}
	=
	\sum_{m,n\geq0}
	b_mb_n\eta^{m+n}e^{i(m-n)\beta}.
	\]
	Integration in $\beta$ eliminates the terms with $m\neq n$.
	Since $a_k=b_k^2$, this yields
	\begin{equation}\label{eq:appendix-SQG-average}
		\frac{1}{2\pi}\int_0^{2\pi}
		\frac{\dd\beta}
		{\left(r^2+\rho^2-2r\rho\cos\beta\right)^{1/2}}
		=
		\frac{1}{r_>}
		\sum_{k=0}^\infty
		a_k\left(\frac{r_<}{r_>}\right)^{2k}.
	\end{equation}
	Because $\supp f$ is compactly contained in $(a,b)$,
	$r_</r_>$ is uniformly bounded by some $\eta_0<1$ for
	$\rho\in\supp f$ in both regions $0<r<a$ and $r>b$.
	Consequently, the series and its derivatives of any fixed order
	in $r$ converge absolutely and uniformly.

	For $0<r<a$, we substitute
	\eqref{eq:appendix-SQG-average} into
	\eqref{eq:appendix-radial-potential} to obtain
	\[
	\psi_f(r)
	=
	c_S\sum_{k=0}^\infty a_kM_kr^{2k},
	\quad c_S:=2\pi c_\Lambda>0,
	\]
	and hence
	\[
	\Omega_f(r)
	=
	c_S\sum_{k=1}^\infty
	2k\,a_kM_kr^{2k-2}.
	\]
	Differentiating once and twice gives
	\eqref{eq:appendix-outer-shear} and
	\eqref{eq:appendix-outer-curvature}, respectively. Their lower
	bounds follow by retaining the $k=2$ term, since all terms are
	nonnegative.
	
	For $r>b$, we have $\rho<r$ on $\supp f$ and obtain
	\[
	\psi_f(r)
	=
	c_S\sum_{k=0}^\infty
	a_k\widetilde M_kr^{-2k-1},
	\]
	and therefore
	\[
	\Omega_f(r)
	=
	-c_S\sum_{k=0}^\infty
	(2k+1)a_k\widetilde M_kr^{-2k-3}.
	\]
	Differentiating once and twice proves
	\eqref{eq:appendix-inner-shear} and
	\eqref{eq:appendix-inner-shear-1}.
	
We prove \textup{(iii)} by considering three regions.

\smallskip
\noindent\emph{Near the origin: $|x|\leq a/2$.}
The interior expansion gives
\[
\Omega_f(|x|)
=
c_S\sum_{k=1}^\infty
2k\,a_kM_k(x_1^2+x_2^2)^{k-1},
\quad
M_k=\int_a^b f(\rho)\rho^{-2k}\,\dd\rho.
\]
Since $M_k\leq a^{-2k}\int_a^b|f(\rho)|\,\dd\rho$, this series
and its derivatives of every fixed order converge uniformly on
$|x|\leq a/2$. It therefore defines a smooth extension through
the origin, with
\[
\sup_{|x|\leq a/2}
\left|
\partial_x^\nu\bigl(\Omega_f(|x|)\bigr)
\right|
\leq C_{f,\nu},
\quad |\nu|\leq3.
\]

\smallskip
\noindent\emph{The intermediate annulus: $a/2\leq|x|\leq2b$.}
Since $f\in C_c^\infty(\R^2)$, we have
$\psi_f=\Lambda^{-1}f\in C^\infty(\R^2)$.
The identity $\Omega_f(r)=\psi_f'(r)/r$ thus gives
\[
\sup_{a/2\leq|x|\leq2b}
\left|
\partial_x^\nu\bigl(\Omega_f(|x|)\bigr)
\right|
\leq C_{f,\nu},
\quad |\nu|\leq3.
\]

\smallskip
\noindent\emph{The exterior region: $|x|\geq2b$.}
The exterior expansion gives
\[
\Omega_f(|x|)
=
-c_S\sum_{k=0}^\infty
(2k+1)a_k\widetilde M_k|x|^{-2k-3},
\quad
\widetilde M_k
=
\int_a^b f(\rho)\rho^{2k+1}\,\dd\rho.
\]
For every multi-index $\nu$ with $|\nu|\leq3$,
\[
\left|\partial_x^\nu|x|^{-2k-3}\right|
\leq
C_\nu(1+k)^{|\nu|}|x|^{-2k-3-|\nu|}.
\]
Using $a_k\leq1$ and $\widetilde M_k\leq C_f b^{2k}$,
we may differentiate the series term by term, since
$b/|x|\leq1/2$. This yields
\[
\begin{aligned}
	\left|
	\partial_x^\nu\bigl(\Omega_f(|x|)\bigr)
	\right|
	&\leq
	C_{f,\nu}|x|^{-3-|\nu|}
	\sum_{k=0}^\infty
	(1+k)^{1+|\nu|}
	\left(\frac{b}{|x|}\right)^{2k}\\
	&\leq
	C_{f,\nu}|x|^{-3-|\nu|}.
\end{aligned}
\]

On $|x|\leq2b$, the weight $(1+|x|)^{-3-|\nu|}$ is bounded
below by a positive constant. On $|x|\geq2b$, it is comparable
to $|x|^{-3-|\nu|}$. The three estimates therefore prove
\eqref{eq:appendix-Omega-global}.	
\end{proof}

    \section*{Acknowledgement}
    The authors thank to Prof. Diego Córdoba for his valuable suggestions.
This project was supported by the National Key Research and Development Program of China (Grant No. 2022YFA1005700). Q. Chen was partially supported by the National Natural Science Foundation of
China (Grant No. 12471149). 
C. Miao was partially supported by the National Natural Science Foundation of China (Grant Nos. 12371095 and 12531005).

\end{document}